%% file: arxiv.tex
\pdfoutput=1
\documentclass[10pt]{amsart}
\usepackage{amsthm,amsmath,amsfonts}
\usepackage{amssymb}
\usepackage[ruled,linesnumbered]{algorithm2e}
\numberwithin{algocf}{section}
\usepackage[toc,page]{appendix}
\usepackage{array}
\usepackage{bm}
\usepackage{bbm}
\usepackage{calligra}
\usepackage{color}
\usepackage{dsfont}
\usepackage{comment}
\usepackage[dvipsnames]{xcolor}
\usepackage{enumerate}
\usepackage{enumitem}
\usepackage{fullpage}
\usepackage{graphicx}
\usepackage{mathMacros}
\usepackage{mathtools}
\usepackage{multirow}
\usepackage{pifont}
\usepackage{caption}
\usepackage{subcaption}
\usepackage{tablefootnote}

\usepackage[backend=biber,style=numeric,hyperref=true,backref=true]{biblatex}
\usepackage{hyperref}
\hypersetup{
    colorlinks=true,
    linkcolor={red!90!black},
    citecolor={red!90!black},
    filecolor=magenta,      
    urlcolor={red!90!black},
    }
\colorlet{linkequation}{blue}
\usepackage{cleveref}
\usepackage{orcidlink}

\numberwithin{equation}{section}
\theoremstyle{plain}
\newtheorem{theorem}{Theorem}[section]

\newtheorem{corollary}[theorem]{Corollary}
\newtheorem{lemma}[theorem]{Lemma}
\newtheorem{proposition}[theorem]{Proposition}
\newtheorem{definition}[theorem]{Definition}
\newtheorem{assumption}[theorem]{Assumption}
\newtheorem{remark}[theorem]{Remark}

\theoremstyle{remark}

\DeclareMathAlphabet{\mathcalligra}{T1}{calligra}{m}{n}
\DeclareMathAlphabet{\mathpzc}{OT1}{pzc}{m}{it}

\renewcommand{\epsilon}{\varepsilon}
\newcommand{\mismatch}{\eta}

\def\LM{{\textnormal{L}}} 

\newcommand*\circled[1]{\tikz[baseline=(char.base)]{
            \node[shape=circle,draw,inner sep=1pt] (char) {#1};}}

\usepackage[textsize=tiny]{todonotes}

\title[Binary Non-uniform HSBM]{Information-Theoretic Limits and Strong Consistency on \\ Binary Non-uniform Hypergraph Stochastic Block Models}
\begin{document}

\author{Hai-Xiao Wang \orcidlink{0000-0003-2730-1439}}
\address{Department of Mathematics, University of Wisconsin-Madison, Madison, WI 53706, USA}
\email{hwang2594@wisc.edu}


\date{\today}

\begin{abstract}
We investigate the unsupervised node classification problem on random hypergraphs under the non-uniform \emph{Hypergraph Stochastic Block Model} (HSBM) with two equal-sized communities. In this model, edges appear independently with probabilities depending only on the labels of their vertices. We identify the threshold for strong consistency, expressed in terms of the \emph{Generalized Hellinger} distance. Below this threshold, strong consistency is impossible, and we derive the \emph{Information-Theoretic} (IT) lower bound on the expected \emph{mismatch ratio}. Above the threshold, the parameter space is typically divided into two disjoint regions. When only the aggregated adjacency matrices are accessible, while \emph{one-stage} algorithms accomplish strong consistency with high probability in the region far from the threshold, they fail in the region closer to the threshold. We propose a new refinement algorithm which, in conjunction with the initial estimation, provably achieves strong consistency throughout the entire region above the threshold, and attains the IT lower bound when below the threshold, proving its optimality. This novel refinement algorithm applies the power iteration method to a weighted adjacency matrix, where the weights are determined by hyperedge sizes and the initial label estimate. Unlike the constant degree regime where a subset selection of uniform layers is necessary to enhance clustering accuracy, in the scenario with diverging degrees, each uniform layer contributes non-negatively to clustering accuracy. Therefore, aggregating information across all uniform layers yields better performance than using any single layer alone.
\end{abstract}
\maketitle

{\tableofcontents}
\input{maintext}
\section*{Acknowledgments}
The author would like to thank Ioana Dumitriu for her stimulating discussions and insightful comments on the early draft of this manuscript. The author would like to thank Ery Arias-Castro and Yizhe Zhu for helpful discussions. Part of this work was done during the \emph{Southeastern Probability Conference} May 15-16 2023 at Duke University in Durham, NC, and the summer school on \emph{Random Matrix Theory and Its Applications} May 22-26 2023 at the Ohio State University in Columbus, OH. The author thanks the organizers for their hospitality. The author acknowledges support from NSF DMS-2154099.

\printbibliography
\newpage
\appendix
\addcontentsline{toc}{section}{Appendices}
\input{appendix}
\end{document}

%% file: maintext.tex
\section{Introduction}\label{sec:intro_binary}

Community detection refers to the task of partitioning the vertices of a graph into groups with similar connectivity patterns. It has attracted significant attention due to its applications in data science, social networks, and biological studies. A classic example is the Zachary's karate club problem \cite{Zachary1977InformationFM}. Community detection is now a central topic in network analysis and machine learning \cite{Newman2002RandomGM, Ng2002SpectralCA, Arias2014CommunityDI}. Random graph models are widely used for analyzing clustering algorithms, benchmarking, and establishing theoretical guarantees. In this work, we focus on the equi-sized two-community model. Formally, let $\gG = (\gV, \gE)$ be a graph with $N$ vertices. The label vector $\rvy \in \{\pm 1\}^{|\gV|}$ assigns each vertex to a community, where $\gV_{+} = \{v\in \gV| \ervy_v = 1\}$ and $\gV_{-} = \{v\in \gV| \ervy_v = -1\}$, with $|\gV_{+}| = |\gV_{-}| = N/2$. Given $\widehat{\rvy}, \rvy \in \{\pm 1\}^{|\gV|}$, define the \emph{Hamming distance} as
\begin{align}
    \D_{\textnormal{HD}}( \rvy, \widehat{\rvy}) = \sum_{v\in \gV} \indi{\ervy_v \neq \widehat{\ervy}_v}, \label{eqn:D_HD_binary}
\end{align}
which counts the number of entries having different values. We further define the \emph{mismatch} ratio as
\begin{align}
   \eta_{N} \coloneqq \eta_{N}(\rvy, \widehat{\rvy}) = \frac{1}{N} \,\,\min_{s \in \{\pm 1\}} \D_{\textnormal{HD}}( s\rvy, \widehat{\rvy}),\label{eqn:misratio_binary}
\end{align}
which is often used for quantifying the accuracy of the estimator $\widehat{\rvy}$. Note that a random guess estimator has expected accuracy $(1/2)^2 + (1/2)^2 = 1/2$, thus the estimator $\widehat{\rvy}$ is meaningful only if $\eta_{N} \leq 1/2$. A natural question arises here: when is it possible to correctly classify all vertices? An estimator $\widehat{\rvy}$ is said to achieve \emph{strong consistency} (exact recovery) if
\begin{align}
    \lim_{N \to \infty}\P(\eta_{N} = 0) = \lim_{N \to \infty}\P( \widehat{\rvy} = \pm \rvy) = 1. \label{eqn:strong_consistency_binary}
\end{align}
At the same time, $\widehat{\rvy}$ is said to achieve \emph{weak consistency} (almost exact recovery) if $\eta_{N} = o(1)$.

In the pioneering work by \cite{Holland1983StochasticBM}, \emph{Stochastic Block Model} (SBM) was first introduced for sociology studies. Motivated by exploring the phase transition behaviors in various regimes, many insightful problems were brought out. It has been successively studied in \cite{Bui1984GraphBA, Boppana1987EigenvaluesAG, Dyer1989TheSO, Snijders1997EstimationAP, Condon1999AlgorithmsFG, McSherry2001SpectralPO, Bickel2009ANV, Coja-Oghlan2010GraphPV, Rohe2011SpectralCA, Choi2012StochasticBW, Agterberg2022JointSC, Chen2022GlobalAI, Ma2023CommunityDI} over the past several decades. A major breakthrough was the establishment of the exact recovery thresholds for the binary case by \cite{Abbe2016ExactRI, Mossel2016ConsistencyTF} and multi-block case by \cite{Abbe2015CommunityDI, Yun2016OptimalCR, Agarwal2017MultisectionIT}. Readers may refer to \cite{Abbe2018CommunityDA} for a more detailed review. The literature abounds with different methods of approach, like spectral algorithms \cite{Ghoshdastidar2014ConsistencyOS, Ahn2016CommunityRI, Ghoshdastidar2017ConsistencyOS, Chien2018CommunityDI, Cole2020ExactRI, Zhang2023ExactRI}, \emph{semidefinite programming} (SDP) \cite{Kim2018StochasticBM, Lee2020RobustHC, Gaudio2023CommunityDI} and \emph{Graph Neural Networks} (GNN) \cite{Chen2018SupervisedCD, Baranwal2022EffectsGC, Baranwal2021GraphCS, Wang2024OptimalER}

The \emph{Hypergraph Stochastic Block Model} (HSBM), first introduced in \cite{Ghoshdastidar2014ConsistencyOS}, is a generalization of SBM to hypergraphs, which models the real life social networks better since it captures higher order interactions among three or more individuals. It has been studied successively in \cite{Angelini2015SpectralDO, Ahn2016CommunityRI,
 Ghoshdastidar2017ConsistencyOS, Kim2018StochasticBM, Chien2018CommunityDI, Chien2019MinimaxMR, Cole2020ExactRI, Pal2021ComunityDI, Dumitriu2025PartialRA, Zhang2023ExactRI, Gu2023WeakRT, Dumitriu2023OptimalAE, Wang2023ProjectedTP} over the last decade. For strong consistency of \textit{uniform} HSBMs, the exact thresholds were given in \cite{Kim2018StochasticBM, Gaudio2023CommunityDI, Zhang2023ExactRI} by generalizing techniques in \cite{Abbe2016ExactRI, Abbe2020EntrywiseEA, Abbe2015CommunityDI}. So far, most results concern community detection on uniform hypergraphs, which require the same number of vertices in each edge. However, the assumption on uniformity between edges is constraining and somewhat impractical, and the non-uniform HSBM was less explored in the literature, with notable results in \cite{Ghoshdastidar2017ConsistencyOS, Dumitriu2025PartialRA, Alaluusua2023MultilayerHC, Dumitriu2023OptimalAE}. According to \cite{Dumitriu2023OptimalAE}, the exact threshold for strong consistency of non-uniform HSBM was established, where some two-stage algorithms were proposed to achieve strong consistency when above the threshold. However, these algorithms require the full set knowledge of adjacency tensors as input, which may not be accessible in practice.

In this paper, we focus on the unsupervised node classification problem under the binary non-uniform HSBM with two equi-sized communities, where only the aggregated adjacency matrices are available. We begin by introducing the basic definitions of the non-uniform HSBM, then the problems of our interests. We will summarize our contributions and outline the organization of this paper at the end of this section.
\subsection{The model}
\begin{definition}[Hypergraph]
    A hypergraph $\gH$ is a pair $\gH = (\gV, \gE)$, where $\gV$ is the vertex set and $\gE$ denotes the set of non-empty subsets of $\gV$. If every hyperedge $e$ is an $\ell$-subset of $\gV$, then $\gH$ is called $\ell$-uniform. The degree $\rD_{v}$ of a vertex $v \in \gV$ is the number of hyperedges in $\gE$ containing $v$.
\end{definition}
\begin{definition}[Symmetric binary uniform HSBM]\label{def:uniform_HSBM_binary}
    Let $\rvy \in \{ \pm 1\}^{|\gV|}$ denote the label vector on $\gV$, where $\rvy$ is chosen uniformly at random among all the vectors satisfying $\ones^{\sT}\rvy = 0$. The $\ell$-uniform hypergraph $\gH_{\ell} = (\gV, \gE_{\ell})$ ($\ell\geq 2$ some fixed integer)  is drawn in the following manner: each $\ell$-hyperedge $e:= \{i_1, \ldots, i_{\ell}\}\subset \gV$ is sampled independently with probability $\alpha_{\ell}$ if $\ervy_{i_{1}} = \ldots = \ervy_{i_{\ell}}$, otherwise with probability $\beta_{\ell}$.
\end{definition}
\begin{definition}[Non-uniform HSBM]\label{def:non_uniform_HSBM_binary}
    Let $\rvy \in \{ \pm 1\}^{|\gV|}$ denote the label vector, sampled uniformly at random among all the vectors satisfying $\ones^{\sT}\rvy = 0$. Let $\sL = \{\ell | \ell\geq 2\}$ be a set of finitely many integers with $\LM$ denoting its largest element. The non-uniform hypergraph $\gH =(\gV, \gE)$ is a collection of independent uniform ones, i.e., $\gH = \cup_{\ell \in \sL}\gH_{\ell}$ with $\gE := \cup_{\ell \in \sL}\gE_{\ell}$, where each $\gH_{\ell} = (\gV, \gE_{\ell})$ is sampled from the $\ell$-uniform \emph{HSBM} \ref{def:uniform_HSBM_binary} independently given the same $\rvy$.
\end{definition}
\begin{remark}\label{rem:model_difference}
Unlike the model in \cite{Dumitriu2023OptimalAE}, where each vertex is assigned to some community independently and the community sizes might fluctuate, our model \ref{def:non_uniform_HSBM_binary} imposes a strict half-half constraint $\ones^{\sT}\rvy = 0$, ensuring exactly equal-sized communities. The proof strategies for the main results differ from those in \cite{Dumitriu2023OptimalAE}, which are tailored to our model. 
\end{remark}
One can associate an $\ell$-uniform hypergraph $\gH_{\ell}$ to an order-$\ell$ symmetric tensor $\tA^{(\ell)}$, where the entry $\etA^{(\ell)}_{e}$ denotes the presence of the $\ell$-hyperedge $e = \{i_1, \ldots, i_{\ell}\}$, i.e., $\etA_{i_1,\dots, i_{\ell}}^{(\ell)} = \indi{e \in \gE_{\ell}}$. However, it is impossible to aggregate information of each uniform hypergraph directly since tensors are of different orders. Meanwhile, most computations involving tensors are NP-hard \cite{Hillar2013MostTP}. To that end, we introduce the following \emph{aggregated adjacency matrix} for hypergraphs, obtained by aggregating information from different tensors entrywisely.
\begin{definition}[Aggregated adjacency matrix]
For the non-uniform hypergraph $\gH = (\gV, \gE)$, let $\tA^{(\ell)}$ be the order-$\ell$ adjacency tensor corresponding to underlying $\ell$-uniform hypergraph for each $\ell \in\sL$. Let
\begin{align}
    \ermA_{ij} = \sum_{\ell \in \sL}\ermA^{(\ell)}_{ij}, \quad \ermA^{(\ell)}_{ij} \coloneqq \indi{i \neq j} \sum_{e\in \gE(\gH_{\ell}),\,\, e \supset \{i, j\} } \,\,\,\, \etA^{(\ell)}_{e}\,,\label{eqn:adjacency_matrix_entry_binary}
\end{align}
denote the number of hyperedges containing both vertices $i$ and $j$ in the non-uniform hypergraph $\gH$. The diagonal entries are defined as $\ermA_{ii} \coloneqq 0$ for each $i\in \gV$ since each $\ell$-hyperedge contains $\ell$ distinct vertices. 
\end{definition}

For the discussion of strong consistency, it typically requires that the expected degree of each vertex grows to infinity with respect to the number of vertices. To simplify the presentation, we adapt the following assumption on the model parameters.
\begin{assumption}\label{ass:asymptotics_binary}
For each $\ell \in \sL$, assume that $\alpha_{\ell} \geq \beta_{\ell}$. We write $\alpha_{\ell} = a_{\ell} \cdot q_{N} /\binom{N-1}{\ell - 1}$ and $\beta_{\ell} = b_{\ell} \cdot q_{N} /\binom{N-1}{\ell - 1}$, where $q_{N}$ denotes the magnitude parameter with $1 \ll q_{N} \lesssim \log(N)$ and $a_{\ell} \geq b_{\ell} \gtrsim 1$.
\end{assumption}

\subsection{Formulation of the problems}
We first define the following \emph{Generalized Hellinger} divergence $\D_{\mathrm{GH}}$ 
\begin{align}
    \D_{\mathrm{GH}} \coloneqq &\, \sum_{\ell \in \sL}\frac{1}{2^{\ell - 1}} \big( \sqrt{a_{\ell}} - \sqrt{b_{\ell}} \big)^2,\label{eqn:DGH_binary}
\end{align}
which quantifies the fundamental \emph{Information-Theoretic} (IT) limits for the non-uniform HSBM model \ref{def:non_uniform_HSBM_binary}, when the full set of adjacency tensors $\{\tA^{(\ell)}\}_{\ell \in \sL}$ is accessible. By contrast, we introduce the following function of the model parameters:
\begin{align}
    \D_{\rm{AM}} \coloneqq &\, \sup_{t\geq 0} \sum_{\ell \in \sL}\frac{1}{2^{\ell - 1}} \Big[ a_{\ell}\big(1 - e^{-(\ell - 1)t} \big) + b_{\ell} \sum_{j=1}^{\ell - 1} \binom{\ell - 1}{j} \big(1 - e^{-(\ell - 1 - 2j)t} \big) \Big],\label{eqn:DAM_binary}
\end{align}
which characterizes the algorithmic limits when only the aggregated \emph{Adjacency Matrix} (AM) serves as input. Except for special cases $\sL = \{2\}$ and $\sL = \{3\}$ illustrated in \Cref{fig:thresholds_plots} (b), a gap emerges between $\D_{\mathrm{GH}}$ and $\D_{\rm{AM}}$, as demonstrated for $\sL = \{4\}$ in \Cref{fig:thresholds_plots} (a) and $\sL = \{2, 3\}$ in \Cref{fig:thresholds_plots} (c).

\begin{figure}
    \centering
    \begin{minipage}[h]{0.328\linewidth}
        \centering
        {\includegraphics[width=0.99\linewidth]{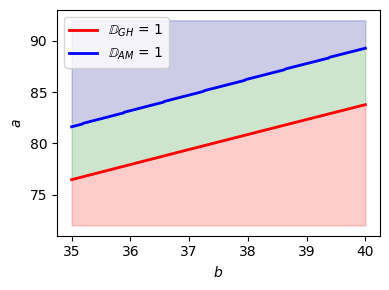} 
        (a) $\sL =\{4\}$
        }
    \end{minipage}
    \begin{minipage}[h]{0.328\linewidth}
        \centering
        {\includegraphics[width=0.99\linewidth]{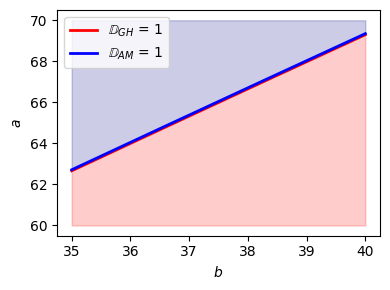} 
        (b) $\sL =\{3\}$
        }
    \end{minipage}
    \begin{minipage}[h]{0.328\linewidth}
    \centering
    {\includegraphics[width=0.99\linewidth]{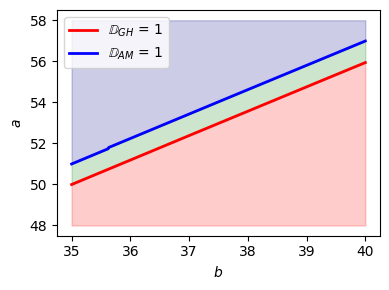}
    (c) $\sL =\{2, 3\}$
    }
    \end{minipage}
    \caption{\small{Plots for areas $\D_{\mathrm{GH}} < 1$ (\textcolor{black}{red}), $\D_{\rm{AM}} < 1 < \D_{\mathrm{GH}}$ (\textcolor{black}{green}) and $\D_{\rm{AM}} > 1$ (\textcolor{black}{blue})}} \label{fig:thresholds_plots}
\end{figure}

\subsubsection{Phase transition threshold for strong consistency}
For \textbf{uniform} hypergraphs, the problem is reduced to the graph case when $\sL = \{2\}$, where the two quantities are $\D_{\mathrm{GH}} = \D_{\rm{AM}} = (\sqrt{a_2}- \sqrt{b_2})^2/2$. It is impossible to achieve strong consistency when $\D_{\mathrm{GH}} < 1$ \cite{Abbe2016ExactRI}, while \emph{one-stage} algorithms, such as spectral clustering \cite{Abbe2020EntrywiseEA, Deng2021StrongCG} and \emph{semidefinite programming} (SDP) \cite{Hajek2016AchievingEC}, achieve strong consistency when $\D_{\mathrm{GH}} > 1$. However, when $\sL = \{\ell\}$ for some fixed $\ell \geq 4$, the gap between $\D_{\mathrm{GH}}$ and $\D_{\rm{AM}}$ appears, and the story is not as elegant as before. According to \cite{Kim2018StochasticBM}, for any algorithm, it is impossible to accomplish exact recovery when $\D_{\mathrm{GH}}<1$ (\textcolor{black}{red} area). On the other hand, when $\D_{\rm{AM}} > 1$ (\textcolor{black}{blue} area), both spectral method and SDP achieve strong consistency using the \emph{aggregated adjacency matrix} $\rmA$ \eqref{eqn:adjacency_matrix_entry_binary}, as proved in \cite{Gaudio2023CommunityDI}. When $\D_{\rm{AM}} < 1 < \D_{\mathrm{GH}}$ (\textcolor{black}{green} area), strong consistency can be accomplished via some \emph{two-stage} algorithms, e.g. \cite{Chien2019MinimaxMR, Zhang2023ExactRI}, where the adjacency tensor $\tA^{(\ell)}$ is required as input. However, it remains unclear whether strong consistency can be achieved using purely the aggregated adjacency matrix $\rmA$ when $\D_{\rm{AM}} < 1 < \D_{\mathrm{GH}}$. In terms of the efficiency of information usage, we raise the following question.
\begin{quote}
    \emph{1. Is it possible to achieve strong consistency when $\D_{\mathrm{GH}}> 1$ using purely the adjacency matrix $\rmA$}?
\end{quote}

For \textbf{non-uniform} hypergraphs, the literature is less explored. It was proved by \cite{Alaluusua2023MultilayerHC} that strong consistency can be accomplished by SDP when $\D_{\rm{AM}} > 1$ (\textcolor{black}{blue} in \Cref{fig:thresholds_plots} (c)) with time complexity $O(N^{3.5})$ \cite{Jiang2020FasterIP}. The phase transition phenomena was presented in \cite{Dumitriu2023OptimalAE}, where the authors proved that strong consistency can be accomplished if and only if $\D_{\mathrm{GH}} > 1$. There, the authors presented two algorithms to achieve strong consistency, both \emph{two-stage} (spectral clustering + refinement), where the two refinement algorithms are based on \emph{majority voting} and \emph{Maximum A Posteriori} (MAP), respectively. The algorithms in \cite{Dumitriu2023OptimalAE} may be a bit redundant for us due to the symmetry in our setting. 
\begin{quote}
\emph{2. Can we design a new refinement algorithm tailored to the symmetric binary case, which achieves strong consistency when $\D_{\mathrm{GH}} > 1$ using purely the adjacency matrix $\rmA$}?
\end{quote}
Besides that, non-uniform HSBM is by definition a multi-layer model, where each uniform layer may contain different amounts of information about the underlying communities. We raise the following natural question.
\begin{quote}
\emph{3. Does aggregating information from multiple uniform layers yield better results than considering each layer in isolation}?
\end{quote}

\subsubsection{Information-Theoretic lower bound}
We start with a simple example to illustrate the concept of the IT lower bound. The model \ref{def:non_uniform_HSBM_binary} is reduced to the graph case with binary symmetric communities when taking $\sL = \{2\}$. Consider the following two extreme scenarios. First, $\alpha_2 = 1$ and $\beta_2 = 0$, then two disjoint cliques are drawn and it is trivial to achieve strong consistency. Second, when $\alpha_2 = \beta_2$, no algorithm could outperform random guess, since the communities themselves are indistinguishable. From this point of view, once the parameters $\{\alpha_{\ell}\}_{\ell \in \sL}, \{\beta_{\ell}\}_{\ell \in \sL}$ in model \ref{def:non_uniform_HSBM_binary} are determined, for any algorithm, there exists a certain mathematical limit on clustering accuracy. In addition, to achieve certain accuracy, the parameters of the underlying model \ref{def:non_uniform_HSBM_binary} have to satisfy certain conditions.

The IT lower bound are less explored in the literature. For graphs, it was studied by \cite{Abbe2015CommunityDI, Zhang2016MinimaxRO, Abbe2020EntrywiseEA, Abbe2022LPT} for the binary case and \cite{Gao2017AchievingOM, Yun2016OptimalCR, Zhang2023FundamentalLO} for the multi-community case. For uniform hypergraphs, \cite{Chien2019MinimaxMR} proved that when restricting the model parameters on some (potential) subset (Eq.(1) therein) of the \textcolor{black}{red} in \Cref{fig:thresholds_plots} (a)-(c), some minimax lower bound on the expected mismatch ratio (Theorem 7.1) can be obtained. However, it remains unclear whether such a lower bound exists for the entire \textcolor{black}{red} area in \Cref{fig:thresholds_plots} (a)-(c) and how to characterize it explicitly in terms of model parameters. We thus raise the following two questions.

\begin{quote}
\noindent \emph{4. Is it possible to characterize the accuracy limit explicitly in terms of model parameters} ? \\
\noindent \emph{5. If so, it is possible to achieve IT lower bound efficiently} ?
\end{quote}
\subsection{Contributions}
Our contribution to the non-uniform binary HSBM can be summarized as follows.
\begin{enumerate}[label=(\arabic*), leftmargin=2em]
    \item We establish the threshold $\D_{\mathrm{GH}} = 1$ for strong consistency. We proved that it is impossible to achieve strong when $\D_{\mathrm{GH}} < 1$ in \Cref{thm:impossibility_binary}.
    \item When only the aggregated adjacency matrix $\rmA$ is given, we propose Algorithms \ref{alg:refinement_spectral} and \ref{alg:refinement_voting_binary}, both provably achieving strong consistency when $\D_{\mathrm{GH}} >1$ by refining the initial estimate obtained from Algorithm \ref{alg:spectral_partition_binary}. This answers question 1.
    \item Algorithm \ref{alg:refinement_spectral} is a novel approach in the literature, where the refinement stage utilizes the power iteration method on a carefully designed matrix, answering question 2.
    \item We demonstrate that all the uniform layers contribute non-negatively to clustering accuracy, showing that aggregating information from multiple uniform layers yields better results compared to considering each layer in isolation, answering question 3.
    \item We present the IT lower bound on expected mismatch ratio for any algorithm in Theorem \ref{thm:IT_lower_bound_binary}, informally stating that $\E \eta_{N} \geq N^{-\D_{\rm{GH}}}$, answering question 4.
    \item Both Algorithms \ref{alg:refinement_spectral} and \ref{alg:refinement_voting_binary} achieves the IT lower bound when $\D_{\rm{GH}} < 1$, answering question 5.
\end{enumerate}

\subsection{Organization} 
We present our main results in \Cref{sec:main_results_binary}. \Cref{sec:information_theoretic_limits_binary} provides the proof outline for the information-theoretic limits. The performance of \Cref{alg:spectral_partition_binary}, using purely the adjacency matrix, is analyzed in \Cref{sec:achievability_matrices_binary}, while the analysis using the normalized Laplacian matrix is deferred to \Cref{app:achievabilityLap}. The analysis of Algorithms \ref{alg:refinement_spectral} and \ref{alg:refinement_voting_binary} is presented in \Cref{sec:refinement_binary}. Finally, \Cref{sec:conclusion_binary} concludes the paper. Detailed proofs for all sections are provided in the appendices.

\section{Main Results}\label{sec:main_results_binary}

\subsection{Information-Theoretic limits}
We first present the necessary condition to achieve strong consistency.
\begin{theorem}[Impossibility of exact recovery]\label{thm:impossibility_binary}
For model \ref{def:non_uniform_HSBM_binary} under Assumption \ref{ass:asymptotics_binary}, suppose
\begin{align}
    \limsup_{N \to \infty} \, \D_{\rm{GH}}\cdot q_{N}/\log(N) < 1, \label{eqn:impossibility_condition_binary}
\end{align}
then with probability tending to $1$, every algorithm will misclassify at least two vertices.
\end{theorem}
\begin{remark}
    Here, $\D_{\mathrm{GH}}$ is the rate function derived from the \emph{Large Deviation Principle (LDP)} analysis using the adjacency tensors $\{\tA^{(\ell)}\}_{\ell \in \sL}$. The detailed arguments are deferred to Lemmas \ref{lem:binom_difference_prob} and \ref{lem:LDP_binomial}.
\end{remark}
The condition \eqref{eqn:impossibility_condition_binary} also covers the scenario when $\D_{\rm{GH}} \gg 1$ ($a_{\ell} \gg b_{\ell}$) and $1\ll q_{N} \ll \log(N)$, which was ignored in previous literature. In addition, when $q_{N} = \log(N)$ and $\sL = \{\ell\}$ for some fixed integer $\ell\geq 2$, \eqref{eqn:impossibility_condition_binary} can be simplified as $\D_{\mathrm{GH}} < 1$, which is the necessary condition of strong consistency for $\ell$-uniform hypergraphs in \cite{Kim2018StochasticBM}. By taking $\sL = \{2\}$, \eqref{eqn:impossibility_condition_binary} becomes $(\sqrt{a_{2}} - \sqrt{b_{2}})^{2} < 1$, as derived in \cite{Abbe2016ExactRI} for graphs.

As will be noted in \Cref{rem:model_difference}, model \ref{def:non_uniform_HSBM_binary} differs from that in \cite{Dumitriu2023OptimalAE}, leading to distinct proof techniques. Interestingly, when the model in \cite{Dumitriu2023OptimalAE} is specialized to the symmetric binary case considered here, their threshold $\D_{\rm{GCH}}$ coincides with our $\D_{\rm{GH}}$. This serves as a mutual validation of both thresholds.

Furthermore, we establish a lower bound on the reconstruction error rate for any estimator, explicitly characterized by the model parameters.
\begin{theorem}[Lower bound on the expected mismatch ratio]\label{thm:IT_lower_bound_binary}
Under Assumption \ref{ass:asymptotics_binary}, for any sequence of estimators $\widehat{\rvy}$, the following holds:
    \begin{align}
        \liminf_{N \to \infty} \, q_{N}^{-1} \log \E\,\eta_{N}(\rvy, \widehat{\rvy}) \geq -\D_{\mathrm{GH}}.
    \end{align}
\end{theorem}
Informally, \Cref{thm:IT_lower_bound_binary} can be interpreted as $\E \eta_{N} \geq e^{-\D_{\mathrm{GH}} \cdot q_{N}}$, which can be further simplified as $\E \eta_{N} \geq N^{-\D_{\rm{GH}}}$ when $q_{N} = \log(N)$. 

\subsection{Spectral partition using purely the adjacency matrix}
\Cref{alg:spectral_partition_binary} only takes the aggregated adjacency matrix $\rmA$ in \eqref{eqn:adjacency_matrix_entry_binary} or its normalized Laplacian as input.

\begin{algorithm}
\caption{\textbf{Spectral partition via the second eigenvector}}\label{alg:spectral_partition_binary}

\KwData{The $N \times N$ matrix $\rmX$, either $\rmX = \rmA$ or $\rmX = \rmL \coloneqq \rmD^{-1/2} \rmA \rmD^{-1/2}$.}

{Compute the eigenvector $\rvu_2$ corresponding to the second largest eigenvalue $\lambda_2$ of $\rmX$.}

{Construct $\widehat{\gV}^{(0)}_{+} = \{v\mid \widehat{\ervy}_{v}^{(0)} >0 \}$ and $\widehat{\gV}^{(0)}_{-} = \{v\mid \widehat{\ervy}_{v}^{(0)} < 0 \}$, where $\widehat{\rvy}^{(0)} = \sign(\rvu_2)$.}

\KwResult{$\widehat{\gV}^{(0)}_{+}$ and $\widehat{\gV}^{(0)}_{-}$.}
\end{algorithm}

\begin{theorem}[Strong consistency]\label{thm:achievability_matrices}
Under Assumption \ref{ass:asymptotics_binary}, for some absolute $\epsilon >0$, suppose
\begin{align}\label{eqn:exact_recovery_condition_matrices}
    \liminf_{N \to \infty} \, \D_{\rm{AM}}\cdot q_{N} /\log(N) \geq  1 + \epsilon,
\end{align}
then with probability tending to $1$, Algorithm \ref{alg:spectral_partition_binary} accomplishes strong consistency.
\end{theorem}

When $q_{N} = \log(N)$, \eqref{eqn:exact_recovery_condition_matrices} can be simplified as $\D_{\rm{AM}} \geq 1 + \epsilon$, which can be reduced to the result in \cite{Gaudio2023CommunityDI} for the uniform hypergraph case, where only the adjacency matrix is taken as input. In addition, we show that the second eigenvector of the \emph{normalized Laplacian} $\rmL = \rmD^{-1/2} \rmA \rmD^{-1/2}$ can be used for clustering. 

\begin{theorem}[Optimality]\label{thm:optimality_matrices}
    Suppose Assumption \ref{ass:asymptotics_binary} and the condition \eqref{eqn:impossibility_condition_binary}. After running Algorithm \ref{alg:spectral_partition_binary}, the following holds with high probability
        \begin{align}
            \limsup_{N \to \infty} q_{N}^{-1} \log \E \eta_{N}(\rvy, \widehat{\rvy}^{(0)}) \leq - \D_{\rm{AM}}.\notag
        \end{align}
\end{theorem}
Informally, \Cref{thm:optimality_matrices} gives an upper bound $e^{-\D_{\rm{AM}}\cdot q_{N}}$ on the expected mismatch ratio of the estimator $\widehat{\rvy}^{(0)}$. This bound is larger than the lower bound given in \Cref{thm:IT_lower_bound_binary} when $\D_{\rm{GH}} < 1$ since $\D_{\rm{AM}} \leq \D_{\rm{GH}}$, meaning that $\widehat{\rvy}^{(0)}$ doesn't attain the lowest possible mismatch ratio.


\subsection{Refinement algorithms}
As shown in previous analysis, $\widehat{\rvy}^{(0)}$ obtained from Algorithm \ref{alg:spectral_partition_binary} neither achieves strong consistency when $\D_{\mathrm{GH}} > 1 > \D_{\rm{AM}}$, nor does it attain the lowest possible mismatch ratio when $\D_{\mathrm{GH}} < 1$. To bridge this gap, we propose Algorithms \ref{alg:refinement_spectral} and \ref{alg:refinement_voting_binary}, both of which achieve strong consistency as long as $\D_{\mathrm{GH}} > 1$, and attain the IT lower bound on the expected mismatch ratio established in \Cref{thm:IT_lower_bound_binary} when $\D_{\mathrm{GH}} > 1$. \Cref{alg:refinement_spectral} introduces a novel refinement approach based on power iteration, while \Cref{alg:refinement_voting_binary} employs the majority voting scheme. Both algorithms take the initial estimate $\widehat{\rvy}^{(0)}$ from Algorithm \ref{alg:spectral_partition_binary} and the set of adjacency tensors $\{\tA^{(\ell)}\}_{\ell \in \sL}$ as input. 

It is worth noting that any initial estimate $\widehat{\rvy}^{(0)}$ satisfying the error rate condition $\eta_{N}(\rvy, \widehat{\rvy}^{(0)}) < q_{N}^{-3}$ is sufficient for both refinement algorithms. Furthermore, the explicit requirement for the adjacency tensors $\{\tA^{(\ell)}\}_{\ell \in \sL}$ can be relaxed. As demonstrated in \cite{bresler2024thresholds}, the $\ell$-th order tensor $\tA^{(\ell)}$ can be exactly reconstructed with high probability from its corresponding adjacency matrix $\rmA^{(\ell)}$. However, for non-uniform hypergraphs, access to the individual layer-wise matrices $\rmA^{(\ell)}$ for each $\ell \in \sL$ remains necessary for this reconstruction, as it is currently unknown whether the tensors $\{\tA^{(\ell)}\}_{\ell\in \sL}$ can be recovered exactly when only the aggregated matrix $\rmA \coloneqq \sum_{\ell\in \sL} \rmA^{(\ell)}$ is available.

\subsubsection{Refinement via power iteration}
We start by introducing necessary concepts. For each vertex $v \in \gV$, we categorize the associated edges based on the community membership of their vertices. For each $0\leq r \leq \ell - 1$ with $\ell \in \sL$, let $\gE^{(r)}_{\ell}(v)$ denote the set of $\ell$-hyperedge containing $v$ such that exactly $r$ vertices in the edge have different community membership from $v$, formally,
\begin{align}
     \gE^{(r)}_{\ell}(v) \coloneqq \{e\in [\gV]^{\ell} \mid v \in e,\,\, |\{u \in e\setminus\{v\} : \ervy_u \neq \ervy_v\}| = r\}. \label{eqn:edge_classification}
\end{align} 
We can approximate $\gE^{(r)}_{\ell}(v)$ by constructing the $\ell$-hyperedge set $\widehat{\gE}^{(r)}_{\ell}(v)$ based on the initial partition $\widehat{\rvy}^{(0)}$. Formally, we define
\begin{align}
     \widehat{\gE}^{(r)}_{\ell}(v) \coloneqq \{e\in [\gV]^{\ell} \mid v \in e,\,\, |\{u \in e\setminus\{v\} : \widehat{\ervy}^{(0)}_u \neq \widehat{\ervy}^{(0)}_v\}| = r\}. \label{eqn:hat_edge_classification}
\end{align}
We consider vertices contained by edge $e\in \widehat{\gE}^{(r)}_{\ell}(v)$ other than $v$. Given $\ell \in \sL$, for each $j\in e \setminus \{v\}$, define the entry $\widehat{\etA}^{(\ell)}_{e, v, j} \coloneqq \etA^{(\ell)}_{e}/(\ell - 1)$ if $r=0, \ell - 1$; when $1\leq r \leq \ell - 2$, we let
\begin{align}
    \widehat{\etA}^{(\ell)}_{e, v, j} \coloneqq \etA^{(\ell)}_{e}/r,\,\, \textnormal{ if }\widehat{\ervy}^{(0)}_v = \widehat{\ervy}^{(0)}_j; \quad \widehat{\etA}^{(\ell)}_{e, v, j} \coloneqq \etA^{(\ell)}_{e}/(\ell - 1-r), \,\,\textnormal{ if }\widehat{\ervy}^{(0)}_v \neq \widehat{\ervy}^{(0)}_j. \label{eqn:hatAevj}
\end{align}
Furthermore, for each $\ell \in \sL$, we compute the following quantity:
\begin{align}
    \widehat{\psi}_{\ell} = (\ell - 2)\log(2) + \log\left( \frac{ \ones_{N}^{\sT} \rmA^{(\ell)} \ones_{N} + (\widehat{\rvy}^{(0)})^{\sT} \rmA^{(\ell)} \widehat{\rvy}^{(0)} }{ \ones_{N}^{\sT} \rmA^{(\ell)} \ones_{N} - (\widehat{\rvy}^{(0)})^{\sT} \rmA^{(\ell)} \widehat{\rvy}^{(0)} } - 1 + 2^{-\ell + 2}\right).\label{eqn:psi_hat}
\end{align}
Then, we construct the weighted adjacency matrix $\widehat{\rmA}$ based on $\widehat{\rvy}^{(0)}$, with each entry defined by
\begin{align}
    \widehat{\ermA}_{vj} \coloneqq \indi{v \neq j} \cdot \sum_{\ell \in \sL} \widehat{\psi}_{\ell} \sum_{r=0}^{\ell - 1} \sum_{e \in \widehat{\gE}^{(r)}_{\ell}(v) } \widehat{\etA}^{(\ell)}_{e, v, j}\,.\label{eqn:hatA}
\end{align}
Inspired by \cite{Abbe2022LPT}, we propose \Cref{alg:refinement_spectral} to refine the initial estimate $\widehat{\rvy}^{(0)}$. The procedure begins by constructing the weighted adjacency matrix $\widehat{\rmA}$ as defined in \eqref{eqn:hatA}. Results concerning the performance of \Cref{alg:refinement_spectral} will be presented in \Cref{thm:achieveability_binary} (1) and \Cref{thm:optimal_recovery_binary} (1).
\begin{algorithm}
\caption{\textbf{Refinement via power iteration}}\label{alg:refinement_spectral}
\KwData{The tensors $\{\tA^{(\ell)}\}_{\ell \in \sL}$ and initial estimation $\widehat{\rvy}^{(0)}$.}

{Construct the weighted adjacency matrix $\widehat{\rmA}$ via \eqref{eqn:hatAevj} and \eqref{eqn:hatA}.}

{Construct $\widehat{\rvy} = \sign(\widehat{\rmA}\widehat{\rvy}^{(0)})$ with $\widehat{\gV}_{+} = \{v\mid \widehat{\ervy}_v >0 \}$ and $\widehat{\gV}_{-} = \{v\mid \widehat{\ervy}_v < 0 \}$.}

\KwResult{The refined partition $\widehat{\gV}_{+}$ and $\widehat{\gV}_{-}$.}
\end{algorithm}

\subsubsection{Refinement via majority voting}
We now introduce the second refinement algorithm, which is based on the majority voting scheme. Unlike \Cref{alg:refinement_spectral}, which relies on a single step of power iteration, \Cref{alg:refinement_voting_binary} iteratively updates the partition over multiple rounds to ensure convergence to the true community assignment. 

Let $t\in \N$ denote the iteration index. We define $\widehat{\rvy}^{(t)}$ as the estimated label vector after iteration $t$ where $\widehat{\rvy}^{(0)}$ represents the initial estimate, with the corresponding vertex partition denoted by $\widehat{\gV}^{(t)}_{+}$ and $\widehat{\gV}^{(t)}_{-}$. Furthermore, let $\widehat{\gE}^{(t)}_{\alpha_{\ell}}$ represent the set of $\ell$-hyperedges where all vertices belong to the same partition (either $\widehat{\gV}^{(t)}_{+}$ or $\widehat{\gV}^{(t)}_{-}$), and let $\widehat{\gE}^{(t)}_{\beta_{\ell}} \coloneqq \gE_{\ell} \setminus \widehat{\gE}^{(t)}_{\alpha_{\ell}}$ denote the remaining edges. For each $v\in \gV$, let $\widehat{\gE}^{(t)}_{\alpha_{\ell}}(v)$ (resp. $\widehat{\gE}^{(t)}_{\beta_{\ell}}(v)$) denote the subset of $\widehat{\gE}^{(t)}_{\alpha_{\ell}}$ (resp. $\widehat{\gE}^{(t)}_{\beta_{\ell}}$) containing $v$. Formally, we define
\begin{subequations}
\begin{align}
     \widehat{\gE}^{(t)}_{\alpha_{\ell}} \coloneqq \Big\{\{i_1, \ldots, i_{\ell} \} \in \gE_{\ell}: \widehat{\ervy}^{(t)}_{i_1} = \cdots = \widehat{\ervy}^{(t)}_{i_{\ell}} \Big\}, \quad &\, \widehat{\gE}^{(t)}_{\alpha_{\ell}}(v) \coloneqq \Big\{e \in \widehat{\gE}^{(t)}_{\alpha_{\ell}} : v \in e \Big\}\,,\label{eqn:Etl_alpha}\\
     \widehat{\gE}^{(t)}_{\beta_{\ell}} \coloneqq \gE_{\ell} \setminus \widehat{\gE}^{(t)}_{\alpha_{\ell}}, \quad &\, \widehat{\gE}^{(t)}_{\beta_{\ell}}(v) \coloneqq \Big\{e \in \widehat{\gE}^{(t)}_{\beta_{\ell}} : v \in e \Big\}.\label{eqn:Etl_beta}
\end{align}
\end{subequations}
Let $\widehat{N}^{(t)}_{\alpha_{\ell}}$ and $\widehat{N}^{(t)}_{\alpha_{\ell}}(v)$ denote the capacity of $\widehat{\gE}^{(t)}_{\alpha_{\ell}}$ and $\widehat{\gE}^{(t)}_{\alpha_{\ell}}(v)$ respectively, formally defined as
\begin{align}
    \widehat{N}^{(t)}_{\alpha_{\ell}} \coloneqq \binom{|\widehat{\gV}^{(t)}_{+}|}{\ell} + \binom{|\widehat{\gV}^{(t)}_{-}|}{\ell}, \quad \widehat{N}^{(t)}_{\alpha_{\ell}}(v) \coloneqq \binom{|\widehat{\gV}^{(t)}_{+}| - 1}{\ell - 1} + \binom{|\widehat{\gV}^{(t)}_{-}| - 1}{\ell - 1}. \label{eqn:Ntl_alpha}
\end{align}
Similarly, after iteration $t\geq 0$, let 
\begin{align}
    \widehat{N}^{(t)}_{\beta_{\ell}} \coloneqq \binom{N}{\ell} - \widehat{N}^{(t)}_{\alpha_{\ell}} \quad \textnormal{and} \quad \widehat{N}^{(t)}_{\beta_{\ell}}(v) \coloneqq \binom{N-1}{\ell - 1} - \widehat{N}^{(t)}_{\alpha_{\ell}}(v)\,,
\end{align}
denote the capacity of $\widehat{\gE}^{(t)}_{\beta_{\ell}}$ and $\widehat{\gE}^{(t)}_{\beta_{\ell}}(v)$ respectively.

First, we compute the following two quantities using all the in-cluster and cross-cluster edges, respectively,
\begin{align}
    \widehat{\alpha}_{\ell} = |\widehat{\gE}^{(0)}_{\alpha_{\ell}}|/ \widehat{N}^{(0)}_{\alpha_{\ell}}, \quad \widehat{\beta}_{\ell} = |\widehat{\gE}^{(0)}_{\beta_{\ell}}|/ \widehat{N}^{(0)}_{\beta_{\ell}}\,, \quad \ell \in \sL. \label{eqn:hatalphabeta}
\end{align}
which serve as estimates for the probabilities $\alpha_{\ell}$ and $\beta_{\ell}$ for each $\ell \in \sL$. Additionally, at each iteration $t\geq 0$, we estimate $\alpha_{\ell}$ and $\beta_{\ell}$ locally for each vertex $v$ using only the incident edges, defined as
\begin{align}
    \widehat{\alpha}^{(t)}_{\ell}(v) = |\widehat{\gE}^{(t)}_{\alpha_{\ell}}(v)|/ \widehat{N}^{(t)}_{\alpha_{\ell}}(v), \quad \widehat{\beta}^{(t)}_{\ell}(v) = |\widehat{\gE}^{(t)}_{\beta_{\ell}}(v)|/ \widehat{N}^{(t)}_{\beta_{\ell}}(v)\,,  \label{eqn:hatalphabeta_v_t}
\end{align}
where $\widehat{N}^{(t)}_{\alpha_{\ell}}(v)$ and $\widehat{N}^{(t)}_{\beta_{\ell}}(v)$ are defined in \eqref{eqn:Ntl_alpha}. Denote the \emph{cross-entropy} by
\begin{align}
    \mathbb{H}_{\mathrm{CE}}(y, \widehat{y}) \coloneqq - y\log \widehat{y} - (1 - y) \log (1 - \widehat{y}), \quad y, \widehat{y} \in [0, 1]. \label{eqn:cross-entropy}
\end{align}
Then for each $v\in \gV$, its alignment with the estimated partition at iteration $t$ is characterized by the following weighted cross-entropy terms:
\begin{subequations}
    \begin{align}
        \widehat{f}^{(t)}_{\mathrm{in}}(v) =&\, \sum_{\ell \in \sL} \widehat{N}^{(t)}_{\alpha_{\ell}}(v) \cdot \mathbb{H}_{\mathrm{CE}}(\widehat{\alpha}^{(t)}_{\ell}(v), \widehat{\alpha}_{\ell}) + \widehat{N}^{(t)}_{\beta_{\ell}}(v) \cdot \mathbb{H}_{\mathrm{CE}}(\widehat{\beta}^{(t)}_{\ell}(v), \widehat{\beta}_{\ell}) \label{eqn:fin}\\
        \widehat{f}^{(t)}_{\mathrm{cross}}(v) =&\, \sum_{\ell \in \sL} \widehat{N}^{(t)}_{\alpha_{\ell}}(v) \cdot \mathbb{H}_{\mathrm{CE}}(\widehat{\alpha}^{(t)}_{\ell}(v), \widehat{\beta}_{\ell}) + \widehat{N}^{(t)}_{\beta_{\ell}}(v) \cdot \mathbb{H}_{\mathrm{CE}}(\widehat{\beta}^{(t)}_{\ell}(v), \widehat{\alpha}_{\ell}) \label{eqn:fcross}
    \end{align}
\end{subequations}

With all the quantities defined above, we can now present the Algorithm \ref{alg:refinement_voting_binary}, inspired by \cite{Yun2016OptimalCR, Dumitriu2023OptimalAE}. The algorithm iteratively refines the initial partition $\widehat{\gV}^{(0)}_{+}$ and $\widehat{\gV}^{(0)}_{-}$ by flipping the sign of vertices based on the majority voting principle, and finally achieves strong consistency after $\lceil \log(N) \rceil$ iterations. Results concerning the performance of \Cref{alg:refinement_voting_binary} will be presented in \Cref{thm:achieveability_binary} (2) and \Cref{thm:optimal_recovery_binary} (2).

\begin{algorithm}
\caption{\textbf{Refinement via majority voting}}\label{alg:refinement_voting_binary}
\KwData{The tensors $\{\tA^{(\ell)}\}_{\ell \in \sL}$ and an initial estimation $\widehat{\rvy}^{(0)}$.}

{Compute $\widehat{\alpha}_{\ell}$ and $\widehat{\beta}_{\ell}$ defined in \eqref{eqn:hatalphabeta}.}

\For{$t = 1, \ldots, \lceil \log(N) \rceil$}{
    {Initialize $\widehat{\gV}^{(t)}_{+} \leftarrow \widehat{\gV}^{(t - 1)}_{+}$, $\widehat{\gV}^{(t)}_{-} \leftarrow \widehat{\gV}^{(t-1)}_{-}$.}
    
    \For{$v\in \gV$}{
        {Compute $\widehat{f}^{(t)}_{\mathrm{in}}$ and $\widehat{f}^{(t)}_{\mathrm{cross}}$ defined in \eqref{eqn:fin} and \eqref{eqn:fcross} respectively.}
        
        {Flip the sign of $v$ if $\widehat{f}^{(t)}_{\mathrm{in}} > \widehat{f}^{(t)}_{\mathrm{cross}}$, otherwise keep the sign unchanged.}    
    }
}
\KwResult{The final partition $\widehat{\gV}_{+} = \widehat{\gV}^{(\lceil \log(N) \rceil)}_{+}$, \,\,$\widehat{\gV}_{-} = \widehat{\gV}^{(\lceil \log(N) \rceil)}_{-}$.}
\end{algorithm}

The mechanism of Algorithm \ref{alg:refinement_voting_binary} can be intuitively understood as follows. The global parameters $\widehat{\alpha}_{\ell}$ and $\widehat{\beta}_{\ell}$ in \eqref{eqn:hatalphabeta} are estimated using all hyperedges (at least $N\log(N)$ of each type), ensuring sufficient concentration around their expectations. Consequently, they serve as reliable ``ground truth'' references in the cross-entropy term $\mathbb{H}_{\mathrm{CE}}(y, \widehat{y})$ defined in \eqref{eqn:cross-entropy}.
In contrast, the local estimates $\widehat{\alpha}^{(t)}_{\ell}(v)$ and $\widehat{\beta}^{(t)}_{\ell}(v)$ in \eqref{eqn:hatalphabeta_v_t} rely solely on hyperedges incident to $v$ (approximately $\log(N)$ edges). These local estimates align with their expectations $\alpha_{\ell}$ and $\beta_{\ell}$ only if $v$ is correctly classified. Since the cross-entropy $\mathbb{H}_{\mathrm{CE}}(y, \widehat{y})$ is minimized when $\widehat{y}$ approaches $y$, a smaller value of $\widehat{f}^{(t)}_{\mathrm{in}}(v)$ (resp. $\widehat{f}^{(t)}_{\mathrm{cross}}(v)$) indicates a stronger alignment of $v$'s neighbors with the current (resp. opposite) community partition.

\subsection{Strong consistency after refinement}
Below, we present performance guarantees for two refinement algorithms, demonstrating that both achieve strong consistency provided that condition \eqref{eqn:achievabilityGH} is satisfied.
\begin{theorem}[Achievability]\label{thm:achieveability_binary}
Under Assumption \ref{ass:asymptotics_binary}, for some absolute $\epsilon >0$, suppose 
\begin{align}
    \liminf_{N \to \infty} \,\, \D_{\rm{GH}}\cdot q_{N} /\log(N) \geq  1 + \epsilon. \label{eqn:achievabilityGH}
\end{align}
\begin{enumerate}
    \item \Cref{alg:refinement_spectral} achieves strong consistency with probability $1 - N^{-\epsilon}$.
    \item Algorithm \ref{alg:refinement_voting_binary} attains strong consistency with probability at least $1 -6N^{-\epsilon}$.
\end{enumerate}

\end{theorem}
Note that \eqref{eqn:achievabilityGH} can be reduced to $\D_{\rm{GH}} > 1 + \epsilon$ when $q_{N} = \log(N)$. \Cref{thm:impossibility_binary} and \Cref{thm:achieveability_binary} establish the sharp phase transition for the achievability of exact recovery. The same threshold was established for the graph case $\sL = \{2\}$ in \cite{Abbe2016ExactRI}. For the $\ell$-uniform hypergraph, the necessity was proved by \cite{Kim2018StochasticBM}, and the regime $\D_{\rm{AM}} > 1$ was covered by \cite{Gaudio2023CommunityDI}. We fill the gap when $\D_{\rm{AM}} < 1 < \D_{\rm{GH}}$ (\textcolor{black}{green}) and establish the threshold for a more general case. 

\paragraph{\textbf{Benefits of aggregation}} Consider the model \ref{def:non_uniform_HSBM_binary} with $\sL = \{2, 3\}$, where $a_2 = 4$, $b_2 = 1$, and the divergence for the graph layer $\ell = 2$ can be calculated as $\D^{(2)}(a_2, b_2) = (\sqrt{a_2} - \sqrt{b_2})^2/2 = 1/2$. When the divergence for the layer $\ell = 3$ satisfies $\D^{(3)}(a_3, b_3) > 1/2$, the entire non-uniform hypergraph satisfies $\D_{\mathrm{GH}} > 1$ according to \eqref{eqn:DGH_binary}, and the strong consistency can be efficiently achieved. Compared with the $3$-uniform HSBM, where the strong consistency can be achieved if and only if $\D^{(3)}(a_3, b_3) > 1$, the difficulty of the task is reduced when different layers are aggregated. Unlike the constant expected degree scenario \cite{Dumitriu2025PartialRA} where a subset selection is necessary to enhance the signal-to-noise ratio, here all layers contribute non-negatively to clustering accuracy.
 
\subsection{Optimality of refinement algorithms}
We now establish the optimality of the two-stage algorithms consisting of Algorithm \ref{alg:spectral_partition_binary} and Algorithm \ref{alg:refinement_voting_binary} in terms of minimizing the expected mismatch ratio $\E \eta_{N}(\rvy, \widehat{\rvy})$.
\begin{theorem}[Optimality]\label{thm:optimal_recovery_binary}
    Suppose Assumption \ref{ass:asymptotics_binary} and the condition \eqref{eqn:impossibility_condition_binary}.
\begin{enumerate}
        \item[(1)] The estimator $\widehat{\rvy}$ by \Cref{alg:refinement_spectral} satisfies
        \begin{align}
            \limsup_{N\to \infty} q_{N}^{-1}\log \E \eta_{N}(\widehat{\rvy}, \rvy) \leq - \D_{\rm{GH}} \label{eqn:optimality_refinement}
        \end{align}
        
        \item[(2)] For the sequence $\kappa_{N} \in (0, 1]$ with $\kappa_{N} \log(N) \to \infty$ as $N \to \infty$, further suppose the following holds
        \begin{align}
        \lim\limits_{N \to \infty} \D_{\mathrm{GH}}\cdot q_{N} /(\kappa_{N}\cdot \log(N)\, ) \geq 1 +  \epsilon,\label{eqn:optimality_condition}
        \end{align}
        for some absolute (small) $\epsilon >0$. Then with probability at least $1 -6e^{-\epsilon \cdot \kappa_{N} \log(N)}$, the estimator $\widehat{\rvy}$ by \Cref{alg:refinement_voting_binary} satisfies 
        \begin{align}
            \lim\limits_{N \to \infty} (N^{\kappa_{N}} \cdot \mismatch_{N}) \leq 1. \notag
        \end{align}
    \end{enumerate}
\end{theorem}
Both Algorithms \ref{alg:refinement_spectral} and \ref{alg:refinement_voting_binary} achieve the lower bound on expected mismatch ratio established in \Cref{thm:IT_lower_bound_binary}, thus proving their optimality. It is worth noting that \Cref{alg:refinement_spectral} requires a milder condition on $\D_{\rm{GH}}$ compared to \Cref{alg:refinement_voting_binary} to attain the IT lower bound. Specifically, \Cref{alg:refinement_spectral} only necessitates $\D_{\rm{GH}} > 1$, while \Cref{alg:refinement_voting_binary} demands a more stringent condition as specified in \eqref{eqn:optimality_condition}. This discrepancy arises from the distinct mechanisms employed by the two algorithms for refinement.

\section{Proof outline of Information-Theoretic limits}\label{sec:information_theoretic_limits_binary}
We present the proof outlines of Theorems \ref{thm:impossibility_binary} and \ref{thm:IT_lower_bound_binary}. The proofs of Lemmas are deferred to \Cref{app:impossibility_binary}.

\subsection{Impossibility of the strong consistency}\label{sec:impossibility_binary}
Let $\widehat{\rvy}$ denote an estimation of $\rvy$ based on the observed hypergraph $\gH$ where $\mathbb{H}$ denotes its law. The probability of failing exact recovery is given by
\begin{align}
    \P_{\mathrm{fail}} \coloneqq \P( \widehat{\rvy} \neq \pm \rvy) = \sum_{\gH} [ 1 - \P( \widehat{\rvy} = \pm \rvy \mid \mathbb{H} = \gH) ]\cdot \P(\mathbb{H} = \gH)\,. \label{eqn:P_fail_binary}
\end{align}
The \emph{Maximum A Posteriori} (MAP) estimator is the best possible estimator in the sense that it minimizes the failure probability $ \P_{\mathrm{fail}}$ in \eqref{eqn:P_fail_binary} among all estimators, which is defined as
\begin{align}
    \widehat{\rvy}_{\mathrm{MAP}}\coloneqq \underset{\rvz \in \{\pm 1\}^{|\gV|}, \, \ones^{\sT}\rvz = 0}{\arg\max}\, \P(\rvy = \rvz \mid \mathbb{H} = \gH)\,.\label{eqn:MAP_binary}
\end{align}
Note that $\widehat{\rvy}_{\mathrm{MAP}}$ is the most likely assignment of $\rvy$ given the observed hypergraph $\gH$. Consequently, no other estimator could succeed exact recovery if $\widehat{\rvy}_{\mathrm{MAP}}$ fails. Furthermore, $\widehat{\rvy}_{\mathrm{MAP}}$ corresponds to the \emph{Maximum Likelihood Estimation} (MLE) since by Bayes' Theorem,
 \begin{align}
 \widehat{\rvy}_{\mathrm{MLE}} & = \underset{\rvz \in \{\pm 1\}^{|\gV|}, \, \ones^{\sT}\rvz = 0}{\arg\max}\,  \P(\mathbb{H} = \gH \mid \rvy = \rvz ) \label{eqn:MLE}\\
 & = \underset{\rvz \in \{\pm 1\}^{|\gV|}, \, \ones^{\sT}\rvz = 0}{\arg\max}\,  \frac{\P(\mathbb{H} = \gH \mid \rvy = \rvz ) \P(\rvy = \rvz)}{\sum\limits_{\rvx \in \{\pm 1\}^{|\gV|},\,\, \ones^{\sT}\rvx = 0} \P(\mathbb{H} = \gH \mid \rvy = \rvx ) \P(\rvy = \rvx)  } \notag \\
 & = \underset{\rvz \in \{\pm 1\}^{|\gV|}, \, \ones^{\sT}\rvz = 0}{\arg\max}\, \P( \rvy = \rvz \mid \mathbb{H} = \gH) = \widehat{\rvy}_{\mathrm{MAP}} \notag
 \end{align}
where the second equality holds since the denominator in the second line is a constant irrelevant to any specific assignment $\rvz$ which can be factored out, and the prior distribution of $\rvy$ is uniform with $\P(\rvy = \rvz) = 1/\binom{|\gV|}{|\gV|/2}$. While MLE provides the optimal estimator in theory, it remains computationally intractable in practice. To make the objective function of the MLE more explicit, we introduce the notion of an \emph{in-cluster} edge.
\begin{definition}[In-cluster edge]\label{def:incluster} 
    The edge $e = \{i_1, \ldots, i_{\ell}\}$ is \emph{in-cluster} with respect to $\rvy$ if $\ervy_{i_1} = \ldots = \ervy_{i_{\ell}}$.
\end{definition}

\begin{lemma}\label{lem:MLE_maximize}
    Let $\gH$ be a hypergraph sampled from model \ref{def:non_uniform_HSBM_binary}. For each $\ell \in \sL$, let $[\gV]^{\ell} \coloneqq \{\setS| \setS\subseteq \gV, |\setS| = \ell \}$ denote the set of $\ell$-subsets of $\gV$, and suppose $\alpha_{\ell} >\beta_{\ell}$. Then, among all the vectors $\rvz\in \{\pm 1\}^{|\gV|}$ with $\ones^{\sT}\rvz = 0$, $\widehat{\rvy}_{\mathrm{MLE}}$ maximizes the following quantity
    \begin{align}
        f(\rvz|\gH) \coloneqq \sum_{\ell \in \sL} \log \Big( \frac{\alpha_{\ell}(1 - \beta_{\ell})}{\beta_{\ell}(1 - \alpha_{\ell})} \Big) \cdot \sum_{e\in [\gV]^{\ell}} \indi{e \textnormal{ is in-cluster w.r.t. } \rvz} \cdot \etA^{(\ell)}_{e}\label{eqn:MLE_maximize_binary}
    \end{align}
    Alternatively, $\widehat{\rvy}_{\mathrm{MLE}}$ minimizes $f(\rvz|\gH)$ when $\alpha_{\ell} < \beta_{\ell}$ (disassortative) for each $\ell \in \sL$.
\end{lemma}

According to the discussions in \eqref{eqn:MAP_binary}, \eqref{eqn:MLE} and \eqref{eqn:MLE_maximize_binary}, \rm{MLE} fails exact recovery when $\widehat{\rvy}_{\mathrm{MLE}}$  returns some $\widetilde{\rvy}$, such that $f(\widetilde{\rvy}|\gH) \geq f(\rvy|\gH)$ but $\widetilde{\rvy}\neq \pm \rvy$. Formally,
\begin{align}\label{eqn:fail_probability}
    \P_{\mathrm{fail}} \coloneqq \P(\widehat{\rvy}_{\mathrm{MLE}} \neq \pm \rvy) = \P(\exists \widetilde{\rvy} \neq \rvy,\, \textnormal{s.t.} \, f(\widetilde{\rvy}|\gH) \geq f(\rvy|\gH)\,).
\end{align}
We construct $\widetilde{\rvy}$ from the true label vector $\rvy$ by flipping the signs of two vertices $a\in \gV_{+}$ and $b\in \gV_{-}$, while keeping the remaining signs unchanged. Specifically, for any $a\in \gV_{+}$ and $b\in \gV_{-}$, we define the vector $\widetilde{\rvy}\in \{\pm 1\}^{|\gV|}$ as
\begin{align}
    \widetilde{\ervy}_v = \begin{cases}
        \ervy_v, & v \in \gV \setminus \{a, b\}, \\
        -\ervy_v, & v \in \{a, b\}
         \end{cases}
        \label{eqn:tilde_rvy}
\end{align}
We will show that, under condition \eqref{eqn:impossibility_condition_binary}, $\widehat{\rvy}_{\rm{MLE}}$ returns $\widetilde{\rvy}$ rather than $\rvy$ with probability tending to $1$.
\begin{remark}
In contrast, the approach in \cite{Dumitriu2023OptimalAE} does not directly analyze the failure of the optimal estimator. Instead, it establishes the existence of $2q_{N}$ pairwise disconnected ambiguous vertices (half from $\gV_{+}$ and half from $\gV_{-}$) under condition \eqref{eqn:impossibility_condition_binary}. These vertices have identical degree profiles, making them indistinguishable; thus, no algorithm can outperform random guessing, resulting in the failure of exact recovery.
\end{remark}

\begin{proof}[Proof of \Cref{thm:impossibility_binary}]
Recall the definition of $\gE_{\ell}^{(r)}$ in \eqref{eqn:edge_classification}. According to \eqref{eqn:MLE_maximize_binary} and \eqref{eqn:tilde_rvy}, the following holds:
\begin{align}
    &\, f(\rvy|\gH) - f(\widetilde{\rvy}|\gH) \notag\\
    = &\,\sum\limits_{\ell \in \sL} \log \Big( \frac{\alpha_{\ell}(1 - \beta_{\ell})}{\beta_{\ell}(1 - \alpha_{\ell})} \Big)\cdot \Big[  \sum\limits_{e\in \gE_{\ell}} \indi{e \textnormal{ is in-cluster w.r.t. } \rvy} - \sum\limits_{e\in \gE_{\ell}} \indi{e \textnormal{ is in-cluster w.r.t. } \widetilde{\rvy}} \Big]\notag\\
    =&\, \sum\limits_{\ell \in \sL} \log \Big( \frac{\alpha_{\ell}(1 - \beta_{\ell})}{\beta_{\ell}(1 - \alpha_{\ell})} \Big)\cdot \Big[ (|\gE^{(0)}_{\ell}(a)| + |\gE^{(0)}_{\ell}(b)|) - (|\gE^{(\ell - 1)}_{\ell}(a)| + |\gE^{(\ell - 1)}_{\ell}(b)|)\Big]\notag\\
    =&\, \rW_{a} + \rW_{b}\,\,,\label{eqn:diff_MLE_function}
\end{align}
where for each $v\in \gV$, the random variable $\rW_{v}$ is defined as
\begin{align}
    \rW_{v} \coloneqq &\, \sum_{\ell \in \sL} \log \Big( \frac{\alpha_{\ell}(1 - \beta_{\ell})}{\beta_{\ell}(1 - \alpha_{\ell})} \Big)\cdot \Big( \sum_{ e\in \gE^{(0)}_{\ell}(v) } \etA^{(\ell)}_{e} - \sum_{ e\in \gE^{(\ell -1)}_{\ell}(v) } \etA^{(\ell)}_{e} \Big). \label{eqn:Wv}
\end{align}
Note that the event $\{\rW_{a} \leq 0\} \cap \{\rW_{b} \leq 0\}$ guarantees $f(\widetilde{\rvy}|\gH) \geq f(\rvy|\gH)$, which yields the following lower bound on the failure probability $\P_{\mathrm{fail}}$ in \eqref{eqn:fail_probability}:
\begin{align}
    \P_{\mathrm{fail}} \geq &\, \P( \exists a\in \gV_{+}, b \in \gV_{-},\, \mathrm{ s.t. }\, \{ \rW_{a} \leq 0\} \cap \{\rW_{b} \leq 0\}\,)\notag \\
    = &\, \P \Big( \underset{a \in \gV_{+}}{\cup} \{\rW_{a} \leq 0\} \bigcap \underset{b \in \gV_{-}}{\cup} \{\rW_{b} \leq 0\} \Big). \label{eqn:lower_bound_Pfail}
\end{align}
However, directly analysis on \eqref{eqn:lower_bound_Pfail} is challenging because $\rW_{a}$ and $\rW_{b}$ are dependent, as there may exist hyperedges containing both $a \in \gV_{+}$ and $b \in \gV_{-}$. Similarly, $\rW_{a}$ and $\rW_{a^{\prime}}$ are dependent for $a, a^{\prime} \in \gV_{+}$.

To address this dependency, we further partition the hyperedges involved in $\rW_{v}$ and introduce pairwise independent random variables. This allows us to lower bound the failure probability using combinations of these new variables. Specifically, let $\gU \subset \gV$ be a subset of vertices of size $\gamma_{N}|\gV|$, with $|\gU \cap \gV_{+}| = |\gU \cap \gV_{-}| = \gamma_{N} |\gV|/2$ and $\gamma_{N}=o(1)$. For $a \in \gV_{+} \cap \gU$ and $b \in \gV_{-} \cap \gU$, define
\begin{subequations}
    \begin{align}
         &\, \rW_{a}[\gU] \label{eqn:WaU} \\
        \coloneqq &\, \sum_{\ell \in \sL} \log \frac{\alpha_{\ell}(1 - \beta_{\ell})}{\beta_{\ell}(1 - \alpha_{\ell})} \cdot \Big( \sum_{ e\in \gE^{(0)}_{\ell}(a) } \etA^{(\ell)}_{e}\indi{\substack{e\cap \gU = \{a\}\\e\setminus\{a\} \subset [\gV_{+} \setminus \gU]}} - \sum_{ e\in \gE^{(\ell - 1)}_{\ell}(a) } \etA^{(\ell)}_{e} \indi{\substack{e\cap \gU = \{a\}\\e\setminus\{a\} \subset [\gV_{-} \setminus \gU]}} \Big),\notag\\
         &\, \rW_{b}[\gU] \label{eqn:WbU} \\
        \coloneqq &\, \sum_{\ell \in \sL} \log \frac{\alpha_{\ell}(1 - \beta_{\ell})}{\beta_{\ell}(1 - \alpha_{\ell})} \cdot \Big( \sum_{ e\in \gE^{(0)}_{\ell}(b) } \etA^{(\ell)}_{e}\indi{\substack{e\cap \gU = \{b\}\\e\setminus\{b\} \subset [\gV_{-} \setminus \gU]}} - \sum_{ e\in \gE^{(\ell - 1)}_{\ell}(b) } \etA^{(\ell)}_{e} \indi{\substack{e\cap \gU = \{b\}\\e\setminus\{a\} \subset [\gV_{+} \setminus \gU]}} \Big). \notag
    \end{align}
\end{subequations}
Clearly, after introducing $\gU$, the edge sets involved in $\rW_{a}[\gU]$ and $\rW_{b}[\gU]$ are disjoint, thus independent. Furthermore, we define the edge set
\begin{align}
    \gW_{\ell}^{(\geq 2)}\coloneqq \{e\in \gE_{\ell} \mid e \textnormal{ contains at least 2 vertices in } \gU \}\,\,. \label{eqn:edgeWm}
\end{align}
It is clear that $\gW_{\ell}^{(\geq 2)}\subset \gE_{\ell}$. For each $v\in \gU$, define the random variable $\rJ_{v}[\sU]$ by
\begin{align}
    \rJ_{v}[\sU] \coloneqq \sum\limits_{\ell \in \sL} \log \Big( \frac{\alpha_{\ell}(1 - \beta_{\ell})}{\beta_{\ell}(1 - \alpha_{\ell})} \Big)\cdot \sum\limits_{e \in \gW_{\ell}^{(\geq 2)}} \etA^{(\ell)}_{e} \indi{e \ni v, \,\, v\in \gU}\,\,. \label{eqn:rJv}
\end{align}
Then for some $\zeta_{N} = o(1)$, we define the random variable $\rJ[\gU]$ as
\begin{align}
    \rJ[\gU] \coloneqq \bigcap_{v\in \gU} \indi{\rJ_{v}[\sU] \leq \zeta_{N} q_{N}}. \label{eqn:rJU}
\end{align}
Lemma \ref{lem:independent_WaWbJ} proves that $\rW_{a}[\gU]$, $\rW_{b}[\gU]$ and $\rJ[\gU]$ are jointly independent.
\begin{lemma}\label{lem:independent_WaWbJ}
    For $a \in \gV_{+} \cap \gU$ and $b \in \gV_{-} \cap \gU$, the random variables $\rW_{a}[\gU]$ and $\rW_{b}[\gU]$, defined in \eqref{eqn:WaU} and \eqref{eqn:WbU} respectively, are independent. Meanwhile, $\rJ[\sU]$, defined in \eqref{eqn:rJU}, is independent of $\rW_{a}$ and $\rW_{b}$. For distinct $a, a^{\prime}\in \gV_{+} \cap \gU$, $\rW_{a}$ and $\rW_{a^{\prime}}$ are independent. Furthermore, the following holds
    \begin{align}
        \{\rW_{a}[\gU] \leq -\zeta_{N} q_{N}\} \cap \{\rJ[\gU] = 1\} &\, \implies \{\rW_{a} \leq 0 \},\notag\\
        \{\rW_{b}[\gU] \leq -\zeta_{N} q_{N}\} \cap \{\rJ[\gU] = 1\}  &\, \implies \{\rW_{b} \leq 0\}, \notag
    \end{align}
    where the random variables $\rW_{v}$ is defined in \eqref{eqn:Wv} for $v\in \gV$.
\end{lemma}
Combining the discussion in \eqref{eqn:lower_bound_Pfail} and Lemma \ref{lem:independent_WaWbJ}, $\P_{\rm{fail}}$ can be further bounded by
\begin{align}
    \P_{\rm{fail}}
    \geq &\, \P \Big( \underset{a \in \gV_{+}}{\cup} \{\rW_{a} \leq 0 \} \bigcap \underset{b \in \gV_{-}}{\cup} \{\rW_{b} \leq 0 \} \Big) \notag\\
    \geq &\, \P \Big( \underset{a \in \gV_{+} \cap \gU }{\cup} \{\rW_{a} \leq 0\} \bigcap  \underset{b \in \gV_{-} \cap \gU}{\cup} \{\rW_{b} \leq 0\} \Big) \notag\\
    \geq &\, \P \Big( \underset{a \in \gV_{+} \cap \gU }{\cup} \{\rW_{a} \leq 0\} \bigcap  \underset{b \in \gV_{-} \cap \gU}{\cup} \{\rW_{b} \leq 0\} \mid \rJ[\gU] = 1\Big)\cdot \P(\rJ[\gU] = 1) \notag\\
    \geq &\, \P \Big( \underset{a \in \gV_{+} \cap \gU }{\cup} \{\rW_{a}[\gU] \leq -\zeta_{N} q_{N}\} \bigcap  \underset{b \in \gV_{-} \cap \gU}{\cup} \{\rW_{b}[\gU] \leq -\zeta_{N} q_{N}\} \mid \rJ[\gU] = 1\Big)\cdot \P(\rJ[\gU] = 1)\notag\\
    \geq &\, \P \Big( \underset{a \in \gV_{+} \cap \gU }{\cup} \{\rW_{a}[\gU] \leq -\zeta_{N} q_{N}\} \Big) \cdot \P\Big(\underset{b \in \gV_{-} \cap \gU}{\cup} \{\rW_{b}[\gU] \leq -\zeta_{N} q_{N}\} \Big) \cdot \P(\rJ[\gU] = 1)\notag\\
    = &\, \Big( 1 - \prod_{a\in \gV_{+} \cap \gU} \P( \rW_{a}[\gU] > -\zeta_{N} q_{N} )  \Big) \cdot \notag \\
    &\, \quad \quad \quad \quad \quad \quad \Big( 1 - \prod_{b\in \gV_{-} \cap \gU} \P( \rW_{b}[\gU] > -\zeta_{N} q_{N} ) \Big) \cdot \Big( 1 - \P(\rJ[\gU] = 0) \Big). \label{eqn:lower_bound_Pfail_second}
\end{align}
The lemma below bounds the probabilities appearing in \eqref{eqn:lower_bound_Pfail_second}, a key step for the subsequent analysis.

\begin{lemma}\label{lem:Wab_J_prob} 
Recall $\D_{\rm{GH}}$ in \eqref{eqn:DGH_binary} and $\rW_{a}[\gU], \rJ[\sU]$ in \eqref{eqn:WaU}, \eqref{eqn:rJU} respectively. The following holds
\begin{align}
        &\, \prod_{a\in \gV_{+} \cap \gU} \P( \rW_{a}[\gU] > -\zeta_{N} q_{N} ) \leq \exp \Big(-\frac{\gamma_{N} \cdot |\gV|}{2} \cdot e^{-\D_{\rm{GH}}\cdot q_{N}} \Big), \label{eqn:Wab_upper_prob}\\
        &\, \P(\rJ[\sU] = 0) \leq \exp\Big( \log(\gamma_{N}N) - q_{N} \cdot \big[\zeta_{N} \log(\zeta_{N}/\gamma_{N}) - (\zeta_{N} - \gamma_{N}) \big]\Big).\label{eqn:JU_upper_prob}
\end{align}
The probabilitiy $\prod_{b\in \gV_{-} \cap \gU} \P( \rW_{b}[\gU] > -\zeta_{N} q_{N} )$ admits the same upper bound as in \eqref{eqn:Wab_upper_prob}.
\end{lemma}
Under the condition \eqref{eqn:impossibility_condition_binary}, there exists some absolute small $\epsilon > 0$ such that $\D_{\rm{GH}}\cdot q_{N} / \log(N) \leq 1 - \epsilon$, thus $Ne^{-\D_{\rm{GH}}\cdot q_{N}} \geq N^{\epsilon}$. Then by choosing $\zeta_{N} = 1/\log(q_{N})$, $\gamma_{N} = N^{-\epsilon/2}$ and substituting \eqref{eqn:Wab_upper_prob} and \eqref{eqn:JU_upper_prob} into \eqref{eqn:lower_bound_Pfail_second}, we obtain
\begin{align}
    \P_{\mathrm{fail}} \geq &\, 1 - 2 \prod_{a\in \gV_{+} \cap \gU} \P( \rW_{a}[\gU] > -\zeta_{N} q_{N} ) \notag \\
    &\, \quad \quad \quad \quad - \P(\rJ[\sU] = 0) - \P(\rJ[\sU] = 0)\cdot \prod_{a\in \gV_{+} \cap \gU} [\P( \rW_{a}[\sU] > -\zeta_{N} q_{N} ) ]^{2} \notag \\
    \geq &\, 1 - 2\exp \Big(-\frac{\gamma_{N} \cdot N}{2} \cdot e^{-\D_{\rm{GH}}\cdot q_{N}} \Big) - \exp\Big( \log(\gamma_{N}N) - q_{N} \cdot \big[\zeta_{N} \log(\zeta_{N}/\gamma_{N}) - (\zeta_{N} - \gamma_{N}) \big]\Big) \notag \\
    &\,\quad \quad \quad \quad - \exp\Big( \log(\gamma_{N}N) - q_{N} \cdot \big[\zeta_{N} \log(\zeta_{N}/\gamma_{N}) - (\zeta_{N} - \gamma_{N}) \big] - \gamma_{N} N \cdot e^{-\D_{\rm{GH}}\cdot q_{N}}\Big) \notag \\
    \geq &\, 1 - 2\exp(-N^{\epsilon/2}/2) \notag\\
    &\, \quad \quad \quad \quad - 2\exp\Big( -q_{N} \cdot \Big[-\big(1 - \frac{\epsilon}{2}\big)\frac{\log(N)}{q_{N}} + \frac{\epsilon\log(N)}{2\log(q_{N})} - \frac{\log\log(q_{N}) + 1}{\log(q_{N})} + N^{-\epsilon/2}\Big]\Big)\notag \\
    = &\, 1 - 2\exp(-N^{\epsilon/2}/2) \notag \\
    &\, \quad \quad \quad \quad - 2\exp\Big( -q_{N} \cdot \Big[ \frac{\epsilon q_{N} - (2-\epsilon)\log(q_{N})}{2q_{N}\log(q_{N})}\log(N) - \frac{\log\log(q_{N}) + 1}{\log(q_{N})} + N^{-\epsilon/2}\Big]\Big)\notag 
\end{align}
where the last term vanishes since $1\ll q_{N} \lesssim \log(N)$ by Assumption \ref{ass:asymptotics_binary} and $\epsilon q_{N} - (2-\epsilon)\log(q_{N}) > 0$ for fixed $\epsilon >0$ and sufficiently large $N$. Consequently, $\P_{\mathrm{fail}} \to 1$ as $N \to \infty$, implying that $\widehat{\rvy}_{\mathrm{MLE}}$ in \eqref{eqn:MLE} fails to recover the true assignment $\rvy$ exactly. Therefore, with probability tending to $1$, it is impossible to finish exact recovery, since no other estimator could outperform $\widehat{\rvy}_{\rm{MLE}}$.
\end{proof}

\subsection{Information-Theoretic lower bound on mismatch ratio}
The proof of \Cref{thm:impossibility_binary} analyzes the optimal estimator $\widehat{\rvy}_{\mathrm{MLE}}$ for the full label vector $\rvy$. In contrast, the proof of \Cref{thm:IT_lower_bound_binary} considers the optimal estimator for each individual label $\ervy_{v}$, and derives the expected number of misclassified vertices by estimating the corresponding error probability. 
For each vertex $v\in \gV$, let $\rvy_{-v} \in \{\pm 1\}^{|\gV|-1}$ denote the label vector of nodes other than $v$. Following the approach in \cite{Abbe2022LPT}, we consider the \emph{genie-aided} estimator
\begin{align}
    \widehat{\ervy}_v^{\,\,\mathrm{genie}} = \underset{s = \pm 1} {\arg \max}~\P \big( \ervy_v = s|\{\tA^{(\ell)}\}_{\ell \in \sL}, \rvy_{-v} \big), \label{eqn:genie-aided}
\end{align}
where $\P \big( \ervy_v = s|\{\tA^{(\ell)}\}_{\ell \in \sL}, \rvy_{-v} \big)$ denotes the \emph{posterior} probability conditioned on $\{\tA^{(\ell)}\}_{\ell \in \sL}$ and $\rvy_{-v}$. However, the estimator in \eqref{eqn:genie-aided} is not directly computable or analyzable, since evaluating the posterior probabilities requires access to the true label vector $\rvy$ in order to compute the likelihood $\P(\{\tA^{(\ell)}\}_{\ell \in \sL}|\ervy_v = s, \rvy_{-v})$, which is unavailable in practice. To address this challenge, Lemma \ref{lem:genie_approximation} reformulates the problem in terms of the \emph{log-odds} ratio, which is more tractable for analysis. Lemma \ref{lem:fundamental_limit_genie-aided} then connects the performance of any estimator to the error probability of the genie-aided estimator.

\begin{lemma}\label{lem:genie_approximation}
    Under Assumption \ref{ass:asymptotics_binary}, the following holds for each $v \in \gV$
    \begin{align}
        \log \bigg( \frac{\P \big( \ervy_v = s|\{\tA^{(\ell)}\}_{\ell \in \sL}, \rvy_{-v} \big)}{\P \big( \ervy_v = -s|\{\tA^{(\ell)}\}_{\ell \in \sL}, \rvy_{-v} \big)} \bigg) = \rW_{v}\,\,, \quad s\in \{\pm 1\},\notag
    \end{align}
    where the variable $\rW_{v}$ is defined in \eqref{eqn:Wv}.
\end{lemma}

\begin{lemma}[Fundamental limit via genie-aided approach, {\cite[Lemma F.3]{Abbe2022LPT}}]\label{lem:fundamental_limit_genie-aided}
Let $\sX$ be a Borel space and $(\rvy, \rmX)$ be a random element in $\{\pm 1\}^{|\gV|} \times \sX$. Let $\sF$ be a family of Borel mappings from $\sX$ to $\{\pm 1\}^{|\gV|}$. Recall the mismatch ratio $\eta_{N}$ defined in \eqref{eqn:misratio_binary} for any two $\rvy, \widehat{\rvy} \in \{\pm 1\}^{|\gV|}$. For each $v\in [N]$, let $\P \big( \ervy_v = s |\rmX, \rvy_{-v} \big)$ with $s\in \{\pm 1\}$ denote the \emph{posterior} probability conditioned on $\rmX \in \sX$ and $\rvy_{-v} \in \{\pm 1\}^{N-1}$. Then for any estimator $\widehat{\rvy} \in \sF$, the following holds
    \begin{align}
        \E \eta_{N}(\widehat{\rvy}, \rvy) \geq \frac{N-1}{3N-1} \cdot \frac{1}{N}\sum_{v\in [N]}\P\Big( \frac{\P(\ervy_v = s|\rmX, \rvy_{-v})}{\P(\ervy_v = -s|\rmX, \rvy_{-v})} < 1\Big).\notag
    \end{align}
\end{lemma}

\begin{proof}[Proof of \Cref{thm:IT_lower_bound_binary}]
    In our setting, we apply Lemma \ref{lem:fundamental_limit_genie-aided} with $\rmX = \{\tA^{(\ell)}\}_{\ell \in \sL}$. By combining this with Lemma \ref{lem:genie_approximation} and leveraging the symmetry among vertices, we obtain the following lower bound for any estimator $\widehat{\rvy}$:
    \begin{align}
        \E \eta_{N}(\widehat{\rvy}, \rvy) \geq &\, \frac{N-1}{3N-1} \cdot \frac{1}{N}\sum_{v\in [N]}\P\Big( \frac{\P(\ervy_v = s|\{\tA^{(\ell)}\}_{\ell \in \sL}, \rvy_{-v})}{\P(\ervy_v = -s|\{\tA^{(\ell)}\}_{\ell \in \sL}, \rvy_{-v})} < 1\Big)
        =\frac{N-1}{3N-1} \cdot \P( \rW_{v} < 0). \notag
    \end{align}
We now consider the probability of the event $\{\rW_{v}\leq -\zeta_{N} q_{N}\}$, where $\zeta_{N} = 1/\log(q_{N})$. Following the same LDP analysis as in Lemma \ref{lem:binom_difference_prob} and Lemma \ref{lem:LDP_binomial}, we have
    \begin{align}
      \liminf_{N\to \infty} q_{N}^{-1} \log \big( \E \eta_{N}(\widehat{\rvy}, \rvy) \big) \geq &\, \liminf_{N\to \infty} q_{N}^{-1} \log\big( \P(\rW_{v} < 0) \big) \notag\\
      \geq &\, \liminf_{N\to \infty} q_{N}^{-1} \log\big(\P(\rW_{v} < -\zeta_{N} q_{N}) \big) = -\D_{\rm{GH}},\notag
    \end{align}
    which completes the proof.
\end{proof}

\section{Performance of the one-stage spectral method}\label{sec:achievability_matrices_binary}
This section is devoted to the proof outlines of Theorems \ref{thm:achievability_matrices} and \ref{thm:optimality_matrices} with details deferred to \Cref{app:achievability_matrices_binary}.
\subsection{Preliminary results}
For each $\ell \in \sL$, let $\rmA^{(\ell)}$ denote the adjacency matrix obtained from tensor $\tA^{(\ell)}$. For $u, v \in \gV$, the expected number of $\ell$-edges containing $u$ and $v$ can be expressed as
\begin{itemize}
    \item $\ervy_{u} = \ervy_{v}$: $\E \ermA_{uv}^{(\ell)} = \alpha_{\ell} \binom{N/2-2}{\ell - 2} + \beta_{\ell} [ \binom{N-2}{\ell - 2} - \binom{N/2-2}{\ell - 2}]$.
    \item $\ervy_{u} \neq \ervy_{v}$: $\E \ermA_{uv}^{(\ell)} = \beta_{\ell} \binom{N-2}{\ell - 2}$.
\end{itemize}
For convenience, we denote
\begin{align}
    \alpha = &\, \sum_{\ell \in \sL} \Bigg[\alpha_{\ell} \binom{N/2-2}{\ell - 2} + \beta_{\ell} \Big[ \binom{N-2}{\ell - 2} - \binom{N/2-2}{\ell - 2}\Big] \Bigg], \quad \beta = \sum_{\ell \in \sL} \beta_{\ell} \binom{N-2}{\ell - 2}.\label{eqn:alpha_beta_def}
\end{align}
Let $\rmA^{\star}$ denote the \emph{expected adjacency matrix}, which can be expressed as
        \begin{align}
            \rmA^{\star} \coloneqq &\, \E[\rmA] = \sum_{\ell \in \sL}\E[\rmA^{(\ell)}] = \begin{bmatrix}
            \alpha &\, \beta \\
            \beta &\, \alpha
            \end{bmatrix} \otimes (\ones_{N/2} \ones_{N/2}^{\sT}) - \alpha \rmI_{N},\label{eqn:EA_binary}
        \end{align}
Let $\rvu_2$ and $\rvu_{2}^{\star}$ denote the second eigenpair of $\rmA$ and $\rmA^{\star}$, respectively. For the proof of \Cref{thm:achievability_matrices}, we want to show that for each $v \in \gV$, $\sqrt{N} \ervy_v \ervu_{2,v} >\epsilon$ for some absolute constant $\epsilon >0$. Similar ideas appeared in \cite{Eldridge2018UnperturbedSA, Abbe2020EntrywiseEA, Gaudio2023CommunityDI}. 
Note that $\rvu_{2}^{\star} = N^{-1/2}\rvy$ and $\sign(\ervu^{\star}_{2, v}) = \sign(\ervy_v)$ for each $v\in \gV$.
However, $\rvu_{2}^{\star}$ is not the ideal candidate for direct comparison, due to the possibility that $\sqrt{N}\|\rvu_2 - \rvu^{\star}_{2}\|_{\infty} > 1$ as shown in {\cite[Theorem 3.3]{Abbe2020EntrywiseEA}}. Instead, Proposition \ref{prop:entrywise_diff_A} shows that the ideal pivot vector is $\rmA\rvu_{2}^{\star}/\lambda_{2}^{\star}$, which is a key step in the proof of \Cref{thm:achievability_matrices}.

\begin{proposition}\label{prop:entrywise_diff_A}
Let $(\lambda_2, \rvu_2)$ (resp. $(\lambda_{2}^{\star}, \rvu_{2}^{\star})$) be the second eigenpair of $\rmA$ (resp $\rmA^{\star}$). For some absolute constant $\const >0$, the following holds with probability at least $1 - O(N^{-2})$
\begin{align}
    \sqrt{N}\cdot \min_{s \in \{\pm 1\}} \|s\rvu_2 - \rmA \rvu_{2}^{\star}/\lambda_{2}^{\star}\|_{\infty} \leq \const/\log(\rho_{N}). \label{eqn:entrywise_diff_A}
\end{align}
Here, $\rho_{N}$ is the \emph{maximum} expected degree of the model \ref{def:non_uniform_HSBM_binary}, formally, 
\begin{align}\label{eqn:rho_binary}
    \rho_{N} \coloneqq \max_{v \in \gV} \E[\rD_{v}]. 
\end{align}
where $\rD_{v}\coloneqq \sum_{j\in \gV, j \neq v} \ermA_{vj}$ denotes the degree of vertex $v\in \gV$.
\end{proposition}

Lemma \ref{lem:LDP_AM} presents a \emph{Large Deviation Principle} (LDP) analysis of the $\sqrt{N} \ervy_{v} (\rmA\rvu_{2}^{\star})_v$, where $\D_{\rm{AM}}$ in \eqref{eqn:DAM_binary} serves as the rate function. The LDP analysis is crucial for proving the effectiveness of the spectral method when only the adjacency matrix $\rmA$ is available.
\begin{lemma}\label{lem:LDP_AM}
Define the following function
    \begin{align}
        \D_{\rm{AM}}(t) \coloneqq &\, \sum\limits_{\ell \in \sL} \frac{1}{2^{\ell - 1}} \Big[ a_{\ell}(1 - e^{- (\ell - 1)t}) + b_{\ell} \sum_{r=1}^{\ell - 1} \binom{\ell - 1}{r} \big( 1 - e^{- (\ell - 1 - 2r)t} \big) \Big] \label{eqn:DAMt_binary}
    \end{align}
Suppose that $\epsilon$ satisfies 
\begin{align}
\epsilon < \sum_{\ell\in \sL} 2^{-\ell + 1} \big( (\ell - 1) a_{\ell} + \sum_{r=1}^{\ell -1 } (\ell - 1 - 2r) b_{\ell} \big)\,. \notag
\end{align}
For any $\delta >0$, there exists some $N_{0} >0$ such that for all $N \geq N_{0}$ and $\epsilon_{N} < \epsilon$, the following holds
\begin{align}
    \P\Big(\sqrt{N}\,\, \ervy_{v} (\rmA\rvu_{2}^{\star})_v \leq \epsilon_{N} \, q_{N} \Big) \leq \exp\Big( - q_{N}\big[ \sup_{t\geq 0}\big(- \epsilon t + \D_{\rm{AM}}(t) \big) - \delta \, \big] \Big).\label{eqn:prob_yvAuv}
\end{align}
\end{lemma}
\begin{proof}[Proof of Lemma \ref{lem:LDP_AM}]
For each $v\in \gV$, recall $\ermA_{vj}$ in \eqref{eqn:adjacency_matrix_entry_binary}, then
    \begin{align}
    \sqrt{N} \ervy_{v} (\rmA\rvu_{2}^{\star})_v = &\, \ervy_{v}\sum\limits_{j\neq v, \,\, \ervy_{j} = +1 } \ermA_{vj} - \ervy_{v}\sum\limits_{j\neq v,\,\, \ervy_{j} = -1} \ermA_{vj} = \sum\limits_{j\neq v, \,\, \ervy_{j} = \ervy_{v} } \ermA_{vj} - \sum\limits_{j\neq v,\,\, \ervy_{j} \neq \ervy_{v}} \ermA_{vj} \notag \\
        =&\, \sum\limits_{j\neq v, \,\, \ervy_{j} = \ervy_{v} } \sum_{\ell \in \sL}\,\,\,\,\sum_{e\in \gE_{\ell},\,\, e \supset \{v, j\} } \,\,\,\, \etA^{(\ell)}_{e} - \sum\limits_{j\neq v,\,\, \ervy_{j} \neq \ervy_{v}} \sum_{\ell \in \sL}\,\,\,\,\sum_{e\in \gE_{\ell},\,\, e \supset \{v, j\} } \,\,\,\, \etA^{(\ell)}_{e} \notag\\
        =&\, \sum\limits_{\ell \in \sL} \sum\limits_{r=0}^{\ell - 1}\sum\limits_{e\in \gE^{(r)}_{\ell}(v)} \Big( (\ell - 1 - r)\cdot \etA^{(\ell)}_{e} - r \cdot \etA^{(\ell)}_{e} \Big) \notag\\
        =&\, \sum\limits_{\ell \in \sL} \Bigg( \sum\limits_{e\in \gE^{(0)}_{\ell}(v)} (\ell - 1)\cdot \etA^{(\ell)}_{e} + \sum\limits_{r=1}^{\ell - 1} \sum\limits_{e\in \gE^{(r)}_{\ell}(v)} (\ell - 1 - 2r)\cdot \etA^{(\ell)}_{e} \Bigg)\,,\notag\\
        =&\, \sum\limits_{\ell \in \sL} \Bigg( (\ell - 1)\cdot \sum\limits_{e\in \gE^{(0)}_{\ell}(v)} \etA^{(\ell)}_{e} + \sum\limits_{r=1}^{\ell - 1} (\ell - 1 - 2r)\cdot \sum\limits_{e\in \gE^{(r)}_{\ell}(v)} \etA^{(\ell)}_{e} \Bigg)\,,\label{eqn:Au_decomposition}
    \end{align}
where the third to last equality follows from the edge classification in \eqref{eqn:edge_classification}, since for each edge $e\in \gE^{(r)}_{\ell}(v)$, it contains $r$ vertices having opposite signs as $v$ with each contributing $-1$, and $\ell - 1 - r$ vertices having the same label with each contributing $+1$.

We then apply Lemma \ref{lem:LDP_binomial} to control the tail probability of \eqref{eqn:Au_decomposition}. We choose $g_{N} = 1$, $\rho_{i}^{(\ell)} = 2^{-\ell + 1}$ and values of $\alpha_{i}^{(\ell)}$, $c_{i}^{(\ell)}$ as in \eqref{eqn:Au_decomposition}. The proof is then completed by substituting these parameters into Lemma \ref{lem:LDP_binomial}, where two conditions can be verified straightforwardly under the Assumption.
\end{proof}

\subsection{Proof of \Cref{thm:achievability_matrices}}
Define the indicator function $\rF_{N} = \indi{\eqref{eqn:entrywise_diff_A} \textnormal{ holds}}$, and denote $\epsilon_{N} = C/\log(\rho_{N})$ with $C$ in \eqref{eqn:entrywise_diff_A}. Let $\hat{s} = \argmin_{s\in \{\pm 1\}}\|s\rvu_2 - \rmA \rvu_{2}^{\star}/\lambda_{2}^{\star}\|_{\infty}$, then
\begin{align}
    &\, \P( \eta_{N}(\rvy, \sign(\hat{s} \,\rvu_{2})) = 0) = \P( \sign(\hat{s} \rvu_{2}) = \rvy) \notag\\
    \geq &\, \P(\sqrt{N} \min_{v \in \gV} (\hat{s}\,\ervu_{2, v} \cdot \ervy_{v} ) \geq \epsilon_{N} \mid \rF_{N} = 1) \cdot \P(\rF_{N} = 1) \notag\\
    \geq &\, \P(\sqrt{N} \min_{v \in \gV} (\hat{s}\,\ervu_{2, v} \cdot \ervy_{v} ) \geq \epsilon_{N}) - \P(\rF_{N} = 0) \notag\\
    \geq &\, \P(\sqrt{N} \min_{v\in \gV}\ervy_{v}(\rmA\rvu_{2}^{\star})_{v}/\lambda_2^{\star} - \epsilon_{N} \geq \epsilon_{N}) - \P(\rF_{N} = 0) \notag\\
    \geq &\, 1- \sum_{v\in \gV}\P \big( \sqrt{N}\ervy_{v}(\rmA\rvu_{2}^{\star})_{v} \leq 2\lambda_{2}^{\star} \epsilon_{N}\big) - \P(\rF_{N} = 0) \notag
\end{align}
We take $q_{N} = \log(N)$. According to Proposition \ref{prop:entrywise_diff_A}, $\P(\rF_{N} = 0) \lesssim O(N^{-2})$. Note that $\lambda_{2}^{\star} = \frac{N}{2}(\alpha - \beta) - \alpha \asymp q_{N}$ due to \eqref{eqn:EA_binary} and \eqref{eqn:alpha_beta_def}. Then by Lemma \ref{lem:LDP_AM}, for any $\delta >0$, there exists some $N_{0}> 0$ such that the following holds for all $N \geq N_{0}$ and $2\lambda_{2}^{\star} \epsilon_{N}/q_{N} < \epsilon$ 
\begin{align}
    \P( \eta_{N}(\rvy, \sign(\hat{s} \,\rvu_{2})) = 0) \geq 1 - N^{1 - \sup_{t\geq 0} [ -\epsilon t + \D_{\rm{AM}}(t) ] + \delta} - N^{-2}\,,
\end{align}
where $\D_{\rm{AM}}(t)$ is defined in \eqref{eqn:DAMt_binary}. Note that $\D_{\rm{AM}} = \sup_{t\geq 0} \D_{\rm{AM}}(t) > 1$ due to \eqref{eqn:exact_recovery_condition_matrices}, the proof is then completed by choosing sufficiently small $\delta$.

\subsection{Proof of \Cref{thm:optimality_matrices}}
Let $\hat{s} = \argmin_{s\in \{\pm 1\}}\|s\rvu_2 - \rmA \rvu_{2}^{\star}/\lambda_{2}^{\star}\|_{\infty}$. By definition,
    \begin{align}
        \E \eta_{N} (\rvy, \sign(\hat{s} \,\rvu_{2})) \leq &\, \frac{1}{N}\sum\limits_{v\in \gV} \P(\ervy_{v} \cdot \hat{s} \ervu_{2, v} < 0) \notag
    \end{align}
    Similarly, define $\rF_{N} = \indi{\eqref{eqn:entrywise_diff_A} \textnormal{ holds}}$ and denote $\epsilon_{N} = C/\log(\rho_{N})$. According to Proposition \ref{prop:entrywise_diff_A}, we have
    \begin{align}
        \P(\ervy_{v} \cdot \hat{s} \ervu_{2, v} < 0) = &\,  \P(\ervy_{v} \cdot \hat{s} \ervu_{2, v} < 0 \mid \rF_{N} = 1) \cdot \P(\rF_{N} = 1) +  \P(\ervy_{v} \cdot \hat{s} \ervu_{2, v} < 0 \mid \rF_{N} = 0)\cdot \P(\rF_{N} = 0) \notag\\
        \leq &\, \P\big( \sqrt{N}\ervy_{v}(\rmA\rvu_{2}^{\star})_{v} \leq 2\lambda_2^{\star}\epsilon_{N}\big)  +  O(N^{-2}) \notag
    \end{align}
Again by Lemma \ref{lem:LDP_AM}, for any $\delta > 0$, there exists some large $N_{0}> 0$ such that the following holds for all $N \geq N_{0}$ and $2\lambda_{2}^{\star} \epsilon_{N}/q_{N} < \epsilon$,
\begin{align}
   \P\big( \sqrt{N}\ervy_{v}(\rmA\rvu_{2}^{\star})_{v} \leq 2\lambda_2^{\star}\epsilon_{N}\big)  \leq N^{- \sup_{t\geq 0} [ -\epsilon t + \D_{\rm{AM}}(t) ] + \delta}\,,\notag
\end{align}
where $\D_{\rm{AM}}(t)$ is defined in \eqref{eqn:DAMt_binary}. Since all the nodes are exchangeable, we have
\begin{align}
    \E \eta_{N} (\rvy, \sign(\hat{s} \,\rvu_{2})) \leq N^{-\sup_{t\geq 0} [ -\epsilon t + \D_{\rm{AM}}(t) ] + \delta} + O(N^{-2})\,.\notag
\end{align}
By letting $\epsilon, \delta$ go to zero, the desired result follows since $\D_{\rm{AM}} = \sup_{t\geq 0} \D_{\rm{AM}}(t)$.

\section{Performance of the refinement algorithms}\label{sec:refinement_binary}
This section outlines the proofs for Theorems \ref{thm:achieveability_binary} and \ref{thm:optimal_recovery_binary}, which concern the performance of the refinement algorithms. The detailed proofs of the supporting lemmas are deferred to \Cref{app:refinement_binary}.

\subsection{Refinement via power iteration}
Recall that the construction of $\widehat{\mA}$ in \eqref{eqn:hatA}, which serves as the input for Algorithm \ref{alg:refinement_spectral}, relies on the sets $\widehat{\gE}^{(r)}_{\ell}(v)$ defined in \eqref{eqn:hat_edge_classification}. These sets approximate the true edge classifications $\gE^{(r)}_{\ell}(v)$ in \eqref{eqn:edge_classification} for each $v\in \gV$ and $0\leq r \leq \ell - 1$, using the initial estimate $\widehat{\vy}^{(0)}$. To facilitate our analysis, we introduce an \emph{ideal} version of $\widehat{\mA}$ below, constructed using the true membership vector $\vy$ and the actual edge sets $\{\gE^{(r)}_{\ell}(v)\}_{v}$. For each edge $e\in \gE^{(r)}_{\ell}(v)$, we consider the vertices in $e \setminus \{v\}$. For $j\in e\setminus \{v\}$, we define the entry $\overline{\etA}^{(\ell)}_{e, v, j} \coloneqq \etA^{(\ell)}_{e}/(\ell - 1)$ when $r \in \{0, \ell - 1\}$; when $1\leq r \leq \ell - 2$, we let
\begin{align}
    \overline{\etA}^{(\ell)}_{e, v, j} =&\, \etA^{(\ell)}_{e}/r,\,\, \textnormal{if }\ervy_v = \ervy_j; \quad \overline{\etA}^{(\ell)}_{e, v, j} = \etA^{(\ell)}_{e}/(\ell - 1-r), \,\,\textnormal{if }\ervy_v \neq \ervy_j. \label{eqn:barAevj}
\end{align} 
For each $\ell \in \sL$, we define the \emph{ideal} version of $\widehat{\psi}_{\ell}$ in \eqref{eqn:psi_hat} as
\begin{align}
    \psi_{\ell} = \log \Big( \frac{\alpha_{\ell}(1 - \beta_{\ell}) }{\beta_{\ell}(1 - \alpha_{\ell})} \Big). \label{eqn:psi_l}
\end{align}
Then we define the \emph{ideal} weighted adjacency matrix $\overline{\rmA}$ by
\begin{align}
    \overline{\ermA}_{vj} \coloneqq \indi{v \neq j} \cdot \sum_{\ell \in \sL} \psi_{\ell} \sum_{r=0}^{\ell - 1} \sum_{e \in \gE^{(r)}_{\ell}(v) } \overline{\etA}^{(\ell)}_{e, v, j}, \quad v, j\in \gV.\label{eqn:barA}
\end{align}

We present two lemmas that are essential in proving \Cref{thm:achieveability_binary} (1). The first lemma controls the approximation error between $\widehat{\rmA}\widehat{\rvy}^{(0)}$ and $\overline{\rmA}\rvy$ in $\ell_{q_{N}}$-norm, while the second lemma is a large deviation result connecting the $\ell_{p}$-norm error to the mismatch ratio.
\begin{lemma}\label{lem:subcritical-Ay-error}
    With probability at least $1 - O(e^{- \frac{\D_{\mathrm{AM}}}{2}q_{N}})$, the following holds,
\begin{align}
  N^{-1/2} \min_{s \in \{\pm 1\}} \|s\widehat{\rmA}\widehat{\rvy}^{(0)} - \overline{\rmA}\rvy\|_{q_{N}}\lesssim N^{1/q_{N}-1/2} \sqrt{q_{N}}.\notag
\end{align}
\end{lemma}
    
\begin{lemma}[Lemma E.1 in \cite{Abbe2022LPT}]\label{lem:LDP-v-error}
Let $\vv$, $\vw$ and $\overline{\vv}$ be three random vectors in $\R^{N}$. Let $\delta_{N} \coloneqq \min_{i\in [N]} |\overline{\evv}_{i}| > 0$ denote the minimum entrywise magnitude of $\overline{\vv}$, and $\eta_{N} \coloneqq \eta_{N}(\sign(\vv), \sign(\overline{\vv}))$ in \eqref{eqn:misratio_binary} denote the mismatch ratio between two vectors $\sign(\vv), \sign(\overline{\vv})\in \{ \pm 1\}^{N}$. For any $\epsilon \in (0, 1)$, define the indicator variable
\begin{align}
    \rT_{j}(\epsilon) \coloneqq \mathds{1}\{ -\evw_{j}\cdot \sign(\overline{\evv}_{j}) \geq (1 - \epsilon)\cdot |\overline{\evv}_{j}|\}\,, \notag
\end{align}
for each $j\in [N]$. Let $p \coloneqq p_{N} \to \infty$. If $\min_{s=\pm 1} \|s\vv - \vw - \overline{\vv}\|_{p} \ll N^{1/p} \delta_{N}$ with probability at least $1 - C e^{-p}$ for some $C >0$, then
\begin{align}
    \limsup_{N\to \infty} p^{-1}\log \E \eta_{N}(\sign(\vv), \sign(\overline{\vv})) \leq \limsup_{\epsilon \to 0}\, \limsup_{N\to \infty} p^{-1} \log \bigg(\frac{1}{N}\sum_{j=1}^{N} \P \big( \rT_{j}(\epsilon) = 1 \big)\bigg).\notag
\end{align}
\end{lemma}

\begin{proof}[Proof Theorem \ref{thm:optimal_recovery_binary} (1)]
Recall $\widehat{\vy} = \sign(\widehat{\mA}\widehat{\vy}^{(0)})$ from \Cref{alg:refinement_spectral}. By taking $\vv = N^{-1/2} \widehat{\mA}\widehat{\vy}^{(0)}$, $\vw = N^{-1/2}\overline{\mA}\vy - q_{N}N^{-1/2}\vy$, $\overline{\vv} = N^{-1/2}\vy$ and $p = q_{N}$, Lemmas \ref{lem:subcritical-Ay-error} and \ref{lem:LDP-v-error} yield that
\begin{align}
    &\, \limsup_{N\to \infty} q_{N}^{-1}\log \E \eta_{N}(\widehat{\rvy}, \rvy) \notag\\
    \leq &\, \limsup_{\epsilon \to 0}\, \limsup_{N\to \infty} q_{N}^{-1} \log \bigg(\frac{1}{N}\sum_{v=1}^{N} \P \big( -([\overline{\mA}\vy]_{v} - q_{N} N^{-1/2}\evy_{v})\evy_{v} \geq (1 - \epsilon) q_{N} N^{-1/2} \big)\bigg)\,,\notag\\
    \leq &\, \limsup_{\epsilon \to 0}\, \limsup_{N\to \infty} q_{N}^{-1} \log \Big(\P \big( \evy_{v}[\overline{\mA}\vy]_{v} \leq \epsilon q_{N} N^{-1/2} \big)\Big).\notag
\end{align}
We now evaluate the tail probability for $\evy_{v}[\overline{\mA}\vy]_{v}$. Note that the $\ell$-hyperedge $e$ may contribute different weights when summing over $v$-th row of $\overline{\rmA}$, since its contribution depends on the position where the hyperedge $e$ shows up in $\ervy_v [\overline{\rmA}\rvy]_v$. Then we have
\begin{align}
    &\, \ervy_v [\overline{\rmA}\rvy]_v = \ervy_v\sum_{j\neq v, \,\, \ervy_j = +1 } \overline{\ermA}_{vj} - \ervy_v\sum_{j\neq v,\,\, \ervy_j = -1} \overline{\ermA}_{vj} =\sum_{j\neq v, \,\, \ervy_j = \ervy_v } \overline{\ermA}_{vj} - \sum_{j\neq v,\,\, \ervy_j \neq \ervy_v} \overline{\ermA}_{vj} \notag\\
        =&\, \sum_{\ell \in \sL} \psi_{\ell} \cdot \Big(\sum_{e\in \gE^{(0)}_{\ell}(v)} (\ell - 1)\cdot \frac{\etA^{(\ell)}_e}{\ell - 1} - \sum_{e\in \gE^{(\ell - 1)}_{\ell}(v)} (\ell - 1)\cdot \frac{\etA^{(\ell)}_e}{\ell - 1} \Big)\notag\\
        &\,\quad + \sum_{\ell \in \sL} \psi_{\ell} \cdot \sum_{r=1}^{\ell - 2}\sum_{e\in \gE^{(r)}_{\ell}} \Big( r\cdot \frac{\etA^{(\ell)}_{e}}{r} - (\ell - 1-r)\cdot\frac{\etA^{(\ell)}_{e}}{\ell - 1-r} \Big)\label{eqn:RVdiff}\\
        = &\, \sum_{\ell \in \sL} \psi_{\ell}\cdot \Big( \sum_{ e\in \gE^{(0)}_{\ell}(v) } \etA^{(\ell)}_{e} - \sum_{ e\in \gE^{(\ell - 1)}_{\ell}(v) } \etA^{(\ell)}_{e} \Big) = \rW_{v} \textnormal{ in }\eqref{eqn:Wv},\notag
\end{align}
where the redundant edges in \eqref{eqn:Au_decomposition} are cancelled out in \eqref{eqn:RVdiff} due to the construction of $\overline{\mA}$, while the remaining term coincides with $\rW_{v}$ in \eqref{eqn:Wv}. By applying Lemma \ref{lem:LDP_binomial}, an analysis similar to Lemma \ref{lem:binom_difference_prob} yields
\begin{align}
    \limsup_{N\to \infty} q_{N}^{-1}\log \Big(\P \big( \evy_{v}[\overline{\mA}\vy]_{v} \leq \epsilon q_{N} N^{-1/2} \big)\Big) \leq - \D_{\rm{GH}}\,.\notag
\end{align}
This completes the proof of Theorem \ref{thm:optimal_recovery_binary} (1). Then, it immediately follows that
\begin{align}
    \limsup_{N\to \infty} q_{N}^{-1}\log \E \eta_{N}(\widehat{\rvy}, \rvy) \leq - \D_{\rm{GH}}.\notag
\end{align}
This completes the proof.
\end{proof}
\begin{proof}[Proof Theorem \ref{thm:achieveability_binary} (1)]
By Markov's inequality, when \eqref{eqn:achievabilityGH} holds, we have
\begin{align}
    \P(\eta_{N}(\widehat{\rvy}, \rvy) \neq 0) = \P\big( \eta_{N}(\widehat{\rvy}, \rvy) \leq N^{-1}\big) \leq N \cdot \E \eta_{N}(\widehat{\rvy}, \rvy) = \exp\big(\log(N) - \D_{\rm{GH}}\cdot q_{N} \big) \leq N^{-\epsilon}.\notag
\end{align}
Therefore, \Cref{alg:refinement_spectral} achieves strong consistency with probability at least $1 - N^{-\epsilon}$.
\end{proof}

\subsection{Refinement via majority voting} 
The key idea is to leverage the results in \cite{Dumitriu2023OptimalAE}, which analyzes a similar majority voting refinement algorithm for non-uniform hypergraph SBM with $K\geq 2$ communities. Note that the model in \cite{Dumitriu2023OptimalAE} allows unbalanced community sizes, while our model \ref{def:non_uniform_HSBM_binary} assumes two equal-sized communities. The main difference lies in the definition of the divergence function.
\begin{proof}[Proof of Theorem \ref{thm:achieveability_binary} (2)]
According to Section 4 in \cite{Dumitriu2023OptimalAE}, the entire vertex set $\gV$ can be split into two disjoint subsets $\setS$ and $\gV\setminus \setS$. The vertices in $\setS$ enjoy three good properties, such that they can be corrected to their own communities in each step. After $\log(N)$ many iterations, there will be no wrongly classified vertex in $\setS$. As for the other subset $\gV\setminus \setS$, it is proved that $|\gV\setminus \setS| < 1$ with probability tending to $1$ when \eqref{eqn:achievabilityGH} holds, thus completing the proof. 
\end{proof}
\begin{proof}[Proof of Theorem \ref{thm:optimal_recovery_binary} (2)]
Suppose \eqref{eqn:optimality_condition}. The vertices in $\setS$ can be corrected by Algorithm \ref{alg:refinement_voting_binary} as well. However, $|\gV\setminus \setS|\leq N^{1 - \kappa_{N}}$ under \eqref{eqn:optimality_condition}, which is not as good as before. The proof is then completed by considering the worst scenario where all nodes in $\gV \setminus \setS$ were wrongly classified.
\end{proof}

\section{Concluding remarks}\label{sec:conclusion_binary}
We conclude this paper with an insightful observation. In deriving the optimal estimator \eqref{eqn:genie-aided} and its transformation \eqref{eqn:Wv}, only two specific types of edges, $\gE^{(0)}_{\ell}(v)$ and $\gE^{(\ell - 1)}_{\ell}(v)$ defined in \eqref{eqn:edge_classification}, contribute to the calculation of the rate function $\D_{\mathrm{GH}}$ for the LDP analysis, even though there are $\ell - 2$ other types of edges that cross two communities. When $\D_{\mathrm{GH}} > 1 > \D_{\rm{AM}}$, Algorithm \ref{alg:spectral_partition_binary} is hampered by the redundant edges in \eqref{eqn:Au_decomposition}, which prevent the one-stage spectral method from achieving strong consistency, as well as the optimal mismatch ratio. However, these redundant edges are cancelled out in \eqref{eqn:RVdiff} by constructing the \emph{ideal} weighted adjacency matrix $\overline{\rmA}$ in \eqref{eqn:barA}, which inspired the design of the \Cref{alg:refinement_spectral}. This cancellation effect is crucial for the refinement algorithms to achieve strong consistency and the optimal mismatch ratio, suggesting that in hypergraph community detection, certain types of edges may carry more informative signals for clustering, while others may introduce noise that hinders performance.

%% file: appendix.tex
\section{Deferred proofs in \Cref{sec:information_theoretic_limits_binary}}\label{app:impossibility_binary}
\begin{proof}[Proof of Lemma ~\ref{lem:MLE_maximize}]
Recall $[\gV]^{\ell} \coloneqq \{\setS\subseteq \gV:|\setS| = \ell \}$. Consider the following log-likelihood
\begin{align}
    &\, \log \, \P (\mathbb{H} = \gH \mid \rvy = \rvz)\notag\\
  = &\,\log \bigg( \prod_{\ell \in \sL} \, \prod_{e\in [\gV]^{\ell}} \big[\P(e\in \gE_{\ell}| \rvy = \rvz) \big]^{\etA^{(\ell)}_{e}} \cdot \big[\P(e\notin \gE_{\ell}| \rvy = \rvz) \big]^{1 - \etA^{(\ell)}_{e}} \bigg)\notag\\
  = &\, \sum\limits_{\ell \in \sL}\bigg( \sum\limits_{e\in [\gV]^{\ell} } \indi{e \textnormal{ is in-cluster}}\cdot \big[\etA^{(\ell)}_{e} \, \log \alpha_{\ell} + (1 - \etA^{(\ell)}_{e})\, \log (1 - \alpha_{\ell}) \big] \notag \\
  &\, \quad \quad \quad \quad +  \sum\limits_{e\in [\gV]^{\ell}} (1 - \indi{e \textnormal{ is in-cluster}})\cdot  \big[\etA^{(\ell)}_{e} \, \log \beta_{\ell} + (1 - \etA^{(\ell)}_{e})\, \log (1 - \beta_{\ell}) \big] \bigg) \notag \\
  = &\, \sum\limits_{\ell \in \sL} \bigg(\log \bigg(\frac{\alpha_{\ell}(1 - \beta_{\ell})}{\beta_{\ell}(1 - \alpha_{\ell})}\bigg) \cdot \sum\limits_{e\in [\gV]^{\ell} } \indi{e \textnormal{ is in-cluster}}\cdot \etA^{(\ell)}_{e} +  \const_{\ell} \bigg),\notag
\end{align}
where $\const_{\ell} >0$ is some constant irrelevant to the membership assignment $\rvy$. Note that the logarithm map preserves the monotocity and $\const_{\ell}$ can be factored out, then MLE in \eqref{eqn:MLE} can be further simplified as
\begin{align}
    \widehat{\rvy}_{\mathrm{MLE}} = &\,\underset{\rvz \in \{\pm 1\}^{|\gV|}, \, \ones^{\sT}\rvz = 0}{\arg\max} \, \sum\limits_{\ell \in \sL} \bigg(\log \bigg(\frac{\alpha_{\ell}(1 - \beta_{\ell})}{\beta_{\ell}(1 - \alpha_{\ell})}\bigg) \cdot \sum\limits_{e\in [\gV]^{\ell} } \indi{e \textnormal{ is in-cluster}}\cdot \etA^{(\ell)}_{e} \bigg) \notag\\
    = &\, \underset{\rvz \in \{\pm 1\}^{|\gV|}, \, \ones^{\sT}\rvz = 0}{\arg\max}\,\, f(\rvz|\gH).\notag
\end{align}
Therefore, the desired argument is proved since each $\ell \in \sL$, the following holds
\begin{align}
    \log \Big( \frac{\alpha_{\ell}(1 - \beta_{\ell})}{\beta_{\ell}(1 - \alpha_{\ell})}\Big) \, \left\{
        \begin{aligned}
            &> 0\,, & \textnormal{if }\alpha_{\ell} > \beta_{\ell}, \\
            &< 0\,, & \textnormal{if } \alpha_{\ell} < \beta_{\ell}.
        \end{aligned}
    \right.\notag
\end{align}
\end{proof}

\begin{proof}[Proof of Lemma ~\ref{lem:independent_WaWbJ}]
First, for any $a\in \gV_{+} \cap \sU$, $b\in \gV_{-} \cap \sU$, the independence between $\rW_{a}[\gU]$ and $\rJ[\gU]$ follows directly. The resaon is that each hyperedge in $\rJ[\gU]$ contains at least two vertices in $\gU$, while hyperedge in $\rW_{a}$ contains exactly one vertex $a\in \gU$, thus no common edges. The independence between $\rW_{b}[\gU]$ and $\rJ[\gU]$ follows similarly. The independence between $\rW_{a}[\gU]$ and $\rW_{b}[\gU]$ is also clear since they do not share any common hyperedge by construction.
At the same time, we have
\begin{align}
    &\, \{\rW_{a}[\gU] \leq -\zeta_{N} q_{N}\} \cap \{\rJ[\gU] = 1\} \implies \{\rW_{a}[\gU] \leq -\zeta_{N} q_{N}\} \cap \{\rJ_{a} \leq \zeta_{N} q_{N} \},\notag\\
    &\, \{\rW_{b}[\gU] \leq -\zeta_{N} q_{N}\} \cap \{\rJ[\gU] = 1\} \implies \{\rW_{b}[\gU] \leq -\zeta_{N} q_{N}\} \cap \{\rJ_{b} \leq \zeta_{N} q_{N} \}.\notag
\end{align}
Thus it suffices to prove the following relationship
\begin{align}
    \{\rW_{a}[\gU] \leq -\zeta_{N} q_{N}\} \cap \{\rJ_{a} \leq \zeta_{N} q_{N} \} \implies \{\rW_{a} \leq 0\}.\notag
\end{align}
To this end, we expand $\rW_{a}$ and plug in the condition on the left hand side of the above equation, then 
\begin{align}
    \rW_{a} = &\, \sum_{\ell \in \sL} \log \Big( \frac{\alpha_{\ell}(1 - \beta_{\ell})}{\beta_{\ell}(1 - \alpha_{\ell})} \Big)\cdot \Big( \sum_{ e\in \gE^{(0)}_{\ell}(a) } \etA^{(\ell)}_{e} - \sum_{ e\in \gE^{(\ell - 1)}_{\ell}(a) } \etA^{(\ell)}_{e} \Big)\notag\\
            \leq &\, \rW_{a}[\gU] + \rJ_{a}\leq 0,\notag
\end{align}
which completes the proof.
\end{proof}

\begin{proof}[Proof of Lemma ~\ref{lem:Wab_J_prob}]
For $a\in \gV_{+}\cap \sU$, according to Lemma \ref{lem:binom_difference_prob}, $\P( \rW_{a}[\sU] \leq - \zeta_{N} \cdot q_{N}) = e^{-\D_{\rm{GH}}\cdot q_{N}}$. According to Lemma \ref{lem:independent_WaWbJ}, $\rW_{a}[\sU]$ and $\rW_{a^{\prime}}[\sU]$ are independent for distinct $a,a^{\prime}\in \gV_{+}\cap \sU$, thus
    \begin{align}
        &\,\prod_{a\in \gV_{+} \cap \sU} \P(\rW_{a}[\sU] > - \zeta_{N} \cdot q_{N}) = \prod_{a\in \gV_{+} \cap \sU} \Big(1 - \P(\rW_{a}[\sU] \leq - \zeta_{N} \cdot q_{N})\Big) \notag \\
        = &\, (1 - e^{-\D_{\rm{GH}}\cdot q_{N}})^{\gamma_{N} |\gV|/2} \leq \exp \Big(-\frac{\gamma_{N} \cdot |\gV|}{2} \cdot e^{-\D_{\rm{GH}}\cdot q_{N}} \Big)\,\,,\notag
    \end{align}
where we used the fact that $1 -x \leq e^{-x}$ for $x=o(1)$ in the last inequality. The upper bound for $\prod_{b\in \gV_{-} \cap \sU} \P(\rW_{b}[\sU] > - \zeta_{N} \cdot q_{N})$ can be obtained similarly. 
    
For the second part of the proof, through a union bound, the problem is transferred to seeking an upper bound on $\P(\rJ_{v}[\sU] > \zeta_{N} q_{N})$ for any $v\in \gU$, since
\begin{align*}
    \P(\rJ[\gU] = 0) = \P(\exists v\in \gU \textnormal{ s.t. } \indi{\rJ_{v}[\sU] \leq \zeta_{N} q_{N} } = 0 ) \leq \sum\limits_{v\in \sU} \P(\rJ_{v}[\sU] > \zeta_{N} q_{N})\,.
\end{align*}
Recall $\gW_{\ell}^{(\geq 2)}$ in \eqref{eqn:edgeWm}, where each edge $e\in \gW_{\ell}^{(\geq 2)}$ contains at least $2$ vertices in $\sU$. Thus for edge $e\in \gW_{\ell}^{(\geq 2)}$ with $e\ni v$, at least one more vertex in $e\setminus\{v\}$ is from $\sU$. We now calculate the expectation of $\rJ_{v}[\sU]$ defined in \eqref{eqn:rJv}. Note that the cardinality of the set $\gW_{\ell}^{(\geq 2)}$ is upper bounded by $\binom{N-1}{\ell - 1} - \binom{(1 - \gamma_{N})N}{\ell - 1}$, then
\begin{align*}
    \E \rJ_{v}[\sU] \coloneqq \sum\limits_{\ell \in \sL} \log \Big( \frac{\alpha_{\ell}(1 - \beta_{\ell})}{\beta_{\ell}(1 - \alpha_{\ell})} \Big)\cdot \bigg[ \binom{N-1}{\ell - 1} - \binom{(1 - \gamma_{N})N}{\ell - 1} \bigg] \cdot \E \etA^{(\ell)}_{e}.
\end{align*}
where $\E \etA^{(\ell)}_{e} = \alpha_{\ell}$ if $e$ is in-cluster, otherwise $\E \etA^{(\ell)}_{e} =\beta_{\ell}$. Since $\ell \in \sL$ is some fixed integer and $\ell \ll N$, we then have
\begin{align*}
    \binom{N-1}{\ell - 1}[1 - (1 - \gamma_{N})^{\ell - 1}] \leq  \binom{N-1}{\ell - 1} - \binom{(1 - \gamma_{N})N}{\ell - 1} \leq \binom{N-1}{\ell - 1}[1 - (1 - \gamma_{N})\cdot(1 - 2\gamma_{N})^{\ell - 2}].
\end{align*}
The expectation of $\rJ_{v}[\sU]$ is then bounded both from above and below, as follows
\begin{align*}
    \E \rJ_{v}[\sU] \geq &\, \sum\limits_{\ell \in \sL} \log \Big( \frac{\alpha_{\ell}(1 - \beta_{\ell})}{\beta_{\ell}(1 - \alpha_{\ell})} \Big)\cdot [1 - (1 - \gamma_{N})^{\ell - 1}] \cdot b_{\ell} \\
    \E \rJ_{v}[\sU] \leq &\, \sum\limits_{\ell \in \sL} \log \Big( \frac{\alpha_{\ell}(1 - \beta_{\ell})}{\beta_{\ell}(1 - \alpha_{\ell})} \Big)\cdot [1 - (1 - 2\gamma_{N})^{\ell - 1}] \cdot a_{\ell}\,.
\end{align*}
Note that $a_{\ell} \gtrsim b_{\ell} \gtrsim 1$ under Assumption \ref{ass:asymptotics_binary}, and $\gamma_{N} = o(1)$ by construction of $\sU$. Since $1 - x + O(x^{2})= e^{-x}$ for $x = o(1)$, we have $1 - (1 - \gamma_{N})^{\ell - 1} = (\ell - 1)\gamma_{N}$ and $1 - (1 - 2\gamma_{N})^{\ell - 1} = (\ell - 1)\gamma_{N}$. Consequently,  $\E \rJ_{v}[\sU] = \Theta(\gamma_{N} q_{N})$ for any $v \in \sU$, and $\rJ_{v}[\sU]$ concentrates around its expectation with high probability. According to Lemma \ref{lem:Chernoffvariant}, the following holds
\begin{align*}
    \P(\rJ_{v}[\sU] \geq \zeta_{N} q_{N}) \leq \exp\Big( -q_{N} \cdot \big[\zeta_{N} \log(\zeta_{N}/\gamma_{N}) - (\zeta_{N} - \gamma_{N}) \big]\Big).
\end{align*}
Recall that $|\sU| = \gamma_{N} |\gV|$, then by a union bound
    \begin{align*}
        \P(\rJ[\sU] = 0) \leq \sum\limits_{v\in \sU} \P(\rJ_{v}[\sU] \geq \zeta_{N} q_{N}) \leq \exp\Big( \log(\gamma_{N} N) - q_{N} \cdot \big[\zeta_{N} \log(\zeta_{N}/\gamma_{N}) - (\zeta_{N} - \gamma_{N}) \big]\Big).
    \end{align*}
which completes the proof.
\end{proof}

\begin{proof}[Proof of Lemma ~\ref{lem:genie_approximation}]
The likelihood for observing $\{\tA^{(\ell)}\}_{\ell \in \sL}$ conditioned on $\ervy_{v} = s \in \{\pm 1\}$ and $\rvy_{-v}$ is
    \begin{align*}
        &\,\P (\{\tA^{(\ell)}\}_{\ell \in \sL} \mid \ervy_{v} = s, \rvy_{-v}) \propto \prod_{\ell \in \sL} \, \prod_{e\ni v} \big[\P(e\in \gE_{\ell}| \ervy_{v} = s, \rvy_{-v}) \big]^{\etA^{(\ell)}_{e}} \cdot \big[\P(e\notin \gE_{\ell}| \ervy_{v} = s, \rvy_{-v}) \big]^{1 - \etA^{(\ell)}_{e}}.
    \end{align*}
where $\propto$ only hides factors that are irrelevant to $v$. Note that the label vector is sampled uniformly at random from all vectors $\rvz\in \{\pm 1\}^{|\gV|}$ satisfying $\ones^{\sT} \rvz = 0$. According to the Bayes Theorem and the analysis in \eqref{eqn:diff_MLE_function}, for $s\in \{\pm 1\}$, the log-odd ratio can be expressed as
    \begin{align*}
        &\,\log \bigg( \frac{\P \big( \ervy_v = s|\{\tA^{(\ell)}\}_{\ell \in \sL}, \rvy_{-v} \big)}{\P \big( \ervy_v = -s|\{\tA^{(\ell)}\}_{\ell \in \sL}, \rvy_{-v} \big)} \bigg) = \log \bigg( \frac{\P (\{\tA^{(\ell)}\}_{\ell \in \sL} \mid \ervy_{v} = 1, \rvy_{-v})}{\P (\{\tA^{(\ell)}\}_{\ell \in \sL} \mid \ervy_{v} = -1, \rvy_{-v})} \bigg) \\
        = &\,\sum_{\ell \in \sL} \log\Big( \frac{\alpha_{\ell}(1 - \beta_{\ell})}{\beta_{\ell}(1 - \alpha_{\ell})} \Big) \cdot\Big[  |\gE^{(0)}_{\ell}(v)| - |\gE^{(\ell - 1)}_{\ell}(v)|\Big]\\
        =&\, \rW_{v},
    \end{align*}
which completes the proof.
\end{proof}


\begin{lemma}[Generalized version of {\cite[Lemma 2]{Kim2018StochasticBM}}]\label{lem:binom_difference_prob}
    According to the Assumptions in Lemma \ref{lem:Wab_J_prob}, the following holds for $\zeta_{N} = o(1)$ and $q_{N} \gg 1$
    \begin{align}
        \P \big( \rW_{a}[\sU] \leq - \zeta_{N} \cdot q_{N} \big) = e^{-\D_{\rm{GH}}\cdot q_{N}} \,\,, \quad \P \big( \rW_{b}[\sU] \leq - \zeta_{N} \cdot q_{N} \big) = e^{-\D_{\rm{GH}}\cdot q_{N}}\,\,,\notag
    \end{align}
    where the rate function is
    \begin{align}
        \D_{\rm{GH}} =  \sum\limits_{\ell \in \sL} \frac{1}{2^{\ell - 1}} (\sqrt{a_{\ell}} - \sqrt{b_{\ell}})^2. \notag
    \end{align}
\end{lemma}
\begin{proof}[Proof of Lemma ~\ref{lem:binom_difference_prob}]
    For vertex $a\in \gV_{+}\cap \sU$, we write $\rW_{a}[\sU] = \rY_a - \rZ_{a}$, where $\rY_{a}$ and $\rZ_{a}$ are defined as
    \begin{align}
        \rY_{a} \coloneqq &\, \sum\limits_{\ell \in \sL} \log \Big( \frac{\alpha_{\ell}(1 - \beta_{\ell})}{\beta_{\ell}(1 - \alpha_{\ell})} \Big)\cdot \sum_{ e\in \gE^{(0)}_{\ell}(a) } \etA^{(\ell)}_{e}\indi{\substack{e\cap \gU = \{a\}\\e\setminus\{a\} \subset [\gV_{+} \setminus \gU]}}\,\,,\notag \\
        \rZ_{a} \coloneqq &\, \sum\limits_{\ell \in \sL} \log \Big( \frac{\alpha_{\ell}(1 - \beta_{\ell})}{\beta_{\ell}(1 - \alpha_{\ell})} \Big)\cdot \sum_{ e\in \gE^{(\ell - 1)}_{\ell}(a) } \etA^{(\ell)}_{e}\indi{\substack{e\cap \gU = \{a\}\\e\setminus\{a\} \subset [\gV_{-} \setminus \gU]}} \,.\notag
    \end{align}
According to Assumption \ref{ass:asymptotics_binary}, hyperedges $e$ in $\rY_{a}$ and $\rZ_{a}$ are sampled independently with probabilities $p^{(\ell)}_1 = \alpha_{\ell} = a_{\ell} \cdot q_{N} /\binom{N-1}{\ell - 1}$ and $p^{(\ell)}_2 = b_{\ell} \cdot q_{N}/\binom{N-1}{\ell - 1}$. We consider apply Lemma \ref{lem:LDP_binomial}. For each $\ell \in \sL$, we take $n_{\ell} = 2$, $\alpha^{(\ell)}_1 = a_{\ell}$, $\alpha^{(\ell)}_2 = b_{\ell}$ with $a_{\ell} > b_{\ell}$, and
\begin{align}
    c_1^{(\ell)} = \log \Big( \frac{\alpha_{\ell}(1 - \beta_{\ell})}{\beta_{\ell}(1 - \alpha_{\ell})} \Big)= \log \Big( \frac{a_{\ell}}{b_{\ell}} \Big), \quad c_2^{(\ell)} = -\log \Big( \frac{\alpha_{\ell}(1 - \beta_{\ell})}{\beta_{\ell}(1 - \alpha_{\ell})} \Big)= \log \Big( \frac{b_{\ell}}{a_{\ell}} \Big).\notag
\end{align}
At the same time, $|\gV_{+}| = |\gV_{-}| = N/2$, $|\gU\cap \gV_{+}| = |\gU\cap \gV_{-}| = \gamma_{N} N/2$ and $\gamma_{N} = o(1)$, we have
\begin{align}
    N_1^{(\ell)} = N_2^{(\ell)} = \binom{(1 - \gamma_{N})N/2}{\ell - 1} = (1 + o(1))\cdot \frac{1}{2^{\ell - 1}} \binom{N-1}{\ell - 1}\,.\notag
\end{align}
Thus $g_{N} = 1$ and $\rho^{(\ell)}_1 = \rho^{(\ell)}_2 = 2^{-(\ell - 1)}$ for each $\ell \in \sL$. Consequently, 
\begin{align*}
    \sum\limits_{\ell \in \sL}\sum\limits_{i=1}^{n_{\ell}}c^{(\ell)}_{i} \alpha^{(\ell)}_{i} \rho^{(\ell)}_{i} = \sum\limits_{\ell \in \sL}2^{-(\ell - 1)}(a_{\ell} - b_{\ell})\log(a_{\ell}/b_{\ell}) >0.
\end{align*}
Then all the conditions of Lemma \ref{lem:LDP_binomial} are satisfied. By taking $\epsilon = -\zeta_{N} < 0$ and $g_{N} = 1$, Lemma ~\ref{lem:LDP_binomial} indicates that
\begin{align}
    \P (\rY_{a} - \rZ_{a} \leq -\zeta_{N} \cdot q_{N}) = \exp \big( -[1 + o(1)] \cdot \D_{\rm{GH}}\cdot g_{N}\cdot q_{N} \big) = e^{-\D_{\rm{GH}}\cdot q_{N}}\,,\notag
\end{align}
where the rate function is defined as
\begin{align}
    \D_{\rm{GH}} \coloneqq \sup_{t\geq 0} \Big( \sum\limits_{\ell \in \sL} \frac{1}{2^{\ell - 1}} [ a_{\ell} + b_{\ell} - (a_{\ell})^{t}(b_{\ell})^{1-t} -(a_{\ell})^{1-t}(b_{\ell})^{t} ]\Big)\,.\notag 
\end{align}
Note that by union bound
\begin{align*}
    \D_{\rm{GH}} \leq \sum\limits_{\ell \in \sL} \frac{1}{2^{\ell - 1}} \sup_{t\geq 0}  [ a_{\ell} + b_{\ell} - (a_{\ell})^{t}(b_{\ell})^{1-t} -(a_{\ell})^{1-t}(b_{\ell})^{t} ]\, = \sum\limits_{\ell \in \sL} \frac{1}{2^{\ell - 1}} (\sqrt{a_{\ell}} - \sqrt{b_{\ell}})^2 
\end{align*}
On the other way around, the lower bound can be achieved by simply taking $t = 1/2$.

For vertex $b\in \gV_{-}\cap \sU$, we write $\rW_{b}[\sU] = \rY_b - \rZ_{b}$, where $\rY_{b}$ and $\rZ_{b}$ are defined as
    \begin{align*}
        \rY_{b} \coloneqq &\, \sum\limits_{\ell \in \sL} \log \Big( \frac{\alpha_{\ell}(1 - \beta_{\ell})}{\beta_{\ell}(1 - \alpha_{\ell})} \Big)\cdot \sum_{ e\in \gE^{(0)}_{\ell}(b) } \etA^{(\ell)}_{e}\indi{\substack{e\cap \gU = \{b\}\\e\setminus\{b\} \subset [\gV_{-} \setminus \gU]}}\,\,,\\
        \rZ_{b} \coloneqq &\, \sum\limits_{\ell \in \sL} \log \Big( \frac{\alpha_{\ell}(1 - \beta_{\ell})}{\beta_{\ell}(1 - \alpha_{\ell})} \Big)\cdot \sum_{ e\in \gE^{(\ell - 1)}_{\ell}(b) } \etA^{(\ell)}_{e}\indi{\substack{e\cap \gU = \{b\}\\e\setminus\{b\} \subset [\gV_{+} \setminus \gU]}}\,.
    \end{align*}
The proof is completed by following the same steps as above.
\end{proof}

\begin{lemma}[Generalized version of {\cite[Theorem 10]{Kim2018StochasticBM}}]\label{lem:LDP_binomial}
 Let $g_{N}$ be some non-decreasing function with respect to $N$ satisfying $1 \lesssim g_{N}  \ll N\log^{-2}(N)$. Let $\sL = \{\ell| \ell \geq 2, \ell\in \N\}$ be a set of integers with finite cardinality. For each $\ell \in \sL$, let $n_{\ell}\geq 2$ denote some finite positive integer. For each $\ell \in \sL$ and $i\in [n_{\ell}]$, let $\rY^{(\ell)}_{i}$ denote the random variable sampled from the binomial distribution $\mathrm{Bin}(N_i^{(\ell)}, p_i^{(\ell)})$, where
\begin{align}
    N_i^{(\ell)} = [1 + o(1)] \cdot \rho^{(\ell)}_{i} \cdot g_{N} \binom{N-1}{\ell - 1} \quad \textnormal{and} \quad p^{(\ell)}_{i} = [1 + o(1)]\cdot \alpha^{(\ell)}_{i} \cdot \frac{q_{N}}{\binom{N-1}{\ell - 1}}\,, \,\,
\end{align}
where $\rho_{i}^{(\ell)}$ and $\alpha^{(\ell)}_{i}$ are some positive constants independent of $N$. Let $\rX_{N} = \sum\limits_{\ell \in \sL}\sum\limits_{i=1}^{n_{\ell}} c^{(\ell)}_{i} \rY^{(\ell)}_{i}$ be a summation of random variables such that 
\begin{enumerate}
    \item at least one of the $c^{(\ell)}_{i}$'s is negative;
    \item $\sum\limits_{\ell \in \sL}\sum\limits_{i=1}^{n_{\ell}}c^{(\ell)}_{i} \alpha^{(\ell)}_{i} \rho^{(\ell)}_{i} >0$.
\end{enumerate}
With the parameters defined above, define the function $\D (t)$ by
\begin{align}
    \D (t) \coloneqq \sum\limits_{\ell \in \sL} \sum\limits_{i=1}^{n_{\ell}} \alpha^{(\ell)}_{i} \rho^{(\ell)}_{i} (1 - e^{-c^{(\ell)}_{i} t}) \,. \label{eqn:optimal_info}
\end{align}
Let $\epsilon \in (-\infty, \sum\limits_{\ell \in \sL}\sum\limits_{i=1}^{n_{\ell}}c^{(\ell)}_{i} \alpha^{(\ell)}_{i} \rho^{(\ell)}_{i})$, then the following large deviation principle holds 
\begin{align}
    \lim_{N \to \infty} [g_{N}q_{N}]^{-1}\log\P(\rX_{N} \leq  \epsilon \, g_{N} q_{N} )  = - \sup_{t\geq 0} \big( -\epsilon t + \D(t) \big) \,. \label{eqn:LDP_rate_function}
\end{align}
Consequently, given $\epsilon \in (-\infty, \sum\limits_{\ell \in \sL}\sum\limits_{i=1}^{n_{\ell}}c^{(\ell)}_{i} \alpha^{(\ell)}_{i} \rho^{(\ell)}_{i})$, for any $\delta >0$, there exists $N_{0} >0$ such that for all $N > N_{0}$ and $\epsilon_{N} < \epsilon$, the following holds
\begin{align}
    \P(\rX_{N} \leq \epsilon_{N} \, g_{N} q_{N} ) \leq e^{-g_{N} \cdot q_{N} \cdot [\sup_{t\geq 0} \{ -\epsilon t + \D(t)\} - \delta]}\,\,.\notag
\end{align}
\end{lemma}
\begin{proof}[Proof of Lemma ~\ref{lem:LDP_binomial}]
We first prove the upper bound of \eqref{eqn:LDP_rate_function} by calculating moment generating function. By Markov inequality and independence between $\rY^{(\ell)}_{i}$, for any $t\geq 0$, we have
\begin{align*}
\P(\rX_{N} \leq x_{N}) = &\,\P(e^{-t\rX_{N}} \geq e^{-tx_{N}}) \\
\leq &\, \frac{\E e^{-t\rX_{N}}}{e^{-tx_{N}}} = e^{tx_{N}}\prod_{\ell \in \sL}\prod_{i=1}^{n_{\ell}}\E e^{-t c^{(\ell)}_{i} \rY^{(\ell)}_{i}} = e^{tx}\prod_{\ell \in \sL}\prod_{i=1}^{n_{\ell}} \Big(1 - p^{(\ell)}_{i}(1 - e^{-tc^{(\ell)}_{i}}) \Big)^{N^{(\ell)}_{i}}\\
\leq &\, \exp \Big(tx_{N} - \sum\limits_{\ell \in \sL}\sum\limits_{i=1}^{n_{\ell}}N^{(\ell)}_{i} p^{(\ell)}_{i}(1 - e^{-c^{(\ell)}_{i} t}) \Big)\,, \quad (\textnormal{since } 1 - x \leq e^{-x})\\
= &\, \exp \bigg(- (1 + o(1))\cdot g_{N} \cdot q_{N} \cdot \Big( -\epsilon t + \sum\limits_{\ell \in \sL} \sum\limits_{i=1}^{n_{\ell}} \alpha^{(\ell)}_{i} \rho^{(\ell)}_{i} (1 - e^{-c^{(\ell)}_{i} t}) \Big) \bigg)\,,
\end{align*}
where we applied $x_{N} = (\epsilon + o(1)) \cdot g_{N} q_{N}$ in the last equality. By taking the supremum of of both sides with respect to $t$, the desired upper bound of \eqref{eqn:LDP_rate_function} follows immediately:
\begin{align*}
   \limsup_{N \to \infty} \frac{\log\P(\rX_{N} \leq x_{N} )}{g_{N} \cdot q_{N}} \leq -\sup_{t\geq 0} \Big(-\epsilon t + \D(t) \Big),
\end{align*}
where $\D(t)$ is defined in \eqref{eqn:optimal_info}.
One more thing to verify is the existence and uniqueness of the extreme point. Define the function $f(t) \coloneqq -t\epsilon + \sum\limits_{\ell \in \sL} \sum\limits_{i=1}^{n_{\ell}} \alpha^{(\ell)}_{i} \rho^{(\ell)}_{i} (1 - e^{-c^{(\ell)}_{i} t})$. For any $t\geq 0$, we have
\begin{align*}
    f^{\prime}(t) = -\epsilon + \sum\limits_{\ell \in \sL} \sum\limits_{i=1}^{n_{\ell}} \alpha^{(\ell)}_{i} c^{(\ell)}_{i}\rho^{(\ell)}_{i} e^{-c^{(\ell)}_{i} t}\,,\quad f^{\prime\prime}(t) = -\sum\limits_{\ell \in \sL} \sum\limits_{i=1}^{n_{\ell}} \alpha^{(\ell)}_{i} (c^{(\ell)}_{i})^2\rho^{(\ell)}_{i} e^{-c^{(\ell)}_{i} t} <0.
\end{align*}
Hence, $f(t)$ is strictly concave since $f^{\prime\prime}(t) < 0$ for all $t \geq 0$. By assumption,
\begin{align*}
    f^{\prime}(0) =  \sum\limits_{\ell \in \sL} \sum\limits_{i=1}^{n_{\ell}} \alpha^{(\ell)}_{i} c^{(\ell)}_{i}\rho^{(\ell)}_{i} - \epsilon >0,
\end{align*}
and $\lim_{t \to \infty}f^{\prime}(t) = -\infty$ since at least one of the $c_{i}^{(\ell)}$'s is negative. Consequently, there exists a unique $t^{\star} > 0$ such that $f^{\prime}(t^{\star}) = 0$ since $f^{\prime}(t)$ is continuous and strictly decreasing. Note that $f(0) = 0$ and we denote $f(t^{\star}) \coloneqq \sup_{t\geq 0}(-\epsilon t + \D(t)) > 0$. We also note that the proof of this lemma can not be finished by directly applying {\cite[Lemma H.5]{Abbe2022LPT}} since $f(t)$ is not convex.

We now turn to the lower bound. For each $\ell \in \sL$, let $\{y^{(\ell)}_{i}\}_{i\in[n_{\ell}]}$ be a collection of realizations sampled from the random variables $\{\rY^{(\ell)}_{i}\}_{i\in[n_{\ell}]}$, satisfying
\begin{align}
    \sum\limits_{\ell \in \sL}\sum\limits_{i=1}^{n_{\ell}} c_{i}^{(\ell)}y^{(\ell)}_{i} \leq x_{N} = \epsilon\cdot (1 + o(1)) \cdot g_{N}\cdot q_{N}\,.\label{eqn:condition_lower_bound}
\end{align}
Such a collection of $\{y^{(\ell)}_{i}\}_{i\in[n_{\ell}]}$ do exists. Recall that $f^{\prime}(t^{\star}) = 0$ for some $t^{\star} >0$, which is equivalent to
\begin{align*}
    \sum\limits_{\ell \in \sL} \sum\limits_{i=1}^{n_{\ell}} c^{(\ell)}_{i} \alpha^{(\ell)}_{i} \rho^{(\ell)}_{i} e^{-c^{(\ell)}_{i} t^{\star}} = \epsilon.
\end{align*}
For each $\ell \in \sL$ and $i\in [n_{\ell}]$ , let $y_{i}^{(\ell)} = (1 - o(1))\cdot \tau^{(\ell)}_{i} g_{N} q_{N}$ where we denote $\tau^{(\ell)}_{i} \coloneqq \alpha^{(\ell)}_{i} \rho^{(\ell)}_{i} e^{-c^{(\ell)}_{i} t^{\star}}$. Obviously, the so constructed collection of samples satisfies \eqref{eqn:condition_lower_bound}, and we have
\begin{align}\label{eqn:binomial_lower}
    \P(\rX_{N}\leq x_{N}) \geq \prod_{\ell \in \sL}\prod_{i=1}^{n_{\ell}}\P(\rY^{(\ell)}_{i} = y^{(\ell)}_{i}) = \prod_{\ell \in \sL}\prod_{i=1}^{n_{\ell}} \binom{N^{(\ell)}_{i}}{y^{(\ell)}_{i}}(p^{(\ell)}_{i})^{y^{(\ell)}_{i}}(1 - p^{(\ell)}_{i})^{N^{(\ell)}_{i} - y^{(\ell)}_{i}}\,.
\end{align}
Note that $y^{(\ell)}_{i} \ll \sqrt{N_i^{(\ell)}}$ since $g_{N} \ll N\log^{-2}(N)$, then by Lemma \ref{lem:stirling},
\begin{align*}
    &\, \log \left[ \binom{N^{(\ell)}_{i}}{y^{(\ell)}_{i}}(p^{(\ell)}_{i})^{y^{(\ell)}_{i}}(1 - p^{(\ell)}_{i})^{N^{(\ell)}_{i} - y^{(\ell)}_{i}} \right] \\
    \geq &\, y^{(\ell)}_{i} \log \Bigg( \frac{eN^{(\ell)}_{i}}{y^{(\ell)}_{i}} \Bigg) - \frac{1}{2}\log(y^{(\ell)}_{i}) - \frac{1}{12 y^{(\ell)}_{i}} - \log(4\sqrt{2\pi}) \\
    &\, \quad \quad \quad \quad \quad \quad \quad \quad \quad \quad \quad \quad + y^{(\ell)}_{i} \log(p_i^{(\ell)}) + (N^{(\ell)}_{i} - y^{(\ell)}_{i})\log(1 - p_i^{(\ell)})\,\\
    = &\, y^{(\ell)}_{i} \log \Bigg( \frac{eN^{(\ell)}_{i} p^{(\ell)}_{i}}{(1 - p^{(\ell)}_{i})y^{(\ell)}_{i}} \Bigg) + N^{(\ell)}_{i} \log(1 - p_i^{(\ell)}) - \frac{1}{2}\log(y^{(\ell)}_{i}) - \frac{1}{12y^{(\ell)}_{i}} - \log(4\sqrt{2\pi})\,.
\end{align*}
Moreover, we know that
\begin{align*}
    &\, y^{(\ell)}_{i} \log \Bigg( \frac{eN^{(\ell)}_{i} p^{(\ell)}_{i}}{(1 - p^{(\ell)}_{i})y^{(\ell)}_{i}} \Bigg) = (1 - o(1))g_{N} \cdot q_{N} \cdot \tau^{(\ell)}_{i} \log \bigg(\frac{e \alpha^{(\ell)}_{i} \rho^{(\ell)}_{i}}{\tau^{(\ell)}_{i}} \bigg)\,,\\
    &\, N^{(\ell)}_{i} \log(1 - p_i^{(\ell)}) = -(1 + o(1)) g_{N}\cdot q_{N} \cdot \alpha^{(\ell)}_{i} \rho^{(\ell)}_{i}\,, \quad \Big(\, \log(1 - x) = -x \textnormal{ for } x=o(1) \Big)\,,\\
    &\, \frac{1}{2}\log(y^{(\ell)}_{i})  - \frac{1}{12y^{(\ell)}_{i}} - \log(4\sqrt{2\pi}) = o(g_{N}\cdot q_{N})\,.
\end{align*}
Consequently, by taking logarithm on both side of \eqref{eqn:binomial_lower}, we have 
\begin{align*}
    \log \P(\rX_{N}\leq x_{N}) \geq -(1 + o(1)) g_{N}\cdot q_{N} \Bigg( \sum\limits_{\ell \in \sL} \sum\limits_{i=1}^{n_{\ell}} \bigg[ \alpha^{(\ell)}_{i} \rho^{(\ell)}_{i} - \tau^{(\ell)}_{i} \log \bigg(\frac{e \alpha^{(\ell)}_{i} \rho^{(\ell)}_{i}}{\tau^{(\ell)}_{i}} \bigg) \bigg] \Bigg)\,.
\end{align*}
Plugging the facts $\tau^{(\ell)}_{i} = \alpha^{(\ell)}_{i} \rho^{(\ell)}_{i} e^{-c^{(\ell)}_{i} t^{\star}}$ and $\sum\limits_{\ell \in \sL} \sum\limits_{i=1}^{n_{\ell}} c^{(\ell)}_{i} \alpha^{(\ell)}_{i} \rho^{(\ell)}_{i} e^{-c^{(\ell)}_{i} t^{\star}} = \epsilon$, we have
    \begin{align*}
        &\,\sum\limits_{\ell \in \sL} \sum\limits_{i=1}^{n_{\ell}} \bigg[ \alpha^{(\ell)}_{i} \rho^{(\ell)}_{i} - \tau^{(\ell)}_{i} \log \bigg(\frac{e \alpha^{(\ell)}_{i} \rho^{(\ell)}_{i}}{\tau^{(\ell)}_{i}} \bigg) \bigg] = \sum\limits_{\ell \in \sL} \sum\limits_{i=1}^{n_{\ell}}  \alpha^{(\ell)}_{i} \rho^{(\ell)}_{i} \bigg[1 - (1 + c^{(\ell)}_{i} t^{\star}) e^{-c^{(\ell)}_{i} t^{\star}} \bigg]\\
        =&\, - \sum\limits_{\ell \in \sL} \sum\limits_{i=1}^{n_{\ell}} \alpha^{(\ell)}_{i} \rho^{(\ell)}_{i} c^{(\ell)}_{i} e^{-c^{(\ell)}_{i} t^{\star}} \cdot t^{\star} + \sum\limits_{\ell \in \sL} \sum\limits_{i=1}^{n_{\ell}} \alpha^{(\ell)}_{i} \rho^{(\ell)}_{i} \cdot ( 1 - e^{-c^{(\ell)}_{i} t^{\star}} ) \\
        =&\, -\epsilon t^{\star} + \sum\limits_{\ell \in \sL} \sum\limits_{i=1}^{n_{\ell}} \alpha^{(\ell)}_{i} \rho^{(\ell)}_{i} ( 1 - e^{-c^{(\ell)}_{i} t^{\star}} ) = f(t^{\star}). 
    \end{align*}
Therefore, by taking infimum of both sides, we have
\begin{align*}
    \liminf_{N \to \infty} \frac{\log \P(\rX_{N}\leq x_{N})}{g_{N}\cdot q_{N}} \geq - \sup_{t\geq 0} \big( -\epsilon t + \D(t) \big)\,\,,
\end{align*}
which complete the proof for the desired lower bound.
\end{proof}

\section{Deferred proofs in \Cref{sec:achievability_matrices_binary}}\label{app:achievability_matrices_binary}

\subsection{Preliminary results}
The analysis of \Cref{thm:achievability_matrices} replies heavily on Lemmas \ref{lem:LDP_AM}, \ref{lem:row_concentration} and \Cref{thm:concentration}, which are the key ingredients to generalize the result from uniform hypergraphs to non-uniform hypergraphs.

\begin{lemma}[Row concentration]\label{lem:row_concentration}
    Following the conventions in model \ref{def:non_uniform_HSBM_binary}, let $p^{(\ell)}_{e}$ denote the probability of the $\ell$-hyperedge $e$ being sampled. Denote 
    \begin{align}
        p^{(\ell)} = \max_{e\in [\gV]^{\ell} } p^{(\ell)}_{e}, \quad \rho_{N} \coloneqq N \sum\limits_{\ell \in \sL} p^{(\ell)}\binom{N-2}{\ell - 2}.\notag
    \end{align}
     Denote $(1\vee x) = \max\{1, x\}$. For some constant $c>0$, define the function 
    \begin{align}\label{eqn:varphix}
        \varphi(x) = (5 + c)\cdot [1\vee \log(1/x)]^{-1}.
    \end{align}
    For any non-zero $\rvw \in \R^{N}$, define the map $\zeta(\rvw, c) : \R^{N} \times \R \mapsto \R$ as
    \begin{align}\label{eqn:zetawc}
         \zeta(\rvw, c) = \varphi\Big( \frac{\|\rvw\|_2}{\sqrt{N} \|\rvw\|_{\infty} } \Big) \cdot \|\rvw\|_{\infty} \rho_{N}.
    \end{align}
    Then for any $v\in \gV$ and $\rvw \in \R^{N}$ independent of $(\rmA - \rmA^{\star})_{v:}$, when $\rho_{N} \gtrsim \log(N)$, the following event holds with probability at least $1 - 2e^{-(3 +c)q_{N}}$
    \begin{align}\label{eqn:rowConcentration}
        |(\rmA - \rmA^{\star})_{v:}\, \rvw| \leq \zeta(\rvw, c).
    \end{align}
    \end{lemma}
    \begin{proof}[Proof of Lemma ~\ref{lem:row_concentration}]
    Without loss of generality, we assume $\|\rvw\|_{\infty} = 1$ since rescaling $\rvw$ does not change the event \eqref{eqn:rowConcentration}. Let $[\gV]^{\ell}$ denote the family of $\ell$-subsets of $V$, i.e., $[\gV]^{\ell} = \{\setT| \setT\subseteq \gV, |\setT| = \ell \}$. By definition \eqref{eqn:adjacency_matrix_entry_binary} of the adjacency matrix $\rmA$, we have 
        \begin{align}
            \rS_{v} \coloneqq (\rmA - \rmA^{\star})_{v:} \, \rvw = \sum\limits_{j=1}^{|\gV|} (\ermA_{vj} - \E \ermA_{vj}) \ervw_{j} = \sum\limits_{j=1}^{|\gV|} \ervw_{j} \indi{\{ v \neq j\} } \cdot \sum\limits_{\ell \in \sL}\,\,\,\,\sum\limits_{\substack{e\in [\gV]^{\ell}\\ e \supset \{v, j\}} }(\etA^{(\ell)}_{e} - \E \etA^{(\ell)}_{e} )\,\notag
        \end{align}
        Due to the independence between hyperedges and Markov inequality, for any $t>0$, we have
        \begin{align}
            \P( \rS_{v} \geq \delta ) = \P(e^{t\rS_{v}} \geq e^{t\delta})\leq e^{-t\delta} \prod_{j=1, j \neq v}^{|\gV|} \prod_{\ell \in \sL} \prod_{\substack{e\in [\gV]^{\ell}\\ e \supset \{v, j\}} }\E \exp[t\ervw_{j}(\etA^{(\ell)}_{e} - \E \etA^{(\ell)}_{e}) ].\notag
        \end{align}
        For simplicity, let $p^{(\ell)}_{e}$ denote the probability of the $\ell$-hyperedge $e$ being generated, then
    \begin{align*}
        &\, \E \exp[t\ervw_{j}(\etA^{(\ell)}_{e} - \E \etA^{(\ell)}_{e}) ] \\
        =&\, p^{(\ell)}_{e} e^{t\ervw_{j}(1 - p^{(\ell)}_{e})} + (1 - p^{(\ell)}_{e} ) e^{- t \ervw_{j} p^{(\ell)}_{e}} = e^{-t \ervw_{j} p^{(\ell)}_{e}}[1 - p^{(\ell)}_{e} + p^{(\ell)}_{e} e^{t\ervw_{j}}]\\
        \leq &\, \exp \big( -t\ervw_{j}p^{(\ell)}_{e} + p^{(\ell)}_{e}(e^{t \ervw_{j}} - 1)\big) = \exp \big( -p^{(\ell)}_{e}( 1 + t\ervw_{j} - e^{t \ervw_{j}})\big) \leq \exp\Big( \frac{e^{t\|\rvw\|_{\infty}}}{2}t^2\ervw_{j}^2 p^{(\ell)}_{e} \Big)
    \end{align*}
where we use the facts $1 + x \leq e^{x} \leq 1 + x + \frac{e^r}{2}x^2$ for $|x|\leq r$ in the last two inequalities. Taking the logarithm of $\P( \rS \geq \delta )$ and plugging in $\|\rvw\|_{\infty} = 1$, we have
    \begin{align}
        \log\P( \rS_{v} \geq \delta ) \leq&\, -t \delta + \sum\limits_{j=1}^{|\gV|} \indi{\{ v \neq j\} } \cdot \sum\limits_{\ell \in \sL}\,\,\,\,\sum\limits_{\substack{e\in [\gV]^{\ell}\\ e \supset \{v, j\}} } \frac{e^{t\|\rvw\|_{\infty}}}{2}t^2 \ervw_{j}^2 p^{(\ell)}_{e} \notag \\
        \leq &\, -t\delta + \frac{e^{t}}{2} t^2 \sum\limits_{\ell \in \sL} p^{(\ell)} \sum\limits_{\substack{e\in [\gV]^{\ell}\\ e \supset \{v, j\}} } \sum\limits_{j=1, j\neq v}^{|\gV|} \ervw_{j}^2 \notag \\
        \leq &\, -t\delta + \frac{e^{t}}{2} t^2 \|\rvw\|_2^2\cdot \sum\limits_{\ell \in \sL} p^{(\ell)} \binom{N-2}{\ell - 2} \leq -t \delta + \frac{e^{t}}{2} t^2 \|\rvw\|_2^2\cdot \rho_{N}, \notag
    \end{align}
    where the last two lines hold since the cardinality of $\ell$-hyperedges with $e\supset \{v, j\}$ is at most $\binom{N-2}{\ell - 2}$ and $\sum\limits_{j=1, j\neq v}^{|\gV|} \ervw_{j}^2 \leq \|\rvw\|_2^2$ with the definition of $\rho_{N}$ \eqref{eqn:rho_binary}. We take $t = 1 \vee \log(\sqrt{N}/\|\rvw\|_2)$ and note that $\log(\sqrt{N}/\|\rvw\|_2) \geq 0$ since $\|\rvw\|_2 \leq \sqrt{N}\|\rvw\|_{\infty} = \sqrt{N}$, then $t \leq 1 + \log(\sqrt{N}/\|\rvw\|_2)$. Also, $(1 \vee \log x )^2 \leq x$ for $x\geq 1$, then
    \begin{align}
         \frac{e^{t}}{2} t^2 \|\rvw\|_2^2 \leq \frac{e\sqrt{N}}{2} t^2\|\rvw\|_2 = \frac{e\sqrt{N}}{2} \|\rvw\|_2 \Big[ 1 \vee \log\Big( \frac{\sqrt{N}}{\|\rvw\|_2} \Big) \Big]^2 \leq \frac{e}{2}N.\notag
    \end{align}
    To conclude the result, for some constant $c>0$, we take 
    \begin{align}
        \delta = \varphi\Big( \frac{\|\rvw\|_2}{\sqrt{N} } \Big) \cdot N \cdot \sum\limits_{\ell \in \sL} p^{(\ell)} \binom{N-2}{\ell - 2} = \varphi\Big( \frac{\|\rvw\|_2}{\sqrt{N} } \Big) \rho_{N}.\notag
    \end{align}
    Therefore, 
    \begin{align}
        \log\P( \rS_{v} \geq \delta ) \leq \Big( - (5 + c) + \frac{e}{2} \Big)\rho_{N} \leq -(3 + c)\rho_{N} \lesssim -(3 + c)q_{N}.\notag
    \end{align}
    By replacing $\rvw$ with $-\rvw$, the similar bound for lower tail could be obtained. 
\end{proof}

Our analysis relies on the concentration result Theorem \ref{thm:concentration}. It holds for \emph{inhomogeneous Erd\H{o}s-R\'{e}nyi} random hypergraphs \ref{def:inhomo_ER_graph}, which is a natural generalization of standard random hypergraph models, where each edge is included independently with probabilities that can vary across different vertex subsets.
\begin{definition}[Inhomogeneous Erd\H{o}s-R\`enyi hypergraph]\label{def:inhomo_ER_graph}
    Let $\tQ^{(\ell)} \in ([0, 1]^{N})^{\otimes \ell}$ be a symmetric probability tensor, i.e., $\etQ_{i_1, \ldots, i_{\ell}} = \etQ_{i_{\pi(1)}, \ldots, i_{\pi(\ell)}}$ for any permutation $\pi$ on $[\ell]$, where $\ell \geq 2$ is some finite integer. Let $\gH_{\ell} = (\gV, \gE_{\ell})$ denote inhomogeneous $\ell$-uniform Erd\H{o}s-R\'{e}nyi hypergraph associated with the probability tensor $\tQ^{(\ell)}$. Let $\tA^{(\ell)}$ denote the adjacency tensor of $\gH_{\ell}$, where each $\ell$-hyperedge $e = \{i_1, \ldots, i_{\ell}\}\subset \gV$ appears with probability $\P(\etA^{(\ell)}_{e} = 1) = \etQ^{(\ell)}_{i_1,\dots, i_{\ell}}$. Let $\sL = \{\ell \mid \ell \geq 2, \ell \in \N \}$ be a finite set of integers, where $\LM$ denotes its maximum element. An inhomogeneous non-uniform Erd\H{o}s-R\'{e}nyi hypergraph is the union of uniform ones, i.e., $\gH = \cup_{\ell \in \sL} \gH_{\ell}$.
\end{definition}

\begin{theorem}\cite{Dumitriu2023OptimalAE}\label{thm:concentration}
Let $\gH = \cup_{\ell \in \sL} \gH_{\ell}$ be the inhomogeneous non-uniform Erd\H{o}s-R\'{e}nyi hypergraph in \Cref{def:inhomo_ER_graph}, associated with the probability tensors $\{\tQ^{(\ell)}\}_{\ell \in \sL}$. For each $\ell \in \sL$, we rescale the tensor by $\tD^{(\ell)} \coloneqq \binom{N-1}{\ell - 1}\tQ^{(\ell)}$, and denote $\etD^{(\ell)}_{\max}\coloneqq \max_{i_1, \ldots, i_{\ell} \in \gV} \etD^{(\ell)}_{i_1,\dots, i_{\ell}}$. Denote $\etD_{\max} \coloneqq \sum\limits_{\ell \in \sL} \etD^{(\ell)}_{\max}$. For some constant $\const_{\eqref{eqn:assumption_d}}>0$, suppose that
\begin{align}
    \etD_{\max} \coloneqq \sum\limits_{\ell \in \sL} \etD^{(\ell)}_{\max} \geq \const_{\eqref{eqn:assumption_d}}\cdot \log(N)\,.\label{eqn:assumption_d}
\end{align}
Then with probability at least $1-2N^{-10}- 2e^{-N}$, the adjacency matrix $\rmA$ of $\gH$ satisfies
\begin{align}
    \|\rmA - \E \rmA \| \leq \const_{\eqref{eqn:concentrateA}}\cdot \sqrt{\etD_{\max}}\,,\label{eqn:concentrateA}
\end{align}
where the constants $\const_{\eqref{eqn:concentrateA}}\coloneqq 10\LM^2 + 2\beta$, with $\beta= \beta_0 \sqrt{\beta_1} + \LM$ and $\beta_0$, $\beta_1$
satisfying
 \begin{align}
    &\, \beta_0 = 16+32 \LM (1+e^2)+1792(1+e^{-2})\LM^2, \notag \\
    &\, \LM^{-1}\beta_1 \log(  \LM^{-1}\beta_1) -  \LM^{-1}\beta_1 + 1 > 11/\const_{\eqref{eqn:assumption_d}}\,.\notag
\end{align} 
\end{theorem}

\subsection{Leave-one-out analysis}
Leave-one-out analysis is a powerful tool in the area of random matrix theory when analyzing the resolvent of matrices by applying Schur complement formula \cite{Bai2010Spectral}. We adapt this idea to our problem. For each $\ell \in \sL$, define $N$ auxiliary tensors $\tA^{(\ell, -1)}$, $\ldots$ $\tA^{(\ell, -N)}$, where $\tA^{(\ell, -v)} \in (\R^{|\gV|})^{\otimes \ell}$ is obtained by zeroing out the hyperedges containing $v$. Let $\rmA^{(\ell, -v)}$ denote the adjacency matrix obtained from $\tA^{(\ell, -v)}$, and define 
\begin{align}\label{eqn:A-v}
    \rmA^{(-v)} = \sum\limits_{\ell \in \sL} \rmA^{(\ell, -v)}.
\end{align}

The rest of this subsection is devoted to results concerning $\rmA, \rmA^{\star}$ and $\rmA^{(-v)}$, in preparation for the proof of Proposition \ref{prop:entrywise_diff_A} in the next subsection.

\begin{lemma}\label{lem:operatorNormA}
    Deterministically, $ \|\rmA^{\star}\|_{2\to \infty} \lesssim \rho_{N}/\sqrt{N}$. With probability at least $1-O(N^{-10})$,
        \begin{align}
             \|\rmA\|_{2\to \infty} \lesssim \sqrt{\rho_{N}} + \rho_{N}/\sqrt{N}.\notag
        \end{align}
    For each $v\in \gV$, the following holds with probability at least $1 - O(N^{-10})$
    \begin{align}
        \|\rmA - \rmA^{(-v)}\| \lesssim \|\rmA\|_{2 \to \infty},  \quad \|\rmA^{(-v)} - \rmA^{\star}\| \lesssim 2\sqrt{\rho_{N}} + \rho_{N}/\sqrt{N}.\notag
    \end{align}
\end{lemma}
\begin{proof}[Proof of Lemma ~\ref{lem:operatorNormA}]
    Let $p^{(\ell)} = \max_{e}p_e^{(\ell)}$, then $\rmA^{\star}_{ij} = \E[\rmA_{ij}] \leq \sum\limits_{\ell \in \sL} p^{(\ell)}\binom{N-2}{\ell - 2}$ for each $i \neq j$.  According to definition of $\rho_{N}$ in \eqref{eqn:rho_binary}, we have $\rho_{N} \asymp N \sum\limits_{\ell \in \sL} p^{(\ell)}\binom{N-2}{\ell - 2}$. The proof of the first sentence is finished by definition, since
        \begin{align}
            \|\rmA^{\star}\|_{2\to \infty} = \max_{\|\rvx\|_2 = 1} \|\rmA^{\star}\rvx\|_{\infty} = \max_{i\in \gV} \|\rmA^{\star}_{i:}\|_2 \leq \sqrt{N} \sum\limits_{\ell \in \sL} p^{(\ell)}\binom{N-2}{\ell - 2} \asymp \frac{\rho_{N}}{\sqrt{N}}.\notag
        \end{align}
     For the second argument, note that $\|\rmA\|_{2\to \infty} \leq \|\rmA\|$, then by \Cref{thm:concentration}, the following hold with probability at least $1-O(N^{-10})$ (where we take $q_{N} =\log(N)$)
        \begin{align}
            \|\rmA\|_{2\to \infty} \leq &\, \|\rmA - \rmA^{\star}\|_{2\to \infty} + \|\rmA^{\star}\|_{2\to \infty} \leq \|\rmA - \rmA^{\star}\| + \|\rmA^{\star}\|_{2\to \infty}\lesssim \sqrt{\rho_{N}} + \rho_{N}/\sqrt{N}.\notag
        \end{align}
     For the last part,  recall the definition of $\rmA^{(-v)}$ in \eqref{eqn:A-v}, then $(\rmA - \rmA^{(-v)})_{v:} = \rmA_{v:}$ while for $i\neq v$,
        \begin{align}
            (\rmA - \rmA^{(-v)})_{ij} = \indi{i \neq j } \cdot \sum\limits_{\ell \in \sL}\,\,\,\,\sum\limits_{\substack{e\in \gE_{\ell}\\ e \ni v,\,\, e \supset \{i, j\}} }\etA^{(\ell)}_{e}.\notag
        \end{align}
    Also, the hyperedge $e \supset \{i, v\}$ is counted $\ell - 1$ times in $i$th row, then
    \begin{align}
        \|(\rmA - \rmA^{(-v)})_{i:}\|_2 \leq \|(\rmA - \rmA^{(-v)})_{i:}\|_1 = \sum\limits_{\ell \in \sL} (\ell - 1)\ermA_{iv}.\notag
    \end{align}
    Since $\LM$ is finite, then
    \begin{align*}
        \|\rmA - \rmA^{(-v)}\|^2 \leq &\, \|\rmA - \rmA^{(-v)}\|_{\frob}^2 = \sum\limits_{i=1}^{|\gV|} \|(\rmA - \rmA^{(-v)})_{i:}\|_2^2 = \|\rmA_{v:}\|_2^2 + \sum\limits_{i\neq v} \|(\rmA - \rmA^{(-v)})_{i:}\|_2^2\\
        \leq &\, \|\rmA_{v:}\|_2^2 + \sum\limits_{i\neq v}\Big( \sum\limits_{\ell \in \sL} (\ell - 1)\ermA_{iv} \Big)^2 = \|\rmA_{v:}\|_2^2 + \Big(\sum\limits_{\ell \in \sL} (\ell - 1)\Big)^2 \cdot \sum\limits_{i\neq v}\ermA_{iv}^2\\
        \leq &\,\bigg( \Big(\sum\limits_{\ell \in \sL} (\ell - 1)\Big)^2 + 1 \bigg) \cdot \|\rmA_{v:}\|_2^2 \lesssim \|\rmA\|_{2\to \infty}.
    \end{align*}
    Consequently, with probability at least $1 - O(N^{-10})$, the following holds for each $v\in \gV$,
    \begin{align*}
        \|\rmA^{(-v)} - \rmA^{\star}\| \leq &\, \|\rmA^{(-v)} - \rmA\| + \|\rmA - \rmA^{\star}\| \lesssim \|\rmA\|_{2 \to \infty} + \|\rmA - \rmA^{\star}\| \lesssim \rho_{N}/\sqrt{N} + 2\sqrt{\rho_{N}},
    \end{align*}
    which complete the proof.
\end{proof}

Let $\{\lambda_i\}_{i=1}^{|\gV|}$, $\{\lambda_i^{(-v)}\}_{i=1}^{|\gV|}$ and $\{\lambda_i^{\star}\}_{i=1}^{|\gV|}$ denote the eigenvalues of $\rmA$, $\rmA^{(-v)}$ and $\rmA^{\star}$ respectively, where the eigenvalues are sorted in the decreasing order.
\begin{corollary}\label{cor:eigenvalue_approx_adj}
     For each $i\in \gV$, the following holds with probability at least $1-O(N^{-10})$,
    \begin{align*}
        |\lambda_i - \lambda_i^{(-v)}| \lesssim \sqrt{\rho_{N}} + \rho_{N}/\sqrt{N}, \quad |\lambda_i^{\star} - \lambda_i^{(-v)}| \lesssim 2\sqrt{\rho_{N}} + \rho_{N}/\sqrt{N}, \quad |\lambda_i - \lambda_i^{\star}| \lesssim \sqrt{\rho_{N}}.
    \end{align*}
    In particular in our model \ref{def:non_uniform_HSBM_binary}, with probability at least $1-O(N^{-10})$,
    \begin{align*}
        |\lambda_i| \asymp |\lambda^{\star}_i| \asymp |\lambda^{(-v)}_i| \asymp \rho_{N}, \quad  i = 1, 2.
    \end{align*}
\end{corollary}
    \begin{proof}[Proof of Corollary \ref{cor:eigenvalue_approx_adj}]
        By Weyl's inequality \ref{lem:weyl} and Lemma ~\ref{lem:operatorNormA}, for any $i\in \gV$,
        \begin{align*}
            |\lambda_i - \lambda_i^{(-v)}| \leq &\, \|\rmA - \rmA^{(-v)}\| \lesssim \sqrt{\rho_{N}} + \rho_{N}/\sqrt{N}, \quad |\lambda_i - \lambda_i^{\star}| \leq \|\rmA - \rmA^{\star}\| \lesssim \sqrt{\rho_{N}},\\
            |\lambda_i^{\star} - \lambda_i^{(-v)}| \leq &\, \|\rmA^{\star} - \rmA^{(-v)}\| \lesssim 2\sqrt{\rho_{N}} + \rho_{N}/\sqrt{N}.
        \end{align*}
       For the second part, $\rho_{N} \gtrsim \log(N)$, then the result follows since $\sqrt{\rho_{N}} \ll \rho_{N}$ and $\rho_{N}/\sqrt{N} \ll \rho_{N}$. 
\end{proof}

\begin{lemma}\label{lem:firstApproximation}
    Under model \eqref{def:non_uniform_HSBM_binary}, the following holds with probability at least $1 - O(N^{-10})$,
    \begin{subequations}
        \begin{align}
            &\, \max_{v\in \gV} \|s^{(-v)} \rvu^{(-v)} - s\rvu\|_2 \lesssim \| \rvu \|_{\infty},\\
            &\, \max_{v\in \gV} \|s^{(-v)} \rvu^{(-v)} - \rvu^{\star} \|_2 \lesssim \frac{1}{ \sqrt{\rho_{N}}} + \frac{1}{\sqrt{N}},\\
            &\, \max_{v\in \gV} \|s^{(-v)} \rvu^{(-v)} - \rvu^{\star} \|_{\infty} \lesssim \| \rvu \|_{\infty} + \frac{1}{\sqrt{N}}.
        \end{align}
    \end{subequations}
\end{lemma}
    \begin{proof}[Proof of Lemma ~\ref{lem:firstApproximation}]
    First, let $\delta = \min_{i\neq 2}|\lambda_{i}^{(-v)} - \lambda|$ denote the absolute eigen gap, then by generalized Davis-Kahan Lemma ~\ref{lem:generalized_Davis_Kahan} where we take $\rmD = \rmI_{N}$, $\Lambda = \lambda_{2}^{(-v)}$, $\rmU = s^{(-v)}\rvu^{(-v)}$, $\rmP = \rmI_{N} - \rmU\rmU^{\sT}$, $\widehat{\lambda} = \lambda$, $\widehat{\rvu}= s\rvu$, we have
    \begin{align}
        &\, \| s^{(-v)}\rvu^{(-v)} - s\rvu \|_2 \leq \sqrt{2}\|\rmP \rvu\|_2  \leq \frac{\sqrt{2} \|(\rmA^{(-v)}  - \lambda \rmI_{N} )\rvu\|_2}{\delta} = \frac{\sqrt{2} \|(\rmA^{(-v)}  - \rmA )\rvu\|_2}{\delta} \label{eqn:u-v-u_upper}
    \end{align}
    By Weyl's inequality Lemma ~\ref{lem:weyl} and Corollary \ref{cor:eigenvalue_approx_adj}, the following holds with probability at least $1 - O(N^{-10})$,
    \begin{align}
        \delta = &\, \min_{i\neq 2}|\lambda_{i}^{(-v)} - \lambda| = \min \big\{ |\lambda_3^{(-v)} - \lambda|, \,\, |\lambda_1^{(-v)} - \lambda|\big\} \notag \\
        \geq &\, \min \big\{ |\lambda_3 - \lambda|, \,\,|\lambda_1 - \lambda|\big\} - \|\rmA - \rmA^{(-v)}\|\notag\\
        \geq &\, \min \big\{ |\lambda^{\star}_3 - \lambda^{\star}|, \,\,|\lambda^{\star}_1 - \lambda^{\star}|\big\} - 2\|\rmA - \rmA^{\star}\| - \|\rmA - \rmA^{(-v)}\| \gtrsim \rho_{N},\label{eqn:delta_lower}
    \end{align}
    where the last line holds by \Cref{thm:concentration} and Lemma ~\ref{lem:operatorNormA}.
    Let $\rvw = (\rmA - \rmA^{(-v)})\rvu$. Note that $\| (\rmA - \rmA^{(-v)})_{i\cdot}\|_{1} = \sum\limits_{\ell \in \sL} (\ell - 1)\ermA_{iv}$ for $i\neq v$, then
    \begin{align}
        |\ervw_{v}| =&\, |(\rmA \rvu)_{v}| = |\lambda|\cdot |\ervu_{v}| \leq |\lambda| \cdot \|\rvu\|_{\infty},\notag\\
        |\ervw_{i}| = &\, |[(\rmA - \rmA^{(-v)})\rvu]_{i}|\leq \| (\rmA - \rmA^{(-v)})_{i\cdot}\|_{1} \cdot \|\rvu\|_{\infty} = \sum\limits_{\ell \in \sL} (\ell - 1)\ermA_{iv} \cdot \|\rvu\|_{\infty}, \,\, i\neq v.\notag
    \end{align}
    Again by Lemmas \ref{lem:operatorNormA}, \ref{cor:eigenvalue_approx_adj}, with probability at least $1 - O(N^{-10})$,
    \begin{align}
        \|(\rmA - \rmA^{(-v)})\rvu\|_2^2 \leq &\, \bigg( |\lambda|^2 + \sum\limits_{i\neq v} \Big(\sum\limits_{\ell \in \sL} (\ell - 1)\ermA_{iv} \Big)^2 \bigg) \cdot \|\rvu\|_{\infty}^2 \notag \\
        \lesssim &\, \big( |\lambda^{\star}|^2 + \|\rmA\|_{2 \to \infty} ) \cdot  \|\rvu\|_{\infty}^2 \lesssim \rho_{N}^2 \cdot \|\rvu\|_{\infty}^2.\notag
    \end{align}
    Therefore, by plugging in the lower bound of $\delta$ \eqref{eqn:delta_lower} into \eqref{eqn:u-v-u_upper}, we have
    \begin{align}
        \| s^{(-v)}\rvu^{(-v)} - s\rvu \|_2 \lesssim \|\rvu\|_{\infty}.\notag
    \end{align}
For the second part, $s^{(-v)} (\rvu^{(-v)} )^{\sT} \rvu^{\star} \geq 0$ by \eqref{eqn:signsofEigen}. Let $\delta^{\star} = \min \{ |\lambda^{\star}_3 - \lambda^{\star}_2|,|\lambda^{\star}_1 - \lambda^{\star}_2|\} \asymp \rho_{N}$, then by Davis-Kahan Lemma ~\ref{lem:DavisKahan} and Lemma ~\ref{lem:operatorNormA}, for any $v\in \gV$, the following holds with probability at least $1 - O(N^{-10})$,
        \begin{align}
            \|s^{(-v)} \rvu^{(-v)} - \rvu^{\star}\|_2 \leq \frac{4\|\rmA^{(-v)} - \rmA^{\star}\|} {\delta^{\star}} \lesssim \frac{2\sqrt{\rho_{N}} + \frac{\rho_{N}}{\sqrt{N}}}{\rho_{N}} \lesssim \frac{1}{ \sqrt{\rho_{N}}} + \frac{1}{\sqrt{N}},\notag
        \end{align}
Note that $\|\rvx\|_{\infty} \leq \|\rvx\|_2$ for any $\rvx \in \R^{N}$, then
    \begin{align}
        \| s^{(-v)}\rvu^{(-v)} - \rvu^{\star} \|_{\infty} \leq &\, \| s^{(-v)}\rvu^{(-v)} - s\rvu \|_{\infty} + \| \rvu \|_{\infty} + \| \rvu^{\star} \|_{\infty} \notag \\
        \leq &\, \| s^{(-v)}\rvu^{(-v)} - s\rvu \|_2 + \| \rvu \|_{\infty} + \| \rvu^{\star} \|_{\infty} \lesssim \| \rvu \|_{\infty} + \| \rvu^{\star} \|_{\infty}, \notag
    \end{align}
    and the desired result follows since $\| \rvu^{\star}\|_{\infty} = 1/\sqrt{N}$ by definition of our model.
\end{proof}
    
\begin{lemma}\label{lem:secondApproximation}
The following holds with probability at least $1 - O(N^{-3})$,
    \begin{align}
        \max_{v\in \gV} | \rmA_{v:}( s^{(-v)} \rvu^{(-v)} - \rvu^{\star}) | \lesssim \frac{\rho_{N}}{\log(\rho_{N})}\cdot \Big( \| \rvu \|_{\infty} + \frac{1}{\sqrt{N}} \Big).\notag
    \end{align}
\end{lemma}
\begin{proof}[Proof of Lemma ~\ref{lem:secondApproximation}]
    By definition of $\rmA^{(-v)}$ in \eqref{eqn:A-v}, the vector $\rvu^{(-v)}$ and the row $\rmA_{v:}$ are independent. Also, $\rvu^{\star}$ is independent of $\rmA$ as well since it is a constant vector. For ease of presentation, we denote $\rvw^{(-v)} \coloneqq s^{(-v)} \rvu^{(-v)} - \rvu^{\star}$. Then $\rvw^{(-v)}$ is independent of $\rmA_{v:}$. By triangle inequality,
        \begin{align}
            |\rmA_{v:} \rvw^{(-v)}|\leq &\, \underbrace{|\rmA_{v:}^{\star} \rvw^{(-v)}|}_{\circled{1}} + \underbrace{|(\rmA_{v:} - \rmA_{v:}^{\star}) \rvw^{(-v)}|}_{\circled{2}}.\notag
        \end{align}
    
    The first term \circled{1} can be bounded by Lemmas  \ref{lem:operatorNormA},\ref{lem:firstApproximation} as follows
    \begin{align}
        \circled{1} \leq \|\rmA^{\star}\|_{2\to \infty} \cdot \|\rvw^{(-v)}\|_2 \leq \frac{\rho_{N}}{\sqrt{N}} \| \rvu \|_{\infty}.\notag
    \end{align}
    The second term \circled{2} can be bounded by row concentration Lemma ~\ref{lem:row_concentration},
    \begin{align}
        |(\rmA_{v:} - \rmA_{v:}^{\star}) \rvw^{(-v)}| \leq \zeta(\rvw, c) = \varphi\Big( \frac{\|\rvw^{(-v)}\|_2}{\sqrt{N} \|\rvw^{(-v)}\|_{\infty} } \Big) \cdot \|\rvw^{(-v)}\|_{\infty} \rho_{N}.\notag
    \end{align}
    By the union bound, the event \eqref{eqn:rowConcentration} holds simultaneously for all $v\in \gV$ with probability at least $1 - O(N^{-2})$. Define 
    \begin{align}\label{eqn:gamman}
        \gamma_{N} = \frac{\const_{\eqref{eqn:concentrateA}}}{\sqrt{\rho_{N}}}.
    \end{align}
    We present the discussions based on the magnitude of $\frac{\|\rvw^{(-v)}\|_2}{\sqrt{N} \|\rvw^{(-v)}\|_{\infty}}$. Recall function $\varphi(x)$ in \eqref{eqn:varphix}.
    \begin{itemize}
        \item $\frac{\|\rvw^{(-v)}\|_2}{\sqrt{N} \|\rvw^{(-v)}\|_{\infty}} \leq \gamma_{N}$. Note that $\varphi(x):\R_{+} \mapsto \R_{+}$ is non-decreasing, then
            \begin{align}\notag
                |(\rmA_{v:} - \rmA_{v:}^{\star}) \rvw^{(-v)}| \leq \varphi(\gamma_{N})\cdot \|\rvw^{(-v)}\|_{\infty}\cdot \rho_{N}.
            \end{align}
        \item $\frac{\|\rvw^{(-v)}\|_2}{\sqrt{N} \|\rvw^{(-v)}\|_{\infty}} > \gamma_{N}$. Note that $\varphi(x)/x:\R_{+} \mapsto \R_{+}$ is non-increasing, then $\varphi(x) \leq x \varphi(\gamma_{N})/\gamma_{N}$ for any $x> \gamma_{N}$, thus
            \begin{align}\notag
                 |(\rmA_{v:} - \rmA_{v:}^{\star}) \rvw^{(-v)}| \leq \frac{\varphi(\gamma_{N})}{\gamma_{N}} \cdot \frac{\|\rvw^{(-v)}\|_2}{\sqrt{N}} \cdot \rho_{N}
            \end{align}
    \end{itemize}
    Combining the bounds above and Lemma ~\ref{lem:firstApproximation}, for each $v\in \gV$, the second term can be bounded as
    \begin{align}
        \circled{2} =&\, |(\rmA_{v:} - \rmA_{v:}^{\star}) \rvw^{(-v)}| \leq \varphi(\gamma_{N}) \rho_{N} \cdot \max\Big\{ \|\rvw^{(-v)}\|_{\infty},\,\,  \frac{\|\rvw^{(-v)}\|_2}{\gamma_{N} \sqrt{N}} \Big\} \notag \\
        \lesssim &\, \varphi(\gamma_{N}) \rho_{N} \cdot \max\Big\{\| \rvu \|_{\infty} + \| \rvu^{\star} \|_{\infty}\,, \,\, \frac{1}{\sqrt{N}} \Big\} \lesssim \frac{\rho_{N}}{\log(\rho_{N})} \cdot ( \| \rvu \|_{\infty} + \| \rvu^{\star} \|_{\infty}), \notag
     \end{align}
    where the last inequality holds since $\sqrt{N}\|\rvu\|_{\infty} \geq \|\rvu\|_2 = 1$ and $\varphi(\gamma_{N}) \asymp [1 \vee \log(\rho_{N})]^{-1}$. Consequently, combining the terms \circled{1} and \circled{2}, we have
    \begin{align}
        \max_{v\in \gV} | \rmA_{v:}( s^{(-v)} \rvu^{(-v)} - \rvu^{\star}) | \lesssim \Big(\frac{\rho_{N}}{\sqrt{N}} +  \frac{\rho_{N}}{\log(\rho_{N})} \Big) \cdot ( \| \rvu \|_{\infty} + \| \rvu^{\star}\|_{\infty}), \notag
    \end{align}
    and the desired result follows since $\| \rvu^{\star}\|_{\infty} = 1/\sqrt{N}$.
    \end{proof}
    
    \begin{lemma}\label{lem:thirdApproximation}
    With probability at least $1 - O(N^{-2})$,
        \begin{align}
           \|\rmA \rvu^{\star}\|_{\infty} \lesssim \frac{\rho_{N}}{\sqrt{N}}. 
        \end{align}
    \end{lemma}
    \begin{proof}[Proof of Lemma ~\ref{lem:thirdApproximation}]
    We first prove $\|(\rmA - \rmA^{\star}) \rvu^{\star}\|_{\infty} \lesssim \rho_{N}/\sqrt{N}$ as an intermediate result. By row concentration Lemma ~\ref{lem:row_concentration}, for each $v\in \gV$, with probability at least $1 - O(N^{-3})$,
    \begin{align}
        |(\rmA - \rmA^{\star})_{v:}\rvu^{\star}| \leq \varphi\Big( \frac{\|\rvu^{\star}\|_2}{\sqrt{N} \|\rvu^{\star}\|_{\infty} } \Big) \cdot \|\rvu^{\star}\|_{\infty} \rho_{N}  \lesssim \frac{\rho_{N}}{\sqrt{N}}, \notag
    \end{align}
    where the last equality holds since $\|\rvu^{\star}\|_2 = 1$, $\|\rvu^{\star}\|_{\infty} = 1/\sqrt{N}$ and $\varphi(1) = 5 + c = O(1)$ for some constant $c>0$. Then it follows by the union bound. Note that $\|\rmA^{\star}\|_{2\to \infty} \lesssim \rho_{N}/\sqrt{N}$ by Lemma ~\ref{lem:operatorNormA}, then by triangle inequality, with probability at least $1 - O(N^{-2})$,
    \begin{align}
        \|\rmA \rvu^{\star}\|_{\infty} \leq \|(\rmA - \rmA^{\star})\rvu^{\star}\|_{\infty} + \|\rmA^{\star}\rvu^{\star}\|_{\infty} \lesssim \frac{\rho_{N}}{\sqrt{N}} + \|\rmA^{\star}\|_{2\to \infty} \|\rvu^{\star}\|_{2} \lesssim \frac{\rho_{N}}{\sqrt{N}}. \notag
    \end{align}
    \end{proof}
    
\subsection{Proof of Proposition \ref{prop:entrywise_diff_A}}
\begin{proof}[Proof of Proposition \ref{prop:entrywise_diff_A}] Recall that $\{(\lambda_i, \rvu_i)\}_{i=1}^{|\gV|}$, $\{(\lambda_i^{\star}, \rvu_i^{\star})\}_{i=1}^{|\gV|}$ and $\{(\lambda_i^{(-v)}, \rvu_i^{(-v)})\}_{i=1}^{|\gV|}$ are eigenpairs of $\rmA$, $\rmA^{\star}$ and $\rmA^{(-v)}$ respectively, with the eigenvalues sorted in the decreasing order. We are in particular interested in the second eigenpairs. For notation convenience, we denote $\lambda = \lambda_2$, $\lambda^{\star} = \lambda_2^{\star}$, $\lambda^{(-v)} = \lambda_2^{(-v)}$ and $\rvu = \rvu_2$, $\rvu^{\star} = \rvu_2^{\star}$, $\rvu^{(-v)} = \rvu_2^{(-v)}$. We first prove that the eigenvector $\rvu$ is delocalized, i.e., 
\begin{align}
    \|\rvu\|_{\infty} \lesssim 1/\sqrt{N}. \label{eqn:udelocalized}
\end{align}
Note that $\<\rvu^{\star}, \rvu^{\star}\> = 1 >0$. Define the following sign functions for $\rvu$ and $\rvu^{(-v)}$,
\begin{align}\label{eqn:signsofEigen}
    s = \sign(\<\rvu, \rvu^{\star}\>), \quad s^{(-v)} = \sign(\<\rvu^{(-v)}, \rvu^{\star}\>),
\end{align}
Then $s\rvu^{\sT}\rvu^{\star}\geq 0$ and $s^{(-v)}(\rvu^{(-v)})^{\sT}\rvu^{\star} \geq 0$. Meanwhile $(\lambda^{\star})^{-1} \asymp (\lambda)^{-1}$ by Corollary \ref{cor:eigenvalue_approx_adj}, then
    \begin{align}
        \|\rvu\|_{\infty} = &\, \Big\|s \frac{\rmA \rvu}{\lambda} \Big\|_{\infty} \leq \Big\|\frac{\rmA( s\rvu - \rvu^{\star}) }{\lambda} \Big\|_{\infty} + \Big\|\frac{\rmA \rvu^{\star}}{\lambda} \Big\|_{\infty} \notag \\
        \lesssim &\, |\lambda^{\star}|^{-1} \Big( \max_{v\in \gV} |\rmA_{v:} (s\rvu - \rvu^{\star})| + \|\rmA\rvu^{\star}\|_{\infty} \Big)\notag\\
        \leq &\, |\lambda^{\star}|^{-1} \Big(  \max_{v\in \gV}| \rmA_{v:}(s\rvu - s^{(-v)} \rvu^{(-v)})| +  \max_{v\in \gV} | \rmA_{v:}( s^{(-v)} \rvu^{(-v)} - \rvu^{\star}) | +  \|\rmA \rvu^{\star}\|_{\infty}  \Big) \notag\\
        \leq &\, |\lambda^{\star}|^{-1} \Big(  \|\rmA\|_{2\to \infty}\| s\rvu - s^{(-v)} \rvu^{(-v)}\|_2 +  \max_{v\in \gV} | \rmA_{v:}( s^{(-v)} \rvu^{(-v)} - \rvu^{\star}) | +  \|\rmA \rvu^{\star}\|_{\infty}  \Big) \notag \\
        \lesssim &\, \rho_{N}^{-1} \Big[ (\sqrt{\rho_{N}} + \frac{\rho_{N}}{\sqrt{N}}) \|\rvu\|_{\infty} + \frac{\rho_{N}}{\log(\rho_{N})}\cdot \Big( \| \rvu \|_{\infty} + \frac{1}{\sqrt{N}} \Big)  + \frac{\rho_{N}}{\sqrt{N}} \Big],\notag
    \end{align}
    where the last inequality holds by Lemmas \ref{cor:eigenvalue_approx_adj}, \ref{lem:firstApproximation}, \ref{lem:secondApproximation}, \ref{lem:thirdApproximation}. By rearranging the terms on both sides,
    \begin{align}
        \Big(1 - \frac{1}{\sqrt{\rho_{N}}} - \frac{1}{\sqrt{N}} - \frac{1}{\log(\rho_{N})}\Big) \cdot \|\rvu\|_{\infty} \lesssim \Big(1 + \frac{1}{\log(\rho_{N})} \Big) \frac{1}{\sqrt{N}}.\notag
    \end{align}
    Then \eqref{eqn:udelocalized} follows since $\rho_{N} \gtrsim \log(N)$. Following the same strategy, we have
    \begin{align}
        &\,\Big\|s\rvu - \frac{ \rmA \rvu^{\star}}{\lambda^{\star}} \Big\|_{\infty} = \Big\|s\frac{\rmA \rvu}{\lambda} - \frac{\rmA \rvu^{\star}}{\lambda} + \frac{\rmA \rvu^{\star}}{\lambda} - \frac{\rmA \rvu^{\star}}{\lambda^{\star}} \Big\|_{\infty} \leq \frac{\| \rmA( s\rvu - \rvu^{\star} )\|_{\infty}  }{|\lambda|} + \frac{|\lambda - \lambda^{\star}|}{|\lambda| \cdot |\lambda^{\star}|} \cdot \|\rmA \rvu^{\star}\|_{\infty},\notag
    \end{align}
    where the first term can be bounded as before. For the second term, by Corollary \ref{cor:eigenvalue_approx_adj}, one has
    \begin{align}
        \frac{|\lambda - \lambda^{\star}|}{|\lambda| \cdot |\lambda^{\star}|} \lesssim \frac{\sqrt{\rho_{N}}}{\rho_{N}^2} = (\rho_{N})^{-3/2}.\notag
    \end{align}
    Then it follows that with probability at least $1 - O(N^{-2})$,
    \begin{align}
        \Big\|s\rvu - \frac{ \rmA \rvu^{\star}}{\lambda^{\star}} \Big\|_{\infty} \lesssim &\, \rho_{N}^{-1} \Big[ (\sqrt{\rho_{N}} + \frac{\rho_{N}}{\sqrt{N}}) \|\rvu\|_{\infty} + \frac{\rho_{N}}{\log(\rho_{N})}\cdot \Big( \| \rvu \|_{\infty} + \frac{1}{\sqrt{N}} \Big) \Big] + (\rho_{N})^{-3/2} \frac{\rho_{N}}{\sqrt{N}}\notag\\
        \lesssim &\, \Big( \frac{1}{\sqrt{\rho_{N}}} + \frac{1}{\sqrt{N}} + \frac{2}{\log(\rho_{N})} + \frac{1}{\sqrt{\rho_{N}}} \Big) \cdot \frac{1}{\sqrt{N}} \lesssim \frac{1}{\sqrt{N}\cdot \log(\rho_{N})},\notag
    \end{align}
    where the last inequality holds since the term $[\log(\rho_{N})]^{-1}$ dominates.
\end{proof}

\section{Deferred proofs in \Cref{sec:refinement_binary}}\label{app:refinement_binary}
\subsection{Proof of Lemma \ref{lem:subcritical-Ay-error}}
\begin{proof}[Proof of Lemma \ref{lem:subcritical-Ay-error}]
Without loss of generality, we assume $s = 1$. We further define $\widetilde{\rmA}$ as an intermediate matrix between $\widehat{\mA}$ \eqref{eqn:hatA} and $\overline{\mA}$ \eqref{eqn:barA}.  For each $v \in \gV$, we construct the edge set $\widehat{\gE}^{(r)}_{\ell}(v)$ based on the initial partition $\widehat{\rvy}^{(0)}$, as defined in \eqref{eqn:hat_edge_classification}. For each edge $e \in \widehat{\gE}^{(r)}_{\ell}(v)$, we consider its vertices other than $v$. For $j\in e\setminus\{v\}$, define the entry $\widetilde{\etA}^{(\ell)}_{e, v, j} = \etA^{(\ell)}_{e}/(\ell - 1)$ when $r\in \{0, \ell - 1\}$; when $1\leq r \leq \ell - 2$, we let
\begin{align}
    \widetilde{\etA}^{(\ell)}_{e, v, j} =&\, \etA^{(\ell)}_{e}/r,\,\, \textnormal{if }\widehat{\ervy}^{(0)}_v = \widehat{\ervy}^{(0)}_j; \quad \widetilde{\etA}^{(\ell)}_{e, v, j} = \etA^{(\ell)}_{e}/(\ell - 1-r), \,\,\textnormal{if }\widehat{\ervy}^{(0)}_v \neq \widehat{\ervy}^{(0)}_j. \label{eqn:tildeAevj}
\end{align}
With $\psi_{\ell}$ in \eqref{eqn:psi_l}, we define the weighted adjacency matrix $\widetilde{\rmA}$ by
\begin{align}
    \widetilde{\ermA}_{vj} \coloneqq \indi{v \neq j} \cdot \sum_{\ell \in \sL} \psi_{\ell} \sum_{r=0}^{\ell - 1} \sum_{e \in \widehat{\gE}^{(r)}_{\ell}(v) } \widetilde{\etA}^{(\ell)}_{e, v, j}.\label{eqn:tildeA}
\end{align}
Note that $\widetilde{\ermA}_{vj}$ and $\widehat{\ermA}_{vj}$ \eqref{eqn:hatA} differs by $\psi_{\ell}$ and $\widehat{\psi}_{\ell}$ only. Then by triangle inequality, we have
\begin{align}
    \|\widehat{\rmA}\widehat{\rvy}^{(0)} - \overline{\rmA}\rvy\|_{q_{N}} \leq \|\widehat{\rmA}\widehat{\rvy}^{(0)} - \widetilde{\rmA} \widehat{\rvy}^{(0)}\|_{q_{N}} + \|\widetilde{\rmA}\widehat{\rvy}^{(0)} - \overline{\rmA}\rvy\|_{q_{N}}.\notag
\end{align}
We claim the following two bounds hold with probability at least $1 - O(N^{-2})$,
\begin{subequations}
\begin{align}
    &\, \|\widehat{\rmA}\widehat{\rvy}^{(0)} - \widetilde{\rmA} \widehat{\rvy}^{(0)}\|_{q_{N}} \leq N^{1/q_{N}}\sqrt{q_{N}}, \label{eqn:Ay-error-claim1}\\
    &\, \|\widetilde{\rmA}\widehat{\rvy}^{(0)} - \overline{\rmA}\rvy\|_{q_{N}} \leq N^{1/q_{N}}\sqrt{q_{N}}. \label{eqn:Ay-error-claim2}
\end{align}
\end{subequations}
Note that $\lambda_{2}^{\star} \asymp q_{N}$, then the desired result follows by applying the above two claims.
\end{proof}

\begin{proof}[Proof of \eqref{eqn:Ay-error-claim2}]
For each $v\in \gV$, using the cancellation trick in \eqref{eqn:RVdiff}, we have
\begin{align}
    &\,\big[\widetilde{\rmA} \widehat{\rvy}^{(0)} - \overline{\rmA}\rvy\big]_{v} \notag \\
    = &\, \sum_{\ell \in \sL} \psi_{\ell}\cdot \bigg( \widehat{\ervy}_{v}^{(0)}\cdot \Big( \sum_{ e\in \widehat{\gE}^{(0)}_{\ell}(v) } \etA^{(\ell)}_{e} - \sum_{ e\in \widehat{\gE}^{(\ell - 1)}_{\ell}(v) } \etA^{(\ell)}_{e} \Big) - \ervy_{v}\cdot \Big( \sum_{ e\in \gE^{(0)}_{\ell}(v) } \etA^{(\ell)}_{e} - \sum_{ e\in \gE^{(\ell - 1)}_{\ell}(v) } \etA^{(\ell)}_{e} \Big) \bigg) \label{eqn:Wv-difference}
\end{align}
For vertices that are correctly classified, i.e., $v\in \gV$ with $\widehat{\ervy}^{(0)}_{v} = \ervy_{v}$, the following holds,   
\begin{align}
 |\eqref{eqn:Wv-difference}| \leq &\, \bigg|\sum_{\ell \in \sL} \psi_{\ell}\cdot \Big( \sum_{ e\in \widehat{\gE}^{(0)}_{\ell}(v)\setminus \gE^{(0)}_{\ell}(v) } \etA^{(\ell)}_{e}\,\, - \sum_{ e\in \widehat{\gE}^{(\ell - 1)}_{\ell}(v)\setminus \gE^{(\ell - 1)}_{\ell}(v) } \etA^{(\ell)}_{e} \Big) \bigg|\notag\\
\leq &\, \sum_{\ell \in \sL} \psi_{\ell}\sum_{ e\in \widehat{\gE}^{(0)}_{\ell}(v)\setminus \gE^{(0)}_{\ell}(v) } \etA^{(\ell)}_{e}\,\, +  \sum_{\ell \in \sL} \psi_{\ell}\sum_{ e\in \widehat{\gE}^{(\ell - 1)}_{\ell}(v)\setminus \gE^{(\ell - 1)}_{\ell}(v) } \etA^{(\ell)}_{e}\label{eqn:Wv-difference-upperbound}
\end{align}
Let $\widehat{\gN}^{(0)}$ denote the set of misclassified vertices in the initial partition $\widehat{\rvy}^{(0)}$ obtained by spectral method. Note that each edge in $\widehat{\gE}^{(0)}_{\ell}(v)\setminus \gE^{(0)}_{\ell}(v)$ or $\widehat{\gE}^{(\ell - 1)}_{\ell}(v)\setminus \gE^{(\ell - 1)}_{\ell}(v)$ contains at least one misclassified vertex in $\widehat{\gN}^{(0)}$ and any other $\ell - 2$ vertices in $\gV \setminus \{v\}$, then the following holds,
\begin{align}
    |\widehat{\gE}^{(0)}_{\ell}(v)\setminus \gE^{(0)}_{\ell}(v)| \leq |\widehat{\gN}^{(0)}|\cdot \binom{N-2}{\ell - 2}, \quad |\widehat{\gE}^{(\ell - 1)}_{\ell}(v)\setminus \gE^{(\ell - 1)}_{\ell}(v)| \leq |\widehat{\gN}^{(0)}|\cdot \binom{N-2}{\ell - 2}.\notag
\end{align}
For misclassified vertices, i.e., $v\in \gV$ with $\widehat{\ervy}^{(0)}_{v} \neq \ervy_{v}$, the following holds,
\begin{align}
 |\eqref{eqn:Wv-difference}| \leq &\, \bigg|\sum_{\ell \in \sL} \psi_{\ell}\cdot \Big( \sum_{ e\in \widehat{\gE}^{(0)}_{\ell}(v)\setminus \gE^{(\ell - 1)}_{\ell}(v) } \etA^{(\ell)}_{e}\,\, - \sum_{ e\in \widehat{\gE}^{(\ell-1)}_{\ell}(v)\setminus \gE^{(0)}_{\ell}(v) } \etA^{(\ell)}_{e} \Big) \bigg|\notag\\
\leq &\, \sum_{\ell \in \sL} \psi_{\ell}\sum_{ e\in \widehat{\gE}^{(0)}_{\ell}(v)\setminus \gE^{(\ell - 1)}_{\ell}(v) } \etA^{(\ell)}_{e}\,\, +  \sum_{\ell \in \sL} \psi_{\ell}\sum_{ e\in \widehat{\gE}^{(\ell - 1)}_{\ell}(v)\setminus \gE^{(0)}_{\ell}(v) } \etA^{(\ell)}_{e}\notag
\end{align}
Since $\widehat{\ervy}^{(0)}_{v} = - \ervy_{v}$, $\widehat{\gE}^{(0)}_{\ell}(v)$ (resp. $\widehat{\gE}^{(\ell - 1)}_{\ell}(v)$) is essentially an estimate of $\gE^{(\ell - 1)}_{\ell}(v)$ (resp. $\gE^{(0)}_{\ell}(v)$). Note that each edge in $\widehat{\gE}^{(0)}_{\ell}(v)\setminus \gE^{(\ell - 1)}_{\ell}(v)$ or $\widehat{\gE}^{(\ell - 1)}_{\ell}(v)\setminus \gE^{(0)}_{\ell}(v)$ contains at least one misclassified vertex in $\widehat{\gN}^{(0)}$ and any other $\ell - 2$ vertices in $\gV \setminus \{v\}$, then the following holds as well,
\begin{align}
    |\widehat{\gE}^{(0)}_{\ell}(v)\setminus \gE^{(\ell - 1)}_{\ell}(v)| \leq |\widehat{\gN}^{(0)}|\cdot \binom{N-2}{\ell - 2}, \quad |\widehat{\gE}^{(\ell - 1)}_{\ell}(v)\setminus \gE^{(0)}_{\ell}(v)| \leq |\widehat{\gN}^{(0)}|\cdot \binom{N-2}{\ell - 2}.\notag
\end{align}
Consequently, it suffices to bound \eqref{eqn:Wv-difference-upperbound} for each $v\in \gV$. For the first term,  note that hyperedges $\etA^{(\ell)}_{e}$ are independent since different vertices are included in each edge, by using Lemma \ref{lem:lp-moments}, we have
\begin{align}
    &\, \Big( \E \big( \sum_{ e\in \widehat{\gE}^{(0)}_{\ell}(v)\setminus \gE^{(0)}_{\ell}(v) } \etA^{(\ell)}_{e} \big)^{q_{N}} \Big)^{1/q_{N}}\notag\\
    \lesssim &\, \frac{q_{N}}{\log(q_{N})} \cdot \E \Big[\sum_{ e\in \widehat{\gE}^{(0)}_{\ell}(v)\setminus \gE^{(0)}_{\ell}(v) } \etA^{(\ell)}_{e} \Big] \leq \frac{q_{N}}{\log(q_{N})} \cdot |\widehat{\gE}^{(0)}_{\ell}(v)\setminus \gE^{(0)}_{\ell}(v)| \cdot \frac{\alpha_{\ell}q_{N}}{\binom{N-1}{\ell-1}} \notag \\
\lesssim &\, \frac{q_{N}}{\log(q_{N})} \cdot |\widehat{\gN}^{(0)}| \binom{N-2}{\ell -2} \cdot \frac{q_{N}}{\binom{N-1}{\ell-1}} \lesssim \frac{q_{N}^{2}}{\log(q_{N})} \cdot e^{-\frac{1}{2}q_{N}[\D_{\mathrm{AM}} + o(1)]}\,,\notag
\end{align}
where last step holds with probability at least $1 - e^{- \frac{1}{2}q_{N}[\D_{\mathrm{AM}} + o(1)]}$, since $|\widehat{\gN}^{(0)}|$ can be controlled by applying Markov inequality on Theorem \ref{thm:optimality_matrices}, i.e.,
\begin{align}
    \P(|\widehat{\gN}^{(0)}| \geq N e^{- \frac{1}{2}q_{N}[\D_{\mathrm{AM}} + o(1)]}) = \P\big(\eta_{N}(\widehat{\vy}^{(0)}, \vy) \geq e^{- \frac{1}{2}q_{N}[\D_{\mathrm{AM}} + o(1)]} \big) \leq e^{- \frac{1}{2}q_{N}[\D_{\mathrm{AM}} + o(1)]}.\notag 
\end{align}
The second term in \eqref{eqn:Wv-difference-upperbound} can be bounded similarly. Note that $\psi_{\ell} = O(1)$ for each $\ell \in \sL$ by assumption. By applying triangle inequality on results above, the following holds with probability at least $1 - e^{- \frac{1}{2}q_{N}[\D_{\mathrm{AM}} + o(1)]}$
\begin{align}
    \E [\|\widetilde{\rmA} \widehat{\rvy}^{(0)} - \overline{\rmA}\rvy\|_{q_{N}}^{q_{N}}] \lesssim &\, N\cdot e^{q_{N}\cdot(2\log(q_{N}) -\log\log(q_{N}))} \cdot e^{- \frac{1}{2}q_{N}^{2}[\D_{\mathrm{AM}} + o(1)]}.\notag
\end{align}
Then by Markov's inequality, 
\begin{align}
    &\, \P(\|\widetilde{\rmA} \widehat{\rvy}^{(0)} - \overline{\rmA}\rvy\|_{q_{N}} \geq N^{1/q_{N}}\sqrt{q_{N}}) \notag \\
    \leq &\, \frac{\E [\|\widetilde{\rmA} \widehat{\rvy}^{(0)} - \overline{\rmA}\rvy\|_{q_{N}}^{q_{N}}]}{(N^{1/q_{N}}\sqrt{q_{N}})^{q_{N}}} \lesssim \frac{N\cdot e^{q_{N}\cdot(2\log(q_{N}) -\log\log(q_{N}))} \cdot e^{- \frac{1}{2}q_{N}^{2}[\D_{\mathrm{AM}} + o(1)]}}{N \cdot e^{\frac{1}{2}q_{N}\log(q_{N})}} \notag \\
    \lesssim &\, e^{\frac{3}{2}q_{N}\log(q_{N}) - q_{N}^{2}[\D_{\mathrm{AM}} + o(1)]} \lesssim e^{ -\frac{1}{4}q_{N}^{2}(\D_{\mathrm{AM}} + o(1))},\notag
\end{align}
where the last inequality follows since $1\ll q_{N} \lesssim \log(N)$ and $\D_{\mathrm{AM}} \gtrsim 1$, thus proved.
\end{proof}

\begin{proof}[Proof of \eqref{eqn:Ay-error-claim1}]
We present some preliminaries first. Recall \eqref{eqn:alpha_beta_def} and \eqref{eqn:EA_binary}, for each $\ell \in \sL$, we have
    \begin{align}
        &\,\E[\rmA^{(\ell)} ] + \Bigg[\alpha_{\ell} \binom{N/2-2}{\ell - 2} + \beta_{\ell} \Big[ \binom{N-2}{\ell - 2} - \binom{N/2-2}{\ell - 2}\Big] \Bigg] \rmI_{N} \notag \\
        =&\, \frac{1}{2} \bigg[ (\alpha_{\ell} - \beta_{\ell})\binom{N/2-2}{\ell - 2} + 2\beta_{\ell} \binom{N-2}{\ell - 2}\bigg]\ones_{N}\ones_{N}^{\sT} +  \frac{1}{2} \bigg[ (\alpha_{\ell} - \beta_{\ell})\binom{N/2-2}{\ell - 2} \bigg]\rvy \rvy^{\sT}. \notag
    \end{align}
Consequently, for each $\ell \in \sL$, the following holds,
    \begin{align}
        \circled{1} &\, \coloneqq \frac{1}{N}\ones_{N}^{\sT} \E[\rmA^{(\ell)}] \ones_{N} = \frac{1}{2} \bigg[ (N-2)(\alpha_{\ell} - \beta_{\ell})\binom{N/2-2}{\ell - 2} + 2(N-1)\beta_{\ell} \binom{N-2}{\ell - 2}\bigg]\,\,, \notag\\
        \circled{2} &\, \coloneqq \frac{1}{N}\rvy^{\sT} \E[\rmA^{(\ell)} ] \, \rvy = \frac{1}{2} \bigg[ (N-2)(\alpha_{\ell} - \beta_{\ell})\binom{N/2-2}{\ell - 2} - 2\beta_{\ell} \binom{N-2}{\ell - 2}\bigg]\,\,.\notag
    \end{align}
Then $\alpha_{\ell}$ and $\beta_{\ell}$ can be represented in terms of $\circled{1}$ and $\circled{2}$ as follows,
    \begin{align}
        \alpha_{\ell} = \frac{1}{(N-2)\binom{N-2}{\ell - 2}} \Bigg[  \frac{\binom{N-2}{\ell - 2}}{\binom{N/2 - 2}{\ell - 2}} \Big( \circled{1} + \circled{2}\Big) - \bigg[ \frac{\binom{N-2}{\ell - 2}}{\binom{N/2 - 2}{\ell - 2}} - 1 \bigg] \Big(\circled{1} - \circled{2}\Big) \Bigg], \quad \beta_{\ell} = \frac{\circled{1} - \circled{2}}{N\binom{N-2}{\ell - 2}}.\notag
    \end{align}
Similarly, we define
    \begin{align}
       \widehat{\circled{1}} \,\, \coloneqq \frac{1}{N}\ones_{N}^{\sT} \rmA^{(\ell)} \ones_{N}, \quad \widehat{\circled{2}} &\, \coloneqq \frac{1}{N}(\widehat{\rvy}^{(0)})^{\sT} \rmA^{(\ell)} \widehat{\rvy}^{(0)}.\notag
    \end{align}
Then $\widehat{\alpha}_{\ell}$ and $\widehat{\beta}_{\ell}$ can be represented in terms of $\widehat{\circled{1}}$ and $\widehat{\circled{2}}$ as follows,
    \begin{align}
        \widehat{\alpha}_{\ell} = \frac{1}{(N-2)\binom{N-2}{\ell - 2}} \Bigg[  \frac{\binom{N-2}{\ell - 2}}{\binom{N/2 - 2}{\ell - 2}} \Big( \widehat{\circled{1}} + \widehat{\circled{2}}\Big) - \bigg[ \frac{\binom{N-2}{\ell - 2}}{\binom{N/2 - 2}{\ell - 2}} - 1 \bigg] \Big(\widehat{\circled{1}} - \widehat{\circled{2}}\Big) \Bigg], \quad \widehat{\beta}_{\ell} = \frac{\widehat{\circled{1}} - \widehat{\circled{2}}}{N\binom{N-2}{\ell - 2}}.\notag
    \end{align}
Note that $\alpha_{\ell} \ll 1$, $\beta_{\ell} \ll 1$ and $\ell \ll N$ when $N$ is sufficiently large, then $\psi_{\ell} = \log(\alpha_{\ell}/\beta_{\ell}) + o(1)$ and
\begin{align}
\log \Big( \frac{\widehat{\alpha}_{\ell}}{\widehat{\beta}_{\ell}} \Big) = \log\Bigg( 2^{\ell-2} \frac{\widehat{\circled{1}} + \widehat{\circled{2}}}{\widehat{\circled{1}} - \widehat{\circled{2}}} -  2^{\ell-2} + 1\Bigg) = \widehat{\psi}_{\ell}.\notag
\end{align}
where $\psi_{\ell}$ and $\widehat{\psi}_{\ell}$ are defined in \eqref{eqn:psi_l} and \eqref{eqn:psi_hat} respectively.

We now turn to the proof of \eqref{eqn:Ay-error-claim1}. Recall $\widetilde{\rmA}$ in \eqref{eqn:tildeA} and $\widehat{\rmA}$ in \eqref{eqn:hatA}. For each $v\in \gV$, we have
\begin{align}
    [ (\widehat{\rmA}- \widetilde{\rmA})\widehat{\rvy}^{(0)}]_{v} = &\, \sum_{\ell \in \sL} (\psi_{\ell} - \widehat{\psi}_{\ell})\cdot \sum_{r=0}^{\ell - 1} \sum_{e \in \widehat{\gE}^{(r)}_{\ell}(v)} \widetilde{\etA}^{(\ell)}_{e, v, j} \notag.
\end{align}
The problem is now reduced to bounding $|\psi_{\ell} - \widehat{\psi}_{\ell}|$ for each $\ell \in \sL$. Since $\log(1 + x) \leq x$, then
\begin{align}
    |\psi_{\ell} - \widehat{\psi}_{\ell}| = &\, |\log( \alpha/\widehat{\alpha}_{\ell})| + |\log(\widehat{\beta}_{\ell}/\beta_{\ell})| \leq \frac{|\widehat{\alpha}_{\ell} - \alpha_{\ell}|}{\alpha_{\ell}} + \frac{|\widehat{\beta}_{\ell} - \beta_{\ell}|}{\beta_{\ell}}.\notag
\end{align}
We claim that the following holds with probability at least $1 - e^{ -\frac{1}{2}q_{N}^{2}(\D_{\mathrm{AM}} + o(1))}$
\begin{align}
    |\widehat{\circled{1}} - \circled{1}| \lesssim  q_{N}^{-1}, \quad |\widehat{\circled{2}} - \circled{2}| \lesssim q_{N}^{-1}.\label{eqn:circle1_circle2_bounds}
\end{align}
Note that $\circled{1} \asymp q_{N}$ and $\circled{2} \asymp q_{N}$, then $|\psi_{\ell} - \widehat{\psi}_{\ell}| \lesssim q_{N}^{-2}$. Using Lemma \ref{lem:lp-moments}, for each $v\in \gV$,
\begin{align}
    \Big( \E [ (\widehat{\rmA}- \widetilde{\rmA})\widehat{\rvy}^{(0)}]_{v}^{q_{N}} \Big)^{1/q_{N}} \lesssim &\, \frac{1}{q_{N}^{2}}\frac{q_{N}}{\log(q_{N})} \Big| \E[(\widehat{\rmA}- \widetilde{\rmA})\widehat{\rvy}^{(0)}]_{v} \Big| \lesssim \frac{1}{q_{N}^{2}} \frac{q_{N}^{2}}{\log(q_{N})} = \frac{1}{\log(q_{N})}. \notag 
\end{align}
Then by Markov's inequality, 
\begin{align}
    &\, \P(\|(\widehat{\rmA}- \widetilde{\rmA})\widehat{\rvy}^{(0)}\|_{q_{N}} \geq N^{1/q_{N}}\sqrt{q_{N}})\leq \frac{\E [\|(\widehat{\rmA}- \widetilde{\rmA})\widehat{\rvy}^{(0)}\|_{q_{N}}^{q_{N}}]}{(N^{1/q_{N}}\sqrt{q_{N}})^{q_{N}}} \lesssim \frac{N e^{-q_{N}\log\log(q_{N}))}}{Ne^{\frac{1}{2}q_{N}\log(q_{N})}} \lesssim e^{-\frac{1}{3}q_{N}\log(q_{N})},\notag
\end{align}
which completes the proof of \eqref{eqn:Ay-error-claim1}.
\end{proof}

\begin{proof}[Proof of \eqref{eqn:circle1_circle2_bounds}]
For $|\widehat{\circled{1}} - \circled{1}|$, according to Bernstein's inequality \ref{lem:Bernstein}, we have
    \begin{align}
        \P\bigg( \frac{1}{N} \Big| \ones_{N}^{\sT} \big(\rmA^{(\ell)} - \E[\rmA^{(\ell)}] \big)\ones_{N}\Big| \geq q_{N}^{-1} \bigg) \leq &\, 2\exp\bigg( -\frac{N^{2}/(2q_{N}^{2})}{\Var(\ones_{N}^{\sT} \rmA^{(\ell)} \ones_{N}) + (2/3)N/q_{N}} \bigg) \leq 2e^{-N/(2q_{N}^{3})},\notag
    \end{align}
where the last inequality holds since $\ell, \alpha_{\ell}, \beta_{\ell}$ are constants and
\begin{align}
    \Var(\ones_{N}^{\sT} \rmA^{(\ell)} \ones_{N}) \lesssim &\, (\ell - 1)^{2} \sum_{e\in \gE^{(\ell)}} \Var(\etA^{(\ell)}_{e}) \lesssim \binom{N}{\ell} \frac{\alpha_{\ell} q_{N}}{\binom{N-1}{\ell - 1}} \cdot \bigg(1 - \frac{\alpha_{\ell} q_{N}}{\binom{N-1}{\ell - 1}}\bigg) \lesssim N q_{N}\,.\notag
\end{align}

For $|\widehat{\circled{2}} - \circled{2}|$, we have the following decomposition,
\begin{align}
    (\widehat{\rvy}^{(0)})^{\sT} \rmA^{(\ell)} \widehat{\rvy}^{(0)} - \rvy^{\sT} \E[\rmA^{(\ell)}] \rvy = &\, (\widehat{\rvy}^{(0)} + \rvy)^{\sT} \rmA^{(\ell)} (\widehat{\rvy}^{(0)} - \rvy) + \rvy^{\sT} (\rmA^{(\ell)} - \E[\rmA^{(\ell)}]) \rvy.\notag
\end{align}
Following similar steps as above, the second term can be bounded using Bernstein's inequality,
\begin{align}
    \P\bigg( \frac{1}{N} \Big| \rvy^{\sT} (\rmA^{(\ell)} - \E[\rmA^{(\ell)}]) \rvy \Big| \geq q_{N}^{-1} \bigg) \leq &\, 2\exp\bigg( -\frac{N^{2}/(2q_{N}^{2})}{\Var(\rvy^{\sT} \rmA^{(\ell)} \rvy) + (2/3)N/q_{N}} \bigg) \leq 2e^{-N/(2q_{N}^{3})}.\notag
\end{align}
For the first term,  note that $\widehat{\ervy}_{j}^{(0)} , \ervy_{j} \in \{\pm 1\}$ for each $j\in [N]$, then $\|\widehat{\rvy}^{(0)} + \rvy\|_{2} \leq 2\sqrt{N}$. Since $\|\rvy\|_{r} \leq N^{1/r - 1/p} \|\rvy\|_{p}$ for any $p \geq r \geq 1$ due to H\H{o}lder's inequality, then we have
\begin{align}
    &\, \frac{1}{N}|(\widehat{\rvy}^{(0)} + \rvy)^{\sT} \rmA^{(\ell)} (\widehat{\rvy}^{(0)} - \rvy)| \leq \frac{1}{N} (\| \widehat{\rvy}^{(0)}\|_{2} + \|\rvy\|_{2}) \cdot \| \rmA^{(\ell)} (\widehat{\rvy}^{(0)} - \rvy)\|_{2} \notag \\
    \leq &\, \frac{1}{\sqrt{N}} \cdot 2\sqrt{N} \cdot N^{1/2 - 1/q_{N}} \cdot \| \rmA^{(\ell)} (\widehat{\rvy}^{(0)} - \rvy)\|_{q_{N}} \notag \leq 2 N^{- 1/q_{N}} \cdot N^{1/q_{N}}q_{N}^{-1}\leq 2q_{N}^{-1},\notag
\end{align}
where in the last line, we applied the fact that the following holds with probability $1 - O(N^{-2})$,
\begin{align}
    \| \rmA^{(\ell)} (\widehat{\rvy}^{(0)} - \rvy)\|_{q_{N}} \leq N^{1/q_{N}}q_{N}^{-1}, \label{eqn:Ayhatminusy}
\end{align}
For the proof of \eqref{eqn:Ayhatminusy}, note that for each $v\in \gV$, we have
\begin{align}
    &\, \Big| [\rmA^{(\ell)}(\widehat{\rvy}^{(0)} - \rvy)]_{v}  \Big| \leq \sum_{j\in [N]} \indi{\widehat{\ervy}_{j}^{(0)} \neq \ervy_{j} }\cdot 2 \sum_{e\in \gE_{\ell},\,e \supset \{v, j\} } \etA^{(\ell)}_{e} \leq 4 \sum_{r = 0}^{\ell - 1}\sum_{e\in \widehat{\gE}_{\ell}^{(r)}(v)\setminus \gE_{\ell}^{(r)}(v)} \etA^{(\ell)}_{e}.\notag
\end{align}
Then by using Lemma \ref{lem:lp-moments}, for $r\in \{0, \ldots,\ell - 1\}$ and $v\in \gV$,
\begin{align}
    \Big( \E \big( \sum_{ e\in \widehat{\gE}^{(r)}_{\ell}(v)\setminus \gE^{(r)}_{\ell}(v) } \etA^{(\ell)}_{e} \big)^{q_{N}} \Big)^{1/q_{N}} \lesssim &\, \frac{q_{N}}{\log(q_{N})} \cdot \E \Big[\sum_{ e\in \widehat{\gE}^{(r)}_{\ell}(v)\setminus \gE^{(r)}_{\ell}(v) } \etA^{(\ell)}_{e} \Big] \leq \frac{q_{N}}{\log(q_{N})} \cdot |\widehat{\gE}^{(r)}_{\ell}(v)\setminus \gE^{(r)}_{\ell}(v)| \cdot \frac{\alpha_{\ell}q_{N}}{\binom{N-1}{\ell-1}} \notag \\
\lesssim &\, \frac{q_{N}}{\log(q_{N})} \cdot |\widehat{\gN}^{(0)}| \binom{N-2}{\ell -2} \cdot \frac{q_{N}}{\binom{N-1}{\ell-1}} \lesssim \frac{q_{N}^{2}}{\log(q_{N})} \cdot e^{-\frac{1}{2}q_{N}[\D_{\mathrm{AM}} + o(1)]}\,,\notag
\end{align}
where in the last line, $|\widehat{\gN}^{(0)}| \leq N e^{-\frac{1}{2}q_{N}[\D_{\mathrm{AM}} + o(1)]}$ holds with probability at least $1 - e^{- \frac{1}{2}q_{N}[\D_{\mathrm{AM}} + o(1)]}$, as proved in the proof of \eqref{eqn:Ay-error-claim2}. Then by Markov's inequality, 
\begin{align}
    &\, \P(\|\rmA^{(\ell)}(\widehat{\rvy}^{(0)} - \rvy)\|_{q_{N}} \geq N^{1/q_{N}}q_{N}^{-1}) \notag\\
    \leq &\, \frac{\E [\|\widetilde{\rmA} \widehat{\rvy}^{(0)} - \overline{\rmA}\rvy\|_{q_{N}}^{q_{N}}]}{N q_{N}^{-q_{N}}} \lesssim e^{q_{N}\cdot(3\log(q_{N}) -\log\log(q_{N}))} \cdot e^{- \frac{1}{2}q_{N}^{2}[\D_{\mathrm{AM}} + o(1)]} \lesssim e^{ -\frac{1}{4}q_{N}^{2}(\D_{\mathrm{AM}} + o(1))},\notag
\end{align}
where the last inequality follows since $1\ll q_{N} \lesssim \log(N)$ and $\D_{\mathrm{AM}} \gtrsim 1$.
\end{proof}

\subsection{Proof of Lemma \ref{lem:LDP-v-error}}
\begin{proof}[Proof of Lemma \ref{lem:LDP-v-error}]
    Without loss of generality, we assume $s = 1$. Let $\sM = \{ j\in [N]: \sign(\evv_{j}) \neq \sign(\overline{\evv}_{j}) \}$ denote the set of indices having different signs between $\sign(\vv)$ and $\sign(\overline{\vv})$. Then by definition,
    \begin{align}
        \eta_{N}(\sign(\vv), \sign(\overline{\vv})) = N^{-1}\cdot |\sM|.\notag
    \end{align}
Denote $\vr\coloneqq \vv - \vw - \overline{\vv}$. For each $j\in \sM$, the following holds,
\begin{align}
    \evv_{j} \cdot \overline{\evv}_{j} \leq 0 \implies -(\evv_{j} - \overline{\evv}_{j})\cdot \sign(\overline{\evv}_{j}) \geq |\overline{\evv}_{j}| \iff -( \evw_{j} + \evr_{j})\cdot \sign(\overline{\evv}_{j}) \geq |\overline{\evv}_{j}|.\notag
\end{align}
Furthermore, for any $\epsilon \in (0, 1)$, we have
\begin{align}
    &\, \{j\in [N]: \,\, - \evr_{j}\cdot \sign(\overline{\evv}_{j}) < \epsilon |\overline{\evv}_{j}| \textnormal{  and  } -\evw_{j} \cdot \sign(\overline{\evv}_{j}) < (1 - \epsilon) |\overline{\evv}_{j}|\} \notag \\
    \subseteq &\, \{j\in [N]:\,\, -( \evw_{j} + \evr_{j})\cdot \sign(\overline{\evv}_{j}) <|\overline{\evv}_{j}|\}.\notag 
\end{align}
Therefore the following inclusion holds,
\begin{align}
    \sM \subseteq &\, \{j\in [N]: -(\evv_{j} - \overline{\evv}_{j})\cdot \sign(\overline{\evv}_{j}) \geq |\overline{\evv}_{j}|\} =  \{j\in [N]: -( \evw_{j} + \evr_{j})\cdot \sign(\overline{\evv}_{j}) \geq |\overline{\evv}_{j}|\} \notag \\
    \subseteq &\, \{j\in [N]: - \evr_{j}\cdot \sign(\overline{\evv}_{j}) \geq \epsilon |\overline{\evv}_{j}|\} \cup \{ j\in [N]: -\evw_{j} \cdot \sign(\overline{\evv}_{j}) \geq (1 - \epsilon) |\overline{\evv}_{j}|\} \notag\\
    \subseteq &\, \{j\in [N]: |\evr_{j}| \geq \epsilon |\overline{\evv}_{j}|\} \cup \{j\in [N]: -\evw_{j} \cdot \sign(\overline{\evv}_{j}) \geq (1 - \epsilon) |\overline{\evv}_{j}|\}.\notag
\end{align}
For any $\epsilon \in (0, 1)$, define the indicator variable $\rR_{j}(\epsilon) \coloneqq \mathds{1}\{|\evr_{j}| \geq \epsilon |\overline{\evv}_{j}|\}$ for each $j\in [N]$. Then a simple union bound implies that
\begin{align}
    N^{-1}\E |\sM| \leq &\, \frac{1}{N}\sum_{j=1}^{N} \E\rR_{j}(\epsilon) + \frac{1}{N}\sum_{j=1}^{N} \E \rT_{j}(\epsilon).\notag
\end{align}
We define the indicator variable $\rR(\epsilon) \coloneqq \mathds{1}\{\|\vr\|_{p} < \epsilon^{2}N^{1/p} \delta_{N}\}$. By assumption, there exists some constant $C>0$ such that $\P(\rR(\epsilon) = 0) \leq C e^{-p/\epsilon}$. Then on the event $\rR(\epsilon) = 1$, we have
\begin{align}
    \frac{1}{N}\sum_{j=1}^{N} \rR_{j}(\epsilon) \leq \frac{1}{N}\sum_{j=1}^{N} \frac{|\evr_{j}|^{p}}{(\epsilon \delta_{N})^{p}} = \frac{\|\vr\|_{p}^{p}}{N (\epsilon \delta_{N})^{p}} \leq \frac{(\epsilon^{2} N^{1/p} \delta_{N})^{p}}{N(\epsilon \delta_{N})^{p}} = \epsilon^{p}.\notag
\end{align}
By taking expectation on both sides, we have
\begin{align}
    \frac{1}{N}\sum_{j=1}^{N} \E\rR_{j}(\epsilon) \leq &\, \epsilon^{p}\cdot \P(\rR(\epsilon) = 1) + 1\cdot \P(\rR(\epsilon) = 0) \leq \epsilon^{p} + C e^{-p/\epsilon} \leq (C + 1)e^{-p\log(1/\epsilon)},\notag
\end{align}
where the last inequality holds $-\log(\epsilon) \leq 1/\epsilon$ for any $\epsilon \in (0, 1)$. Therefore, we have
\begin{align}
    \log\Big(\E N^{-1} |\sM| \Big) \leq &\, \log\bigg( (C + 1)e^{-p\log(1/\epsilon)} + \frac{1}{N}\sum_{j=1}^{N} \P \big( \rT_{j}(\epsilon) = 1 \big) \bigg).\notag\\
    \leq &\, \log(2(C + 1)) + \max\bigg\{ -p\log(1/\epsilon),\,\, \log\bigg(\frac{1}{N}\sum_{j=1}^{N} \P \big( \rT_{j}(\epsilon) = 1 \big)\bigg) \bigg\}.\notag
\end{align}
By dividing $p$ and taking $\limsup$ on both sides, the desired result follows if we further let $\epsilon \to 0$.
\end{proof}

\section{Analysis of spectral method via normalized graph Laplacian}\label{app:achievabilityLap}
This section is devoted to the performance analysis of Algorithm \ref{alg:spectral_partition_binary} when the input is normalized graph Laplacian $\rmL = \rmD^{-1/2} \rmA \rmD^{-1/2}$. We will show that the spectral method can achieve the same performance as the adjacency matrix $\rmA$ in \Cref{app:achievability_matrices_binary}. For the sake of simplicity, we only outline the differences in the proofs.

\subsection{Generalized eigenvalue problem}
We define $\rD_{v} = \sum\limits_{j = 1}^{|\gV|}\ermA_{vj}$ for each $v\in \gV$, and the diagonal matrix $\rmD = \mathrm{diag}\{ \rD_{1}, \ldots, \rD_{N}\}$. By definition, $\rD_{v} \geq 0$ for each $v\in \gV$. We say $(\lambda, \rvu)$ with $\|\rvu\|_2 = 1$ is an eigenpair of $(\rmA, \rmD)$ if 
\begin{align}\label{eqn:generaleigenpair}
    \rmA \rvu = \lambda \rmD \rvu.
\end{align}
With a slightly abuse of notation, let $\rmD^{-1}$ denote the \emph{Moore-Penrose} inverse of $\rmD$, that is, $(\ermD^{-1})_{vv} = (\rD_{v})^{-1}$ if $\rD_{v} \neq 0$, otherwise $(\ermD^{-1})_{vv} = 0$, which happens only if $\ermA_{vj} = 0$ for each $j\in \gV$.

The normalized Laplacian is defined as $\rmL = \rmD^{-1/2}\rmA\rmD^{-1/2}$, where similarly, $(\rmD^{-1/2})_{vv} = (\rD_{v})^{-1/2}$ if $\rD_{v} \neq 0$, otherwise $(\rmD^{-1/2})_{vv} = 0$. Meanwhile,
\begin{align}
    \rmA \rvu = \lambda \rmD \rvu \iff \rmD^{-1/2}\rmA\rmD^{-1/2} \, \rmD^{1/2}\rvu = \lambda \rmD^{1/2} \rvu \iff \rmL\rmD^{1/2}\rvu = \lambda \rmD^{1/2} \rvu,\notag
\end{align}
where $\rmD^{1/2} \rvu$ and $\rvu$ share the same sign entrywisely, then it suffices to study the generalized eigenvalue problem \eqref{eqn:generaleigenpair} if one is interested in the spectral properties of graph Laplacian $\rmL$.

Analogously, following the ideas in \Cref{app:achievability_matrices_binary}, the expected and the leave-one-out versions are
\begin{align}
    &\, \rD^{\star}_{vv} = \sum\limits_{\rJ = 1}^{|\gV|} A^{\star}_{vj}, \quad \rmD^{\star} = \mathrm{diag}(\rD^{\star}_{11}, \ldots, \rD^{\star}_{nn}), \quad \rmL^{\star} = (\rmD^{\star})^{-1/2}\rmA^{\star}(\rmD^{\star})^{-1/2},\notag\\
    &\, \rD^{(-v)}_{uu} = \sum\limits_{\rJ = 1}^{|\gV|} A^{(-v)}_{uj}, \,\, \rmD^{(-v)} = \mathrm{diag}(\rD^{(-v)}_{11}, \ldots, \rD^{(-v)}_{nn}), \,\, \rmL^{(-v)} = (\rmD^{(-v)})^{-1/2}\rmA^{(-v)}(\rmD^{(-v)})^{-1/2}.\notag
\end{align}

In this section, let $\{(\lambda_i, \rvu_i)\}_{i=1}^{|\gV|}$, $\{(\lambda_i^{\star}, \rvu_i^{\star})\}_{i=1}^{|\gV|}$ and $\{(\lambda_i^{(-v)}, \rvu_i^{(-v)})\}_{i=1}^{|\gV|}$ denote the eigenpairs of the following generalized eigenvalue problems, respectively
\begin{align}
    \rmA \rvu = \lambda \rmD \rvu, \quad \rmA^{\star} \rvu^{\star} = \lambda \rmD^{\star} \rvu, \quad \rmA^{(-v)} \rvu = \lambda \rmD^{(-v)} \rvu, \notag
\end{align}
where the eigenvalues are sorted in the decreasing order. We then introduce the following analogous results.
\begin{lemma}\label{lem:Dmin}
   Suppose that $\D_{\mathrm{AM}} = 1 + \delta$ for some constant $\delta >0$, then $\rD_{v} \geq \epsilon \rho_{N}$ for each $v\in \gV$ with probability at least $1 - e^{-(\D_{\mathrm{AM}} + \delta/2)q_{N}}$ for some constant $\epsilon >0$ depending on $\delta$, $\alpha$ and $\beta$ in \eqref{eqn:alpha_beta_def}. Furthermore, $\rD_{\min} = \min_{v\in \gV}\rD_{v} \gtrsim \rho_{N}$.
\end{lemma}
\begin{proof}[Proof of Lemma ~\ref{lem:Dmin}]
    The proof follows similarly as the proof for Lemma ~\ref{lem:LDP_AM}. For some $\epsilon >0$, $\P(\rD_{v} < \epsilon \rho_{N}) \leq \P(\sqrt{N} \ervy_{v}(\rmA \rvu^{\star}_2)_v \leq \delta \rho_{N}) \leq e^{-(\D_{\mathrm{AM}} + \delta/2)q_{N}}$ when $\delta = \D_{\mathrm{AM}} - 1 > 0$.
\end{proof}

\begin{corollary}\label{cor:concentrateNL}
    Following the conventions in \Cref{thm:concentration}, suppose that $\D_{\mathrm{AM}} = 1 + \delta$ for some constant $\delta >0$, then with probability at least $1 - e^{-(\D_{\mathrm{AM}} + \delta/2)q_{N}}$,
\begin{align}
    \|\rmD^{-1}\rmA - (\rmD^{\star})^{-1} \rmA^{\star} \| \leq \const_{\eqref{cor:concentrateNL}}/ \sqrt{\rho_{N}}. \notag 
\end{align}
\end{corollary}
\begin{proof}[Proof of Corollary \ref{cor:concentrateNL}]
    By triangle inequality
    \begin{align}
        &\, \|\rmD^{-1}\rmA - (\rmD^{\star})^{-1} \rmA^{\star} \| \leq \|\rmD^{-1}\rmA - \rmD^{-1}\rmA^{\star} + \rmD^{-1}\rmA^{\star} - (\rmD^{\star})^{-1}\rmA^{\star} \|\notag\\
        \leq &\, \|\rmD^{-1}\|\cdot \|\rmA - \rmA^{\star}\| + \|(\rmD^{-1} - (\rmD^{\star})^{-1} )\rmA^{\star} \|\notag
    \end{align}
    where $\|\rmD^{-1}\| = (\rD_{\min})^{-1} \lesssim \rho_{N}^{-1}$ by Lemma ~\ref{lem:Dmin} and $\|\rmA - \rmA^{\star}\| \lesssim \sqrt{\rho_{N}}$ by \Cref{thm:concentration}. For the second term, recall the definitions of $\alpha, \beta$ in \eqref{eqn:alpha_beta_def} and $\rho_{N}$ in \eqref{eqn:rho_binary}, then with probability at least $1 - e^{-(\D_{\mathrm{AM}} + \delta/2)q_{N}}$,
    \begin{align*}
        &\, \|(\rmD^{-1} - (\rmD^{\star})^{-1} )\rmA^{\star}\|^2 \leq \|(\rmD^{-1} - (\rmD^{\star})^{-1} )\rmA^{\star}\|^2_{\frob}\\
        \leq &\, \sum\limits_{i=1}^{|\gV|} (\rD^{-1}_{ii} - (\rD^{\star}_{ii})^{-1})^2  \sum\limits_{j=1}^{|\gV|} \alpha^2 \leq \frac{N\alpha^2}{ \rD_{\min}^2 \cdot ( \rD^{\star}_{ii})^2} \sum\limits_{i=1}^{|\gV|} ( \rD_{ii} - \rD^{\star}_{ii})^2\\
        \lesssim &\, \frac{N \alpha^2}{ \rD_{\min}^2 \cdot (\rD^{\star}_{ii})^2} \sum\limits_{i=1}^{|\gV|} \sum\limits_{k=1}^{|\gV|} (A_{ik} - A^{\star}_{ik})^2 \leq \frac{N \alpha^2}{ \rD_{\min}^2 \cdot (\rD^{\star}_{ii})^2} \|\rmA - \rmA^{\star}\|^2_{\frob}\\
        \leq &\, \frac{N^2\alpha^2}{ \rD_{\min}^2 \cdot (\rD^{\star}_{ii})^2} \|\rmA - \rmA^{\star}\|^2_{2} \lesssim \frac{1}{\rho_{N}},
    \end{align*}
    where the last inequality holds by \Cref{thm:concentration} since $\rD_{\min} \gtrsim \rho_{N}$ by Lemma ~\ref{lem:Dmin}, $\rD^{\star}_{ii} = \rho_{N}$ and $N \alpha \asymp \rho_{N}$. The desired result follows by combining the two estimates above.
\end{proof}

\begin{lemma}\label{lem:operatorNormNL}
    Deterministically, $ \|(\rmD^{\star})^{-1}\rmA^{\star}\|_{2\to \infty} \lesssim 1/\sqrt{N}$. Suppose that $\D_{\mathrm{AM}} = 1 + \delta$ for some $\delta >0$, then with probability at least $1-e^{-(\D_{\mathrm{AM}} + \delta/2)q_{N}}$,
        \begin{align}
             \|\rmD^{-1}\rmA\|_{2\to \infty} \lesssim 1/\sqrt{\rho_{N}} + 1/\sqrt{N}.\notag
        \end{align}
    For each $v\in \gV$, the following holds with probability at least $1 - e^{-(\D_{\mathrm{AM}} + \delta/2)q_{N}}$
    \begin{align}
        &\, \|\rmD^{-1}\rmA - (\rmD^{(-v)})^{-1}\rmA^{(-v)}\| \lesssim \|\rmD^{-1}\rmA\|_{2 \to \infty}, \notag \\ 
        &\, \|(\rmD^{(-v)})^{-1}\rmA^{(-v)} - (\rmD^{\star})^{-1}\rmA^{\star}\| \leq 2/\sqrt{\rho_{N}} + 1/\sqrt{N}.\notag
    \end{align}
\end{lemma}
\begin{proof}[Proof of Lemma ~\ref{lem:operatorNormNL}]
       By Lemma ~\ref{lem:operatorNormA}, $\|\rmA^{\star}\|_{2\to \infty} \leq \rho_{N}/\sqrt{N}$, the first result then follows since $\rmD^{\star} = \mathrm{diag}(\rho_{N}, \ldots, \rho_{N})$. The remaining results follows the same arguments as in the proof of Lemma ~\ref{lem:operatorNormA}, and the lower bound of the minimum degree in $\rmD$, as well as the lower bound of non-zero entries in $\rmD^{(-v)}$ (for the sake of \emph{Moore-Penrose} inverse).
\end{proof}

\begin{corollary}\label{cor:eigenvalueApproxLap}
     For each $i\in \gV$, the following holds with probability at least $1-N^{-\D_{\mathrm{AM}} + \delta/2}$,
    \begin{align}
        |\lambda_i - \lambda_i^{\star}| \lesssim \frac{1}{\sqrt{\rho_{N}}}, \quad |\lambda_i - \lambda_i^{(-v)}| \lesssim \frac{1}{\sqrt{\rho_{N}}} + \frac{1}{\sqrt{N}}, \quad |\lambda_i^{\star} - \lambda_i^{(-v)}| \lesssim \frac{2}{\sqrt{\rho_{N}}} + \frac{1}{\sqrt{N}}.
    \end{align}
    In particular in our model \ref{def:non_uniform_HSBM_binary}, with probability at least $1 - e^{-(\D_{\mathrm{GH}} + \delta/2)q_{N}}$,
    \begin{align}
        &\, |\lambda_i| \asymp |\lambda^{\star}_i| \asymp |\lambda^{(-v)}_i| \asymp 1, \quad  i = 1, 2.\notag \\
        &\, |\lambda_i| \asymp |\lambda^{\star}_i| \asymp |\lambda^{(-v)}_i| \asymp \frac{1}{\sqrt{\rho_{N}}}, \quad  i = 3, \ldots, N.\notag
    \end{align}
\end{corollary}
\begin{proof}[Proof of Corollary \ref{cor:eigenvalueApproxLap}]
    By Weyl's inequality \ref{lem:weyl} and Corollary \ref{cor:concentrateNL}, for any $i\in \gV$,
    \begin{align}
        |\lambda_i - \lambda_i^{\star}| \leq \|\rmD^{-1}\rmA - (\rmD^{\star})^{-1}\rmA^{\star}\| \lesssim 1/\sqrt{\rho_{N}}.\notag
    \end{align}
    The remaining result follows as well by Lemma ~\ref{lem:operatorNormNL} and Lemma ~\ref{lem:weyl}.
\end{proof}

\subsection{Proof of algorithm functionality}
Following the approach in \Cref{sec:achievability_matrices_binary}, we use $\frac{1}{\lambda_2} (\E \rmD)^{-1} \rmA \rvu_2$ as the pivot vector. Proposition \ref{prop:entrywiseLap} provides an analogous result to Proposition \ref{prop:entrywise_diff_A}. 
Using the analytical framework established for proving \Cref{thm:achievability_matrices}, we show that \Cref{alg:spectral_partition_binary} achieves identical performance regardless of whether the input is the normalized graph Laplacian $\rmL$ or the adjacency matrix $\rmA$.

\begin{proposition}\label{prop:entrywiseLap}
Suppose $\D_{\mathrm{AM}} = 1 + \delta$ for some constant $\delta >0$, then with probability at least $1 - e^{-(\D_{\mathrm{AM}} + \delta/2)q_{N}}$
    \begin{align}
        \sqrt{N} \cdot \min_{s\in \{\pm 1\}} \|s\rvu_2 - \frac{1}{\lambda_2} (\E \rmD)^{-1} \rmA \rvu_2\|_{\infty} \leq C/\log(\rho_{N}).
    \end{align}
\end{proposition}
\begin{proof}[Proof of Proposition \ref{prop:entrywiseLap}]
We first show that the eigenvector $\rvu_2$ is delocalized, i.e., $\|\rvu_2\|_{\infty}\lesssim N^{-1/2}$. For convenience in notation, we denote $\lambda = \lambda_2$, $\lambda^{\star} = \lambda_2^{\star}$, $\lambda^{(-v)} = \lambda_2^{(-v)}$ and $\rvu = \rvu_2$, $\rvu^{\star} = \rvu_2^{\star}$, $\rvu^{(-v)} = \rvu_2^{(-v)}$. By Corollary \ref{cor:eigenvalueApproxLap}, we then have
\begin{align*}
    &\, \|\rvu\|_{\infty} = \|s\rmD^{-1}\rmA \rvu/\lambda \|_{\infty} \\
    \leq &\, |\lambda|^{-1}\cdot \Big( \|\rmD^{-1}\rmA (s\rvu - \rvu^{\star}) \|_{\infty} + \|\rmD^{-1}\rmA\rvu^{\star} \|_{\infty}\Big)\\
    \lesssim &\, |\lambda^{\star}|^{-1}\cdot \Big( \max_{v\in \gV} |\rmD^{-1}\rmA_{v:}(s\rvu - \rvu^{\star})| + \|\rmD^{-1}\rmA\rvu^{\star} \|_{\infty} \Big)\\
    \lesssim &\, |\lambda^{\star}|^{-1}\cdot \Big( \|\rmD^{-1}\rmA\|_{2\to \infty}\cdot \|s\rvu - s^{(-v)}\rvu^{(-v)}\|_2 + \max_{v\in \gV} |\rmD^{-1}\rmA_{v:}(s^{(-v)}\rvu^{(-v)} - \rvu^{\star})| +  \|\rmD^{-1}\rmA\rvu^{\star} \|_{\infty} \Big)\\
    \lesssim &\, \Big[ \big( \frac{1}{\sqrt{\rho_{N}}} + \frac{1}{\sqrt{N}} \big) \|\rvu\|_{\infty} + \frac{1}{\log(\rho_{N})}\cdot \Big( \| \rvu \|_{\infty} + \frac{1}{\sqrt{N}} \Big)  + \frac{1}{\sqrt{N}} \Big]
\end{align*}
where in the second to last line, the first term could be bounded by Lemma ~\ref{lem:firstApproxLap} and Lemma ~\ref{lem:operatorNormNL}, while the last two terms can still be bounded by arguments combining lower bound of entries in $\rmD$ and Lemmas \ref{lem:secondApproximation}, \ref{lem:thirdApproximation}, respectively. By rearranging terms, one could prove $\|\rvu\|_{\infty} \lesssim 1/\sqrt{N}$. Similarly, with the help from Lemmas \ref{lem:firstApproximation}, \ref{lem:secondApproximation}, \ref{lem:thirdApproximation}
\begin{align*}
    &\, \Big\| s\rvu - \frac{(\E\rmD)^{-1} \rmA \rvu^{\star} }{\lambda^{\star}} \Big\|_{\infty} \\
    = &\, \Big\| s\frac{\rmD^{-1} \rmA\rvu }{\lambda} - \frac{\rmD^{-1}\rmA\rvu^{\star} }{\lambda} + \frac{\rmD^{-1}\rmA\rvu^{\star}}{\lambda} - \frac{(\E\rmD)^{-1} \rmA\rvu^{\star} }{\lambda^{\star}} \Big\|_{\infty}\\
    \lesssim &\, |\lambda^{\star}|^{-1} \|\rmD^{-1}\rmA(s\rvu - \rvu^{\star})\|_{\infty} + |\lambda^{\star}|^{-1} \|\rmD^{-1} - (\E \rmD)^{-1}\|\cdot \|\rmA \rvu^{\star}\| \\
    \lesssim &\, |\lambda^{\star}|^{-1}\cdot \Big( \|\rmD^{-1}\rmA\|_{2\to \infty}\cdot \|s\rvu - s^{(-v)}\rvu^{(-v)}\|_2 + \max_{v\in \gV} |\rmD^{-1}\rmA_{v:}(s^{(-v)}\rvu^{(-v)} - \rvu^{\star})| \Big) \\
    &\,  + |\lambda^{\star}|^{-1}\frac{|\rD_{\min} - \E \rD_{v}|}{|\rD_{\min}|\cdot |\E \rD_{v}|}\cdot \|\rmA \rvu^{\star}\| \\
    \lesssim &\, \Big[ \big( \frac{1}{\sqrt{\rho_{N}}} + \frac{1}{\sqrt{N}} \big) \|\rvu\|_{\infty} + \frac{1}{\log(\rho_{N})}\cdot \Big( \| \rvu \|_{\infty} + \frac{1}{\sqrt{N}} \Big)  + \frac{1}{\sqrt{N}} \Big] +  \frac{\sqrt{\rho_{N}}}{\rho_{N}\cdot \rho_{N}} \cdot \frac{\rho_{N}}{\sqrt{N}} \\
    \lesssim &\, \frac{1}{\sqrt{N}\log(\rho_{N})},
\end{align*}
where in the second to last line, the difference $|\lambda - \lambda^{\star}|$ is controlled by Corollary \ref{cor:eigenvalueApproxLap}, and a simple application of Chernoff proves that $|\rD_{v} - \E \rD_{v}|\leq \sqrt{\rho_{N}}$ with high probability for each $v\in \gV$.
\end{proof}
The thing remaining is to show that the Algorithm \ref{alg:spectral_partition_binary} achieves exact recovery when $\D_{\mathrm{AM}}> 1$, which follows exactly as the proof of \Cref{thm:achievability_matrices}.

\begin{lemma}\label{lem:firstApproxLap}
Under the setting in Proposition \ref{prop:entrywiseLap},
    \begin{align}
        \|s\rvu - s^{(-v)}\rvu^{(-v)}\|_2 \lesssim \|\rvu\|_{\infty}.
    \end{align}
\end{lemma}
\begin{proof}[Proof of Lemma ~\ref{lem:firstApproxLap}]
    Let $\delta = \min_{i\neq 2}|\lambda_{i}^{(-v)} - \lambda|$ denote the absolute eigen gap, then by generalized Davis-Kahan Lemma ~\ref{lem:generalized_Davis_Kahan} where we take $\rmA^{(-v)}$, $\rmD^{(-v)}$, $\Lambda = \lambda_{2}^{(-v)}$, $\rmU = s^{(-v)}\rvu^{(-v)}$, $\rmP = \rmI_{N} - \rmU\rmU^{\sT}$, $\widehat{\lambda} = \lambda$, $\widehat{\rvu}= s\rvu$, we have
    \begin{align}
        &\, \| s^{(-v)}\rvu^{(-v)} - s\rvu \|_2 \leq \sqrt{2}\|\rmP \rvu\|_2 \notag \\
        \leq &\, \frac{\sqrt{2}\cdot \|(\rmD^{(-v)})^{-1}\rmA^{(-v)}  - \lambda \rmI_{N} )\rvu\|_2}{\delta} = \frac{\sqrt{2} \cdot  \|\big[ (\rmD^{(-v)})^{-1}\rmA^{(-v)}  - \rmD^{-1}\rmA \big]\rvu\|_2}{\delta} \notag
    \end{align}
    By Lemma ~\ref{lem:weyl} and Corollary \ref{cor:eigenvalue_approx_adj}, the following holds with probability at least $1 - O(N^{-10})$,
    \begin{align*}
        \delta = &\, \min_{i\neq 2}|\lambda_{i}^{(-v)} - \lambda| = \min \big\{ |\lambda_3^{(-v)} - \lambda|, \,\, |\lambda_1^{(-v)} - \lambda|\big\}  \\
        \geq &\, \min \big\{ |\lambda_3 - \lambda|, \,\,|\lambda_1 - \lambda|\big\} - \|(\rmD^{(-v)})^{-1}\rmA^{(-v)}  - \rmD^{-1}\rmA\|\\
        \geq &\, \min \big\{ |\lambda^{\star}_3 - \lambda^{\star}|, \,\,|\lambda^{\star}_1 - \lambda^{\star}|\big\} - 2\|\rmD^{-1}\rmA - (\rmD^{\star})^{-1}\rmA^{\star}\| - \|(\rmD^{(-v)})^{-1}\rmA^{(-v)}  - \rmD^{-1}\rmA\| \\
        \gtrsim &\, 1,
    \end{align*}
    where the last line holds by Corollary \ref{cor:concentrateNL}.
    Let $\rvw = \big[ \rmD^{-1}\rmA - (\rmD^{(-v)})^{-1}\rmA^{(-v)} \big]\rvu$. For $i\neq v$,
    \begin{align}
        \| \big[ \rmD^{-1}\rmA - (\rmD^{(-v)})^{-1}\rmA^{(-v)} \big]_{i\cdot}\|_{1} = \indi{\rD_{ii} \neq 0} \sum\limits_{\ell \in \sL} (\ell - 1)\ermA_{iv}\cdot (\rD_{ii})^{-1}. \notag
    \end{align}
    Then
    \begin{align*}
        |\ervw_{v}| =&\, |(\rmD^{-1}\rmA \rvu)_{v}| = |\lambda|\cdot |\ervu{v}| \leq |\lambda| \cdot \|\rvu\|_{\infty},\\
        |\ervw_{i}| = &\, |\big[ \rmD^{-1}\rmA \rvu - (\rmD^{(-v)})^{-1}\rmA^{(-v)}\rvu \big]_{i}|\leq \| \big[ \rmD^{-1}\rmA - (\rmD^{(-v)})^{-1}\rmA^{(-v)} \big]_{i\cdot}\|_{1} \cdot \|\rvu\|_{\infty}
    \end{align*}
    Again by Corollary \ref{cor:concentrateNL}, Corollary \ref{cor:eigenvalueApproxLap}, with probability at least $1 - O(N^{-10})$,
    \begin{align*}
        \|(\rmA - \rmA^{(-v)})\rvu\|_2^2 \leq &\, \bigg( |\lambda|^2 + \indi{ \rD_{ii} \neq 0} \sum\limits_{\ell \in \sL} (\ell - 1)\ermA_{iv}\cdot (\rD_{ii})^{-1} \bigg) \cdot \|\rvu\|_{\infty}^2 \\
        \lesssim &\, \big( |\lambda^{\star}|^2 + \|\rmD^{-1}\rmA\|_{2 \to \infty} ) \cdot  \|\rvu\|_{\infty}^2 \lesssim \cdot \|\rvu\|_{\infty}^2.
    \end{align*}
    Therefore, by plugging in the lower bound for $\delta$ \eqref{eqn:delta_lower} into \Cref{eqn:u-v-u_upper}, we have
    \begin{align}
        \| s^{(-v)}\rvu^{(-v)} - s\rvu \|_2 \lesssim \|\rvu\|_{\infty}, \notag
    \end{align}
    completing the proof.
\end{proof}

\section{Performance analysis of SDP}\label{app:SDP}

This section is devoted to the performance analysis of the following \emph{semidefinite programming} (SDP).
\begin{align}
    \max\,\,&\,\<\rmA, \rmX\> \label{eqn:SDP}\\
    \textnormal{subject to } &\, \rmX \succeq 0, \,\, \<\rmX , \ones \ones^{\sT}\> = 0,\,\, \ermX_{ii} = 1,\quad \forall i\in\gV. \notag
\end{align}



\begin{algorithm}
\caption{\textbf{SDP via adjacency matrix}}\label{alg:SDP}
\KwData{The adjacency matrix $\rmA$}

Solve the SDP \eqref{eqn:SDP} with $\widehat{\Sigma}$ denoting its solution.

Let $\widehat{\rvu}$ denote the eigenvector corresponding to the largest eigenvalue of $\widehat{\Sigma}$.

\KwResult{$\widehat{\gV}_{+} = \{v: \sign(\widehat{\ervu}_v) >0 \}$ and $\widehat{\gV}_{-} = \{v: \sign(\widehat{\ervu}_v) < 0 \}$.}
\end{algorithm}

\begin{theorem}\label{thm:phaseTransitionSDP}
    When the adjacency matrix $\rmA$ is used as the input, Algorithm \ref{alg:SDP} achieves exact recovery if and only if $\D_{\mathrm{AM}} > 1$.
\end{theorem}

The proof of \Cref{thm:phaseTransitionSDP} is included here for completeness, so that one could witness the appearance of $\D_{\mathrm{AM}}$, while the discrepancy between $\D_{\mathrm{AM}}$ and $\D_{\mathrm{GH}}$ could be seen straightforwardly.


\subsection{Necessity}
The proof strategy follows the same idea as \cite[Theorem 8]{Kim2018StochasticBM}. Consider the Min-Bisection problem
\begin{align}
    \max\,\,&\, \sum\limits_{1\leq i< j \leq N} \ervx_i \ermA_{ij} \ervx_j \label{eqn:minBisection}\\
    \textnormal{subject to } &\, \rvx \in \{\pm 1\}^{|\gV|}, \,\, \ones^{\sT}\rvx = 0. \notag
\end{align}
Note that any feasible $\rvx$ to \eqref{eqn:minBisection} corresponds to a feasible solution $\rmX = \rvx\rvx^{\sT}$ to \eqref{eqn:SDP}. Consequently, \eqref{eqn:SDP} is a relaxation of \eqref{eqn:minBisection} \cite{Goemans1995ImprovedAA}. Let $\widehat{\Sigma}$ and $\widehat{\rvy}_{\mathrm{MB}}$ denote the solutions to \eqref{eqn:SDP} and \eqref{eqn:minBisection} respectively. The probability that SDP \eqref{eqn:SDP} fails exact recovery is then lower bounded by
\begin{align}
    \P_{\mathrm{fail}} = \P(\widehat{\Sigma} \neq \rvy \rvy^{\sT}) \geq \P(\widehat{\rvy}_{\mathrm{MB}} \neq \pm \rvy).\notag
\end{align}
Then it suffices to consider the lower bound of the probability that $\widehat{\rvy}_{\mathrm{MB}}$ fails exact recovery. Given the adjacency matrix $\rmA$ and the assignment vector $\rvx \in \{\pm 1\}^{|\gV|}$ with $\ones^{\sT}\rvy = 0$, define
\begin{align}
    g(\rvx) = \sum\limits_{1\leq i< j \leq N} \ervx_i \ermA_{ij} \ervx_j.\notag
\end{align}
The failure of $\widehat{\rvy}_{\mathrm{MB}}$ implies the existence of some $\widetilde{\rvy}$ such that $g(\widetilde{\rvy}) \geq g(\rvy)$, i.e., 
\begin{align}
    \P(\widehat{\rvy}_{\mathrm{MB}} \neq \pm \rvy) = \P\big( \exists \widetilde{\rvy} \textnormal{ s.t. } g(\widetilde{\rvy}) \geq g(\rvy) \big).\notag
\end{align}

Following the same idea as the proof of \Cref{thm:impossibility_binary},  one could construct such $\widetilde{\rvy}$ by swapping signs of $a\in \gV_{+}$ and $b\in \gV_{-}$ but fixing others in $\rvy$. Then we have the decomposition
\begin{align}
    &\,g(\rvy) - g(\widetilde{\rvy}) = \sum\limits_{1\leq i < j \leq N} \ermA_{ij}\ervy_{i} \ervy_{j} - \sum\limits_{1\leq i < j \leq N} \ermA_{ij}\widetilde{\ervy}_{i} \widetilde{\ervy}_{j} \notag \\
    =&\, \sum\limits_{1\leq i < j \leq N} \sum\limits_{\ell \in \sL} \sum\limits_{e\supset\{i, j\}}\etA^{(\ell)}_{e} ( \ervy_{i} \ervy_{j} - \widetilde{\ervy}_{i} \widetilde{\ervy}_{j}) = \sum\limits_{\ell \in \sL} \sum\limits_{e\in \gE_{\ell}} \etA^{(\ell)}_{e} \sum\limits_{\{i, j\}\subset e} (\ervy_{i} \ervy_{j} - \widetilde{\ervy}_{i} \widetilde{\ervy}_{j}) \notag \\
    =&\, \underbrace{\sum\limits_{\ell \in \sL} \sum\limits_{e\ni a,\,e\not\ni b} \etA^{(\ell)}_{e} \sum\limits_{\{i, j\}\subset e } (\ervy_{i} \ervy_{j} - \widetilde{\ervy}_{i} \widetilde{\ervy}_{j})}_{\rX_{a}} + \underbrace{\sum\limits_{\ell \in \sL} \sum\limits_{e\not\ni a,\,e\ni b} \etA^{(\ell)}_{e} \sum\limits_{\{i, j\}\subset e} (\ervy_{i} \ervy_{j} - \widetilde{\ervy}_{i} \widetilde{\ervy}_{j})}_{\rX_{b}} \notag \\
    &\, + \underbrace{\sum\limits_{\ell \in \sL} \sum\limits_{e\supset\{ a, b \} } \etA^{(\ell)}_{e} \sum\limits_{\{i, j\}\subset e} (\ervy_{i} \ervy_{j} - \widetilde{\ervy}_{i} \widetilde{\ervy}_{j})}_{\rX_{ab}}. \notag
\end{align}
Furthermore, the contributions from $\rX_{a}$, $\rX_{b}$ and $\rX_{ab}$ are independent due to the edge independence. For $\rX_{ab}$, note that the terms irrelevant to $a$ and $b$ cancel out, since $\rvy$ and $\widetilde{\rvy}$ are the same except $a, b$. Then for each $e\in \gE^{(r)}_{\ell}(a)$ and $e\supset \{a, b\}$,
\begin{align*}
    &\, \sum\limits_{\{i, j\}\subset e } (\ervy_{i} \ervy_{j} - \widetilde{\ervy}_{i} \widetilde{\ervy}_{j}) \\
    =&\, (\ervy_{a}  \ervy_{b}  - \widetilde{\ervy}_{a} \widetilde{\ervy}_{b} ) + \sum\limits_{\{i, j\}\subset e \setminus \{a, b\}} (\ervy_{i} \ervy_{j} - \widetilde{\ervy}_{i} \widetilde{\ervy}_{j}) \\
    =&\, (\ervy_{a}  \ervy_{b}  - \widetilde{\ervy}_{a} \widetilde{\ervy}_{b} ) + \ervy_{a}  \sum\limits_{j\in e\setminus \{a, b\}} \ervy_{j} - \widetilde{\ervy}_{a}  \sum\limits_{j\in e\setminus \{a, b\}} \widetilde{\ervy}_{j} + \ervy_{b}  \sum\limits_{j\in e\setminus \{a, b\}} \ervy_{j} - \widetilde{\ervy}_{b}  \sum\limits_{j\in e\setminus \{a, b\}} \widetilde{\ervy}_{j} \\
    =&\,(\ervy_{a}  \ervy_{b}  - \widetilde{\ervy}_{a} \widetilde{\ervy}_{b} ) + 2\ervy_{a}  \sum\limits_{i\in e\setminus \{a, b\}} \ervy_{i}  + 2\ervy_{b}  \sum\limits_{j\in e\setminus \{a, b\}} \ervy_{j} \\
    =&\, 0 + 2\ervy_{a}  \big(\ell - 2 - 2r \big) + 2\ervy_{b}  \big( \ell - 2 - 2r \big) = 0.
\end{align*}
where in the last line, we used the facts that $\ervy_{a} \ervy_{b} = \widetilde{\ervy}_{a} \widetilde{\ervy}_{b}$ and $\ervy_{a} + \ervy_{b}=0$.

For $\rX_{a}$, the contributions of the terms irrelevant to $a$ cancels, since $\rvy$ and $\widetilde{\rvy}$ are the same except $a, b$. If $e\in \gE^{(r)}_{\ell}(a)$, then
\begin{align}
    \sum\limits_{\{i, j\}\subset e } (\ervy_{i} \ervy_{j} - \widetilde{\ervy}_{i} \widetilde{\ervy}_{j}) = 2\ervy_{a}  \sum\limits_{i\in e\setminus \{a\}} \ervy_{i} = 2\ervy_{a}  \big(\ell - 1 - 2r \big).\notag
\end{align}
The same argument holds for $\rX_{b}$ similarly. Denote $\rF_{v} \coloneqq \mathbbm{1}\{\rX_v \leq 0\}$, then $\{ \rF_{a} = 1 \} \cap \{ \rF_{b} = 1\} \iff g(\rvy) \leq g(\widetilde{\rvy})$, and it suffices to show that $\{ \rF_{a} = 1 \}$ happens with high probability. However, $\rF_{a}$ and $\rF_{a^{\prime}}$ are not independent for distinct $a, a' \in \gV_{+}$ due to the existence of edges containing $a, a'$ simultaneously. Following the same strategy in \Cref{app:impossibility_binary}, let $\gU \subset \gV$ be a subset of vertces of size $\gamma_{N} N$ where $|\gU \cap \gV_{+}| = |\gU \cap \gV_{-}| = \gamma_{N} N/2$ with $\gamma_{N} = q_{N}^{-3}$ and denote $\gW_{\ell}^{(\geq 2)} \coloneqq \{e\in [\gV]^{\ell} \mid e \textnormal{ contains at least 2 vertices in } \gU\} \subset \gE_{\ell}$. We define
\begin{align}
    \rY_{a} =&\, \sum\limits_{\ell \in \sL} \sum\limits_{\substack{e \cap \gU= \{a\} \\ e\not\ni b}} \etA^{(\ell)}_{e} \sum\limits_{\{i, j\}\subset e } (\ervy_{i} \ervy_{j} - \widetilde{\ervy}_{i} \widetilde{\ervy}_{j}),\quad \rZ_{a} = \sum\limits_{\ell \in \sL} \sum\limits_{\substack{e\in \gW_{\ell}^{(\geq 2)} \\e\ni a,\, e\not\ni b}} \etA^{(\ell)}_{e} \sum\limits_{\{i, j\}\subset e } (\ervy_{i} \ervy_{j} - \widetilde{\ervy}_{i} \widetilde{\ervy}_{j}),\notag\\
    \rY_b =&\, \sum\limits_{\ell \in \sL} \sum\limits_{\substack{e \cap \gU= \{b\} \\ e\not\ni a}} \etA^{(\ell)}_{e} \sum\limits_{\{i, j\}\subset e } (\ervy_{i} \ervy_{j} - \widetilde{\ervy}_{i} \widetilde{\ervy}_{j}),\quad  \rZ_b = \sum\limits_{\ell \in \sL} \sum\limits_{\substack{e\in \gW_{\ell}^{(\geq 2)} \\e\ni b,\, e\not\ni a}} \etA^{(\ell)}_{e} \sum\limits_{\{i, j\}\subset e } (\ervy_{i} \ervy_{j} - \widetilde{\ervy}_{i} \widetilde{\ervy}_{j}).\notag
\end{align}
Note that $\rY_{a}, \rY_b, \rZ_{a}, \rZ_b$ are pairwise independent since there is no common edges. Denote events
\begin{align*}
    \rA_{a} =&\, \indi{\rY_{a} \leq -\zeta_{N} q_{N}}, \quad \rJ_{a} = \indi{\rZ_{a} \leq \zeta_{N} q_{N}},\\
    \rB_{b}=&\, \indi{\rY_b \leq -\zeta_{N} q_{N}}, \quad \rJ_{b} = \indi{\rZ_b \leq \zeta_{N} q_{N}},
\end{align*}
where we take $\zeta = (\log q_{N})^{-1}$. Obviously, 
\begin{align*}
    \{ \rA_{a} = 1\} \cap \{\rJ_{a} = 1\} \implies \{\rF_{a} = 1\}, \quad \{ \rB_{b}= 1\} \cap \{\rJ_{b} = 1\} \implies \{\rF_{b} = 1\}.
\end{align*}
Let $\D_{\mathrm{AM}} = 1 - \epsilon$ for some absolute $\epsilon >0$ and $q_{N} = \log(N)$. With proofs deferred later, we have the following inequalities
\begin{align}
    \prod_{a\in \gV_{+}\cap \gU} \P( \rA_{a} = 0) \leq \exp\Big( -\frac{1}{2}N^{\epsilon + o(1)} \Big), \quad \P(\rJ_{a} = 0) \leq N^{-3 + o(1)}.\label{eqn:SDPLDP}
\end{align}
Then the probability of SDP failing exact recovery is lower bounded by
\begin{align*}
    &\, \P_{\mathrm{fail}} \\
    = &\,\P(\widehat{\Sigma} \neq \rvy \rvy^{\sT}) \geq \P(\widehat{\rvy}_{\mathrm{MB}} \neq \pm \rvy) = \P\big( \exists \widetilde{\rvy} \textnormal{ s.t. } g(\widetilde{\rvy}) \geq g(\rvy) \big) \\
    \geq &\, \P\big( \exists a\in \gV_{+},\, b\in \gV_{-} \textnormal{ s.t. } \{ \rF_{a} = 1 \} \cap \{ \rF_{b} = 1 \} \big) = \P\Big( \cup_{a\in \gV_{+} } \{ \rF_{a} = 1 \} \bigcap \cup_{b\in \gV_{-}} \{ \rF_{b} = 1 \} \Big) \\
    =&\, \P\Big( \cup_{a\in \gV_{+} } \{ \rF_{a} = 1 \} \bigcap \cup_{b\in \gV_{-}} \{ \rF_{b} = 1 \} \big| \{ \rJ_{a} = 1 \} \cap \{\rJ_{b} = 1\} \Big) \cdot \P(\{ \rJ_{a} = 1 \} \cap \{\rJ_{b} = 1\} )\\
    \geq &\, \P\Big( \cup_{a\in \gV_{+}\cap \gU} \{ \rA_{a} = 1 \} \bigcap \cup_{b\in \gV_{-} \cap \gU} \{ \rB_{b}= 1 \} \big| \{ \rJ_{a} = 1 \} \cap \{\rJ_{b} = 1\} \Big) \cdot \P(\{ \rJ_{a} = 1\} \cap \{ \rJ_{b} = 1\})\\
    \geq &\, \P\big( \cup_{a\in \gV_{+}\cap \gU} \{ \rA_{a} = 1 \} \big) \cdot \P\big( \cup_{b\in \gV_{-}\cap \sU} \{ \rB_{b}= 1 \} \big) \cdot \P(\rJ_{a} = 1) \cdot \P(\rJ_{b} = 1)\\
    = &\, \Big( 1 - \prod_{a\in \gV_{+}\cap \gU} \P(\rA_{a} = 0 ) \Big) \cdot \Big( 1 - \prod_{b\in \gV_{-}\cap \gU} \P (\rB_{b}= 0 ) \Big) \cdot \Big( 1 - \P(\rJ_{a} = 0) \Big) \cdot \Big( 1 - \P(\rJ_{b} = 0) \Big)\\
    \geq &\, 1 - 2\exp\Big( -\frac{N^{\epsilon + o(1)} }{2}\Big) - 2N^{-3 + o(1)} \\
    &\, \quad \quad - 2N^{-3 + o(1)}\cdot \exp( -N^{\epsilon + o(1)}) - 2N^{-6 + o(1)}\cdot \exp\Big( -\frac{N^{\epsilon + o(1)}}{2}\Big).
\end{align*}
Therefore, as $N$ goes to infinity, with probability tending to $1$, SDP \eqref{eqn:SDP} fails exact recovery when $\D_{\mathrm{AM}} < 1$.

\begin{proof}[Proof of \eqref{eqn:SDPLDP}]
Note that $\rY_{a}$ admits the decomposition
\begin{align}
     \rY_{a} = \sum\limits_{\ell \in \sL}\sum\limits_{r=0}^{\ell - 1} c^{(\ell)}_r \rY^{(\ell)}_r, \notag
\end{align}
where $c^{(\ell)}_r = 2\ervy_{a}  (\ell - 1 - 2r)$ and $\rY^{(\ell)}_r \sim \mathrm{Bin}(N^{(\ell)}_r, p^{(\ell)}_r)$ with $p^{(\ell)}_r = \alpha_{\ell}$ if $r = 0$, otherwise $p^{(\ell)}_r = \beta_{\ell}$, and
\begin{align}
    N^{(\ell)}_r = \binom{\frac{N}{2} - \frac{\gamma_{N} N}{2} - 1}{r} \binom{\frac{N}{2} - \frac{\gamma_{N} N}{2} - 1}{\ell - 1-r} = \big( 1 + o(1) \big) \cdot \frac{1}{2^{\ell - 1}} \binom{\ell - 1}{r} \binom{N - 1}{\ell - 1},\notag
\end{align}
Apply Lemma ~\ref{lem:LDP_binomial} with $g_{N} = 1$, $\delta = 0$, $\rho^{(\ell)}_r = \frac{1}{2^{\ell - 1}}\binom{\ell - 1}{r}$, and $\alpha^{(\ell)}_r = a_{\ell}$ if $r = 0$, otherwise $\alpha^{(\ell)}_r = b_{\ell}$. We then have
\begin{align}
    \P( \rA_{a} = 1) = \exp\big( -(1 + o(1))\cdot \D \cdot q_{N} \big),\notag
\end{align}
where the rate function $\D$ there becomes 
\begin{align*}
    \D = &\, \max_{t\geq 0} \Big(-t\delta + \sum\limits_{\ell \in \sL} \sum\limits_{r=0}^{\ell - 1} \alpha^{(\ell)}_r \rho^{(\ell)}_r (1 - e^{-c^{(\ell)}_r t}) \Big) \\
   = &\, \max_{t\geq 0} \sum\limits_{\ell \in \sL}\frac{1}{2^{\ell - 1}} \Big[ a_{\ell}\big(1 - e^{-(\ell - 1)2t} \big) + b_{\ell} \sum\limits_{r=1}^{\ell - 1} \binom{\ell - 1}{r} \big(1 - e^{-(\ell - 1 - 2r)2t} \big) \Big]\\
   =&\, \D_{\mathrm{AM}} \textnormal{  in  } \eqref{eqn:DAM_binary}.
\end{align*}
We take $q_{N} = \log(N)$. Since $\D_{\mathrm{AM}} = 1 - \epsilon$ for some absolute $\epsilon >0$, by applying $1 - x \leq e^{-x}$, we have
\begin{align}
    \prod_{a\in \gV_{+} } \P( \rA_{a} = 0) =&\, \prod_{a \in \gV_{+}} \Big( 1 - \P( \rA_{a} = 1)\Big) \leq (1 - N^{- (1 + o(1))\cdot \D_{\mathrm{AM}}})^{\frac{N}{2}} \notag\\
    \leq &\, \exp\Big( -\frac{1}{2}N^{1 - (1 + o(1))\cdot \D_{\mathrm{AM}}} \Big) = \exp\Big( -\frac{1}{2}N^{\epsilon + o(1)} \Big).\notag
\end{align}
The bound for $\P(\rJ_{a} = 0)$ could be obtained similarly by following the same argument in Lemma ~\ref{lem:Wab_J_prob}.
\end{proof}

\subsection{Sufficiency}
The proof strategy for sufficiency follows the same idea as \cite{Hajek2016AchievingEC}. For some $\nu \in \R$, define
\begin{align}
    \rmS = \rmD - \rmA + \nu \ones \ones^{\sT}\notag
\end{align}

\begin{lemma}[{\cite[Lemma 3]{Hajek2016AchievingEC}}]
    Suppose that there exists some diagonal matrix $\rmD\in \R^{N \times N}$ and $\nu \in \R$ such that $\lambda_{N-1}(\rmS) >0$, $\rmS \rvy = \bzero$ and $\rmS$ being positive semidefinite, then $\widehat{\Sigma} = \rvy \rvy^{\sT}$ is the unique solution to \eqref{eqn:SDP}.
\end{lemma}

For each $v\in \gV$, define
\begin{align}
    \rD_{v} = \sum\limits_{j=1}^{|\gV|}\ermA_{vj}\ervy_{v} \ervy_{j}.\notag
\end{align}
Then $\rD_{v} \ervy_{v} = \sum\limits_{j=1}^{|\gV|}\ermA_{vj} \ervy_{j}$, leading to $\rmD\rvy = \rmA \rvy$. Consequently, $\rmS \rvy = \bzero$ since $\ones^{\sT}\rvy = 0$. Then it suffices to verify that $\rmS$ is positive definite in the subspace orthogonal to $\rvy$ with high probability, i.e., 
\begin{align}
    \P\Big( \inf_{\rvx\perp \rvy, \|\rvx\| = 1} \rvx^{\sT}\rmS \rvx >0\Big) \geq 1 - N^{-\Omega(1)}.\notag
\end{align}
Note that $\rmA^{\star} = \E[\rmA] = \frac{\alpha + \beta}{2} \ones \ones^{\sT} + \frac{\alpha - \beta}{2}\rvy\rvy^{\sT} - \alpha\rmI_{N}$ with $\alpha, \beta$ defined in \eqref{eqn:alpha_beta_def}, then one has
\begin{align}
   &\,\rvx^{\sT}\rmS \rvx = \rvx^{\sT}\rmD\rvx - \frac{\alpha - \beta}{2}\rvx^{\sT} \rvy \rvy^{\sT} \rvx + \Big( \nu - \frac{\alpha + \beta}{2}\Big) \rvx^{\sT} \ones \ones^{\sT} \rvx + \alpha - \rvx^{\sT}(\rmA - \rmA^{\star})\rvx \notag \\
   \geq &\,\rvx^{\sT}\rmD\rvx + \alpha - \rvx^{\sT}(\rmA - \rmA^{\star})\rvx \geq \min_{v\in \gV} \rD_{v} + \alpha - \|\rmA - \rmA^{\star}\|,\notag
\end{align}
where the second to last inequality holds by choosing some proper $\nu \geq \frac{\alpha + \beta}{2}$. Note that $\|\rmA - \rmA^{\star}\| \lesssim \sqrt{\rho_{N}}$ by \Cref{thm:concentration}, then the lower bound above is determined by $\min_{v\in \gV} \rD_{v}$. Following the same strategy in the proof of \Cref{thm:achievability_matrices}, one could do the following decomposition
\begin{align}
    &\, \sqrt{N} \ervy_{v} \rmA \rvy = \ervy_{v} \sum\limits_{j\neq v, \,\, \ervy_{j} = +1 } \ermA_{vj} - \ervy_{v}\sum\limits_{j\neq v,\,\, \ervy_{j} = -1} \ermA_{vj} =\sum\limits_{j\neq v, \,\, \ervy_{j} = \ervy_{v} } \ermA_{vj} - \sum\limits_{j\neq v,\,\, \ervy_{j} \neq \ervy_{v}} \ermA_{vj} \notag\\
        =&\, \sum\limits_{\ell \in \sL} \sum\limits_{e\in \gE^{(0)}_{\ell}(v)} (\ell - 1)\cdot \tA^{(\ell)}_e - \sum\limits_{e\in \gE^{(\ell - 1)}_{\ell}(v)} (\ell - 1)\cdot \tA^{(\ell)}_e \notag\\
        &\, + \sum\limits_{\ell \in \sL}\sum\limits_{r=1}^{\ell - 2}\sum\limits_{e\in \gE^{(r)}_{\ell}} \Big( (\ell - 1-r) \cdot \etA^{(\ell)}_{e} - r \cdot \etA^{(\ell)}_{e} \Big), \label{eqn:extraTermsSDP}
\end{align}
where the contribution of \eqref{eqn:extraTermsSDP} does not vanish without the weighting trick, which impairs the way of obtaining the correct threshold $\D_{\mathrm{GH}}$. By an simple application of Lemma ~\ref{lem:LDP_binomial}, one could prove that $\P(\sqrt{N} \ervy_{v} \rmA \rvy > \delta \rho_{N}) \geq 1 - N^{-\D_{\mathrm{AM}} + \delta/2}$ when $\D_{\mathrm{AM}} = 1 + \delta$ for some $\delta >0$. Therefore, Algorithm \ref{alg:SDP} achieves exact recovery with high probability when $\D_{\mathrm{AM}} >1$.

\section{Technical Lemmas}

\begin{lemma}\label{lem:stirling}
    For integers $n, k \geq 1$, we have
    \begin{align*}
        \log(n!) = &\, n\log(n) - n + \frac{1}{2}\log(2\pi n) + O(n^{-1})\\
        \log \binom{n}{k} =&\, \frac{1}{2}\log \frac{n}{2\pi k (n-k)} + n\log(n) - k\log(k) - (n-k) \log(n-k) + O\Big( \frac{1}{n} + \frac{1}{k} + \frac{1}{n-k} \Big)\\
        \log \binom{n}{k} =&\, k\log\Big( \frac{n}{k} - 1 \Big) -\frac{1}{2}\log(2\pi k) + o(k^{-1})\,, \textnormal{ for } k = \omega(1) \textnormal{ and } \frac{k}{n} = o(1).
    \end{align*}
    Moreover, for any $1\leq k \leq \sqrt{n}$, we have
    \begin{align*}
        \frac{n^{k}}{4 \cdot k!} \leq \binom{n}{k} \leq \frac{n^{k}}{k!}\,,\quad \log \binom{n}{k} \geq k\log \Big(\frac{en}{k} \Big) - \frac{1}{2}\log(k)  - \frac{1}{12k} - \log(4\sqrt{2\pi})\,.
    \end{align*}
\end{lemma}
\begin{proof}[Proof of Lemma \ref{lem:stirling}]
According to Stirling's series \cite{Diaconis1986AnEP}, for any $n \geq 1$, 
\begin{align*}
    \sqrt{2\pi n}\, \left(\frac{n}{e}\right)^{n}e^{\frac{1}{12n+1}} n! \sqrt{2\pi n}\, \left(\frac{n}{e}\right)^{n}e^{\frac{1}{12n}}.
\end{align*}
Then the following asymptotic expansion holds:
\begin{align*}
    \Big| \log(n!) - \Big(\frac{1}{2} \log(2\pi n) + n\log(n) - n \Big) \Big|\leq \frac{1}{12n}.
\end{align*}
Thus for any $k \geq 1$, we have $\log \binom{n}{k} = \log(n!) - \log(k!) - \log((n-k)!)$, as well as
\begin{align*}
    \Big| \log \binom{n}{k} - \Big( \frac{1}{2}\log \frac{n}{2\pi k (n-k)} + n\log(n) - k\log(k) - (n-k) \log(n-k) \Big) \Big| \leq \frac{1}{12} \Big( \frac{1}{n} + \frac{1}{k} + \frac{1}{n-k} \Big).
\end{align*}
When $k = \omega(1)$ but $\frac{k}{n} = o(1)$, we have $\log(1 - \frac{k}{n}) = - \frac{k}{n} + O(N^{-2})$, and 
\begin{align*}
    \log \binom{n}{k} =&\, -\frac{1}{2}\log(2\pi k) - \frac{1}{2}\log\Big(1 - \frac{k}{n} \Big) + n\log(n) - k\log(k) \\
    &\, \quad - (n-k) \log(n) - (n-k)\log\Big(1 - \frac{k}{n} \Big) + o(1)\\
    =&\, k\log(n) - k\log(k) + k + o(k).
\end{align*}
At the same time, we write
    \begin{align*}
        \binom{n}{k} = \frac{n(n-1)\cdots (n - k + 1)}{k!} = \frac{n^{k}}{k!} \cdot \Big(1 - \frac{1}{n} \Big)\Big(1 - \frac{2}{n} \Big) \cdots \Big(1 - \frac{k-1}{n} \Big) \geq \frac{n^{k}}{k!}\Big(1 - \frac{k-1}{n} \Big)^{k-1} 
    \end{align*}
The upper bound is trivial. For the lower bound, let $f(x) = (1 - x/n)^{x}$ denote the function. It is easy to see that $f(x)$ is decreasing with respect to $x$, and 
\begin{align*}
    \Big(1 - \frac{k-1}{n} \Big)^{k-1} \geq \Big(1 - \frac{1}{\sqrt{n}} \Big)^{\sqrt{n}}\,.
\end{align*}
Let $g(t) = (1 - 1/t)^{t}$ and we know that $g(t)$ is increasing when $t \geq 1$, hence $g(t) \geq g(2) = 1/4$ for any $t \geq 2$ and the desired lower bound follows. Then for any $k\leq \sqrt{n}$, we have
\begin{align*}
    \log \binom{n}{k} \geq &\,\log \frac{n^{k}}{4 \cdot k!} \\
    \geq &\, \log \left( \frac{n^{k}}{4 \cdot \sqrt{2\pi k}\ (\frac{k}{e})^{k}e^{\frac {1}{12k}}} \right) = k\log \Big(\frac{en}{k} \Big) - \frac{1}{2}\log(k)  - \frac{1}{12k} - \log(4\sqrt{2\pi})\,.
\end{align*}
\end{proof}

\begin{lemma}[Weyl's inequality, \cite{Weyl1912DasAV}]\label{lem:weyl}
Let $\rmA, \rmE \in \R^{m \times n}$ be two real $m\times n$ matrices, then $|\sigma_i(\rmA + \rmE) - \sigma_i(\rmA)| \leq \|\rmE\|$ for every $1 \leq i \leq \min\{ m, n\}$. Furthermore, if $m = n$ and $\rmA, \rmE \in \R^{n \times n}$ are real symmetric, then $|\lambda_i(\rmA + \rmE) - \lambda_i(\rmA)| \leq \|\rmE\|$ for all $1 \leq i \leq n$.
\end{lemma}

\begin{lemma}[Eckart–Young–Mirsky Theorem \cite{Eckart1936TheAO}]\label{lem:EYM}
Suppose that the matrix $\rmA \in \R^{m \times n}$ ($m\leq n$) adapts the singular value decomposition $\rmA = \rmU \mSigma \rmV^{\sT}$, where $\rmU = [\rvu_1, \ldots, \rvu_{m}] \in \R^{m\times m}$, $\rmV = [\rvv_1, \ldots, \rvv_{n}] \in \R^{n \times n}$ and $\mSigma \in \R^{m\times n}$ contains diagonal elements $\sigma_1 \geq \ldots \geq \sigma_{m}$. Let $\rmA^{(K)} = \sum_{k=1}^{K} \sigma_k \rvu_k \rvv_k ^{\sT}$ be the rank-$K$ approximation of $\rmA$. Then $\|\rmA - \rmA^{(K)}\|_2 = \sigma_{K+1} \leq \|\rmA - \rmB\|_2$ for any matrix $\rmB$ with $\mathrm{rank}(\rmB) = K$.
\end{lemma}

\begin{lemma}[Davis-Kahan Theorem, \cite{Davis1970TheRO, Yu2015UsefulVO}]\label{lem:DavisKahan}
    Let $\rmA$ and $\rmA^{\star}$ be $n\times n$ symmetric matrices with eigenpairs $\{(\lambda_i, \rvu_i)\}_{i=1}^{|\gV|}$ and $\{(\lambda^{\star}_i, \rvu^{\star}_i)\}_{i=1}^{|\gV|}$ respectively, where eigenvalues are sorted in the decreasing order. We take $\lambda_0 \coloneqq +\infty$ and $\lambda_{n+1} \coloneqq -\infty$. For fixed $j\in \gV$, assume that $\delta = \min\{\lambda_{j-1}-\lambda_{j}, \lambda_{j}-\lambda_{j+1}\} > 0$. Then for eigenpairs of $(\lambda_{j}, \rvu_{j})$ of $\rmA$ and $(\lambda^{\star}_{j}, \rvu^{\star}_{j})$ of $\rmA^{\star}$, if $\<\rvu_{j}, \rvu^{\star}_{j}\> \geq 0$, one has
    \begin{align*}
        \|\rvu_{j} - \rvu^{\star}_{j} \|_2 \leq \frac{\sqrt{8} \|\rmA -\rmA^{\star}\|}{\delta}.
    \end{align*}
\end{lemma}

\begin{lemma}[Generalized Davis-Kahan Theorem, \cite{Eisenstat1998RelativePR, Deng2021StrongCG}]\label{lem:generalized_Davis_Kahan}
    Consider the generalized eigenvalue problem $\rmA \rvu = \lambda \rmD \rvu$ where $\rmA\in \R^{n\times n}$ being symmetric and $\rmD \in \R^{n\times n}$ being diagonal. Let $\rmD^{-1} \in \R^{n \times n}$ denote the Moore-Penrose inverse of $\rmD$, that is $(\rmD^{-1})_{vv} = (\rD_{v})^{-1}$ if $\rD_{v} \neq 0$, otherwise $(\rmD^{-1})_{vv} = 0$. Then the matrix $\rmD^{-1}\rmA \in \R^{n \times n}$ is diagonalizable, which admits the decomposition
    \begin{align*}
        \rmD^{-1}\rmA = \sum\limits_{i=1}^{|\gV|}\lambda_i \rvu_i \rvu_i^\sT = \begin{bmatrix}
                \rmU & \rmU_{\perp}
            \end{bmatrix}
            \begin{bmatrix}
              \Lambda & \bzero\\
                \bzero & \Lambda_{\perp}
            \end{bmatrix}
            \begin{bmatrix}
               \rmU^{\sT}\\
                \rmU_{\perp}^{\sT}
            \end{bmatrix}
    \end{align*}
    Let $\rmP = \rmU_{\perp} \rmU_{\perp}^{\sT} = \rmI - \rmU \rmU^{\sT}$ denote the projection onto the subspace orthogonal to the space spanned by columns of $\rmU$. Let $\delta = \min_{i}|(\Lambda_{\perp})_{ii} - \widehat{\lambda} |$ denote the absolute separation of some $\widehat{\lambda}$ from $\Lambda_{\perp}$, then for any vector $\widehat{\rvu}$, 
    \begin{align*}
        \|\rmP \widehat{\rvu}\| \leq \frac{\|(\rmD^{-1} \rmA - \widehat{\lambda} \rmI)\widehat{\rvu}\| }{\delta}\,\,.
    \end{align*}
\end{lemma}
\begin{proof}[Proof of Lemma \ref{lem:generalized_Davis_Kahan}]
    First, note that
    \begin{align*}
        (\rmD^{-1} \rmA - \widehat{\lambda} \rmI)\widehat{\rvu} = \begin{bmatrix}
                \rmU & \rmU_{\perp}
            \end{bmatrix}
            \begin{bmatrix}
              \Lambda - \widehat{\lambda} \rmI & \bzero\\
                \bzero & \Lambda_{\perp} - \widehat{\lambda} \rmI
            \end{bmatrix}
            \begin{bmatrix}
               \rmU^{\sT} \widehat{\rvu}\\
                \rmU_{\perp}^{\sT} \widehat{\rvu},
            \end{bmatrix} 
    \end{align*}
    thus
    \begin{align*}
         \rmU_{\perp}^{\sT} \widehat{\rvu} = (\Lambda_{\perp} - \widehat{\lambda} \rmI)^{-1}\rmU_{\perp}^{\sT}(\rmD^{-1} \rmA - \widehat{\lambda} \rmI)\widehat{\rvu} 
    \end{align*}
    Note that columns of $\rmU_{\perp}$ are pairwise orthogonal, therefore 
    \begin{align*}
        \| \rmP\widehat{\rvu} \| = &\, \| \rmU_{\perp} \rmU_{\perp}^{\sT} \widehat{\rvu} \| = \| \rmU_{\perp}(\Lambda_{\perp} - \widehat{\lambda} \rmI)^{-1}\rmU_{\perp}^{\sT}(\rmD^{-1} \rmA - \widehat{\lambda} \rmI)\widehat{\rvu}\|\\
        \leq &\, \frac{\|\rmU_{\perp}\rmU_{\perp}^{\sT}\| \cdot \|(\rmD^{-1} \rmA - \widehat{\lambda} \rmI)\widehat{\rvu}\|}{\delta} \leq \frac{\|(\rmD^{-1} \rmA - \widehat{\lambda} \rmI)\widehat{\rvu}\|}{\delta}.
    \end{align*}
\end{proof}

\begin{lemma}[Markov's inequality, {\cite[Proposition $1.2.4$]{Vershynin2018HighDP}}]\label{lem:Markov}
    For any non-negative random variable $\rX$ and $t>0$, we have 
    \begin{align*}
        \P(\rX > t) \leq \E(\rX)/t.
    \end{align*}
\end{lemma}

\begin{lemma}[Hoeffding's inequality, {\cite[Theorem $2.2.6$]{Vershynin2018HighDP}}]\label{lem:Hoeffding}
    Let $\rX_1,\dots, \rX_{N}$ be independent  random variables with $\rX_i \in [a_i, b_i]$, then for any, $t \geq 0$, we have
    \begin{align*}
        \P\Bigg( \bigg|\sum_{i=1}^{N}(\rX_i - \E \rX_i ) \bigg| \geq t \Bigg) \leq 2\exp \Bigg( -\frac{2t^{2} }{ \sum_{i=1}^{N}(b_i - a_i )^{2} } \Bigg)\,.
    \end{align*}
\end{lemma}

\begin{lemma}[Chernoff's inequality, {\cite[Theorem $2.3.1$]{Vershynin2018HighDP}}]\label{lem:Chernoff}
    Let $\rX_i$ be independent Bernoulli random variables with parameters $p_i$. Consider their sum $\rS_{N} = \sum_{i=1}^{N}\rX_i$ and denote its mean by $\mu = \E \rS_{N}$. Then for any $t > \mu$, 
    \begin{align*}
        \P \big( \rS_{N} \geq t \big) \leq e^{-\mu} \left( \frac{e \mu}{t} \right)^{t}\,.
    \end{align*}
\end{lemma}

\begin{lemma}[Bernstein's inequality, {\cite[Theorem $2.8.4$]{Vershynin2018HighDP}}]\label{lem:Bernstein}
    Let $\rX_1,\dots, \rX_{N}$ be independent mean-zero random variables such that $|\rX_i|\leq K$ for all $i$. Let $\sigma^2 = \sum_{i=1}^{N}\E \rX_i^2$. Then for every $t \geq 0$,
    \begin{align*}
        \P \Bigg( \Big|\sum_{i=1}^{N} \rX_i \Big| \geq t \Bigg) \leq 2 \exp \Bigg( - \frac{t^2/2}{\sigma^2 + Kt/3} \Bigg)\,.
    \end{align*}
\end{lemma}

\begin{lemma}[Bennett's inequality, {\cite[Theorem $2.9.2$]{Vershynin2018HighDP} }]\label{lem:Bennett}
    Let $\rX_1,\dots, \rX_{N}$ be independent random variables. Assume that $|\rX_i - \E \rX_i| \leq K$ almost surely for every $i$. Then for any $t>0$, we have
    \begin{align*}
        \P \Bigg( \sum_{i=1}^{N} (\rX_i - \E \rX_i) \geq t \Bigg) \leq \exp \Bigg( - \frac{\sigma^2}{K^2} \cdot h \bigg( \frac{Kt}{\sigma^2} \bigg)\Bigg)\,, \notag 
    \end{align*}
    where $\sigma^2 = \sum_{i=1}^{N}\Var(\rX_i)$ is the variance of the sum and $h(u) := (1 + u)\log(1 + u) - u$.
\end{lemma}

\begin{lemma}[Variants of Chernoff]\label{lem:Chernoffvariant}
Let $\rX_i$ be independent Bernoulli random variable with mean $p_i$. Consider their sum $\rS_{N} \coloneqq \sum_{i=1}^{N}\rX_i$ and denote its mean by $\mu \coloneqq \E \rS_{N}$.
    \begin{enumerate}
        \item Let $\mu = \gamma \cdot q_{N}$ for some $\gamma >0$ and $q_{N} \gg 1$, then for any $\zeta > \gamma$, we have
            \begin{align*}
                \P(\rS_{N} \geq \zeta q_{N} ) \leq \exp\big( - q_{N} \cdot \big[ (\gamma - \zeta) + \zeta \log (\zeta/\gamma) \big] \big)
            \end{align*}
        Specifically, $\P(\rS_{N} \geq t ) \leq N^{-[\zeta \log \frac{\zeta}{\gamma}  +  ( \gamma - \zeta) ]}$ when $q_{N} = \log(N)$.
        \item If $\mu = \gamma_{N} \cdot q_{N}$ for some vanishing sequence $\{\gamma_{N}\}_{N\geq 1}$ with $\gamma_{N} = o(1)$, then for any vanishing sequence $\{\zeta_{N}\}_{ N \geq 1}$ with $\zeta_{N} = \omega(\gamma_{N})$, we have
            \begin{align*}
                \P(\rS_{N} \geq \zeta_{N} \cdot q_{N} ) \leq \exp(-q_{N}\cdot [\zeta_{N} \log \frac{\zeta_{N}}{\gamma_{N}}( 1  +  o(1)\, ) ]).
            \end{align*}
    \end{enumerate}
\end{lemma}
\begin{proof}[Proof of Lemma \ref{lem:Chernoffvariant}]
Chernoff Lemma \ref{lem:Chernoff} gives
\begin{align*}
    \P(\rS_{N} \geq t ) &\,  \leq e^{-\mu} \Big( \frac{e \mu}{t} \Big)^{t} = \exp\Big( -(\mu-t) - t\log \Big(\frac{t}{\mu} \Big) \Big)
\end{align*}
Let $f(t) = t\log(\frac{t}{\mu}) + (\mu - t)$, where $f(\mu) = 0$ and $f^{'}(t) = \log(\frac{t}{\mu}) > 0$ for $t > \mu$, thus $f(t) > 0$ for all $t > \mu$.
\begin{enumerate}
    \item By taking $t = \zeta q_{N} $, we have
        \begin{align*}
            \P(\rS_{N} \geq t ) \leq \exp\big( - (\gamma - \zeta) q_{N} - \zeta \log( \zeta/\gamma) \cdot q_{N} \big).
        \end{align*}
    \item Note that $\gamma_{N} = o(1)$, $\zeta_{N} = o(1)$ but $\frac{\zeta_{N}}{\gamma_{N}} = \omega(1)$, then $\frac{\gamma_{N} - \zeta_{N}}{\zeta_{N} \log(\zeta_{N}/\gamma_{N}) } = o(1)$, hence follows.
\end{enumerate}
\end{proof}

\begin{lemma}[{\cite[Corollary 3]{latala1997estimation}}]\label{lem:lp-moments}
    Let $\rX_{1}, \ldots, \rX_{N}$ be a sequence of independent non-negative random variables. For $p \geq 1$, there exists some absolute constant $\const > 0$ such that
    \begin{align}
        \Big( \E \big( \sum_{i=1}^{N} \rX_{i} \big)^{p} \Big)^{1/p} \leq \const \frac{p}{\ln(p)} \cdot \max \Big\{ \sum_{i=1}^{N} \E \rX_{i}\,\,, \quad \big(\sum_{i=1}^{N} \E |\rX_{i}|^{p}\big)^{1/p} \Big\}. \notag
    \end{align}
Furthermore, if $\rX_{i}$ are independent symmetric random variables and $p \geq 2$, then
    \begin{align}
        \Big( \E \big( \sum_{i=1}^{N} \rX_{i} \big)^{p} \Big)^{1/p} \leq \const \frac{p}{\ln(p)} \cdot \max \Big\{ \big(\sum_{i=1}^{N} \E |\rX_{i}|^{2}\big)^{1/2} \,\,, \quad \big(\sum_{i=1}^{N} \E |\rX_{i}|^{p}\big)^{1/p} \Big\}. \notag
    \end{align} 
\end{lemma}

%% file: bibliography.bib
@article{Dumitriu2025PartialRA,
    author={Dumitriu, Ioana and Wang, Hai-Xiao and Zhu, Yizhe},
    title={Partial recovery and weak consistency in the non-uniform hypergraph stochastic block model}, 
    volume={34}, 
    number={1}, 
    journal={Combinatorics, Probability and Computing},
    doi={10.1017/S0963548324000166},
    archivePrefix = "arXiv",
    eprint = {2112.11671},
    primaryClass = "math.ST",
    year={2025},
    pages={1–51},
    publisher={Cambridge University Press}
}

@article{Dumitriu2023OptimalAE,
    title={Optimal and exact recovery on the general non-uniform hypergraph stochastic block model},
    author={Ioana Dumitriu and Hai-Xiao Wang},
    journal={Accepted by Annals of Statistics},
    archivePrefix = "arXiv",
    eprint = {2304.13139},
    primaryClass = "math.ST",
    year={2023}
}

@inproceedings{Wang2024OptimalER,
    title={Optimal exact recovery in semi-supervised learning: A study of spectral methods and graph convolutional networks}, 
    author={Wang, Hai-Xiao and Wang, Zhichao},
    booktitle={Proceedings of the 41st International Conference on Machine Learning},
    pages={51614--51649},
    volume = 	{235},
    series = 	{Proceedings of Machine Learning Research},
    year  = {2024},
    month = {07},
    url = 	 {https://proceedings.mlr.press/v235/wang24bt.html},
    archivePrefix = "arXiv",
    eprint = {2412.13754},
    primaryClass = "cs.LG"
}

@book{Bai2010Spectral,
  title={Spectral Analysis of Large Dimensional Random Matrices},
  author={Bai, Zhidong and Silverstein, Jack W.},
  isbn={9781441906618},
  lccn={2009942423},
  series={Springer Series in Statistics},
  url={https://books.google.com/books?id=kd-o5Qdm7ngC},
  year={2010},
  publisher={Springer New York}
}

@article{Davis1970TheRO,
  title = {The Rotation of Eigenvectors by a Perturbation. III},
  author = {Davis, Chandler and Kahan, W. M.},
  journal = {SIAM Journal on Numerical Analysis},
  volume = {7},
  number = {1},
  pages = {1--46},
  year = {1970},
  doi = {10.1137/0707001},
  url = {https://doi.org/10.1137/0707001},
  issn = {00361429},
  publisher = {Society for Industrial and Applied Mathematics}
}

@book{Vershynin2018HighDP,
    place={Cambridge},
    series={Cambridge Series in Statistical and Probabilistic Mathematics},
    title={High-Dimensional Probability: An Introduction with Applications in Data Science},
    publisher={Cambridge University Press},
    author={Vershynin, Roman}, 
    year={2018},
    collection={Cambridge Series in Statistical and Probabilistic Mathematics}
}

@article{Yu2015UsefulVO,
    title = {A useful variant of the Davis–Kahan theorem for statisticians},
    author={Yu, Yi and Wang, Tengyao and Samworth, Richard J.},
    journal = {Biometrika},
    volume = {102},
    number = {2},
    pages = {315-323},
    year = {2014},
    month = {04},
    issn = {0006-3444},
    doi = {10.1093/biomet/asv008},
    url = {https://doi.org/10.1093/biomet/asv008},
    eprint = {https://academic.oup.com/biomet/article-pdf/102/2/315/9642505/asv008.pdf},
    publisher={Oxford University Press}
}

@article{Abbe2015CommunityDI,
    title={Community Detection in General Stochastic Block models: Fundamental Limits and Efficient Algorithms for Recovery},
    author={Abbe, Emmanuel and Sandon, Colin},
    journal={2015 IEEE 56th Annual Symposium on Foundations of Computer Science},
    year={2015},
    pages={670-688},
    doi ={10.1109/FOCS.2015.47}
}

@article{Abbe2016ExactRI,
    title={Exact recovery in the stochastic block model},
    author={Abbe, Emmanuel and Bandeira, Afonso S and Hall, Georgina},
    journal={IEEE Transactions on Information Theory},
    volume={62},
    number={1},
    pages={471--487},
    year={2016},
    doi={10.1109/TIT.2015.2490670}
}

@article{Abbe2018CommunityDA,
    author  = {Abbe, Emmanuel},
    title   = {Community Detection and Stochastic Block Models: Recent Developments},
    journal = {Journal of Machine Learning Research},
    year    = {2018},
    volume  = {18},
    number  = {177},
    pages   = {1-86},
    doi = {https://doi.org/10.48550/arXiv.1703.10146},
    url     = {http://jmlr.org/papers/v18/16-480.html}
}

@article{Abbe2020EntrywiseEA,
    title={Entrywise eigenvector analysis of random matrices with low expected rank},
    author={Abbe, Emmanuel and Fan, Jianqing and Wang, Kaizheng and Zhong, Yiqiao},
    journal={Annals of statistics},
    year={2020},
    volume={48},
    number={3},
    pages={1452-1474},
    doi = {https://doi.org/10.1214/19-AOS1854},
    publisher={Institute of Mathematical Statistics}
}

@article{Abbe2022LPT,
    title={An $\ell_p$ theory of PCA and spectral clustering},
    author={Abbe, Emmanuel and Fan, Jianqing and Wang, Kaizheng},
    journal={The Annals of Statistics},
    volume={50},
    number={4},
    pages={2359--2385},
    doi = {https://doi.org/10.1214/22-AOS2196},
    year={2022},
    publisher={Institute of Mathematical Statistics}
}

@article{Agarwal2017MultisectionIT,
  title={Multisection in the stochastic block model using semidefinite programming},
  author={Agarwal, Naman and Bandeira, Afonso S and Koiliaris, Konstantinos and Kolla, Alexandra},
  journal={Compressed Sensing and its Applications: Second International MATHEON Conference 2015},
  pages={125--162},
  year={2017},
  organization={Springer}
}

@article{Agterberg2022JointSC,
    title={Joint Spectral Clustering in Multilayer Degree-Corrected Stochastic Blockmodels},
    author={Agterberg, Joshua and Lubberts, Zachary and Arroyo, Jes{\'u}s},
    journal={arXiv preprint arXiv:2212.05053},
    year={2022}
}

@inproceedings{Ahn2016CommunityRI,
    Author = {Ahn, Kwangjun and Lee, Kangwook and Suh, Changho},
    Booktitle = {Communication, Control, and Computing (Allerton), 2016 54th Annual Allerton Conference on},
    Organization = {IEEE},
    Pages = {657--663},
    Title = {Community recovery in hypergraphs},
    Year = {2016}
 }

@article{Alaluusua2023MultilayerHC,
    title={Multilayer hypergraph clustering using the aggregate similarity matrix},
    author={Alaluusua, Kalle and Avrachenkov, Konstantin and Kumar, BR Vinay and Leskel{\"a}, Lasse},
    journal={International Workshop on Algorithms and Models for the Web-Graph},
    pages={83--98},
    year={2023},
    doi = {https://doi.org/10.1007/978-3-031-32296-9_6},
    organization={Springer}
}

@article{Angelini2015SpectralDO,
    author = {Angelini, Maria Chiara and Caltagirone, Francesco and Krzakala, Florent and Zdeborov{\'a}, Lenka},
    title = {Spectral detection on sparse hypergraphs},
    journal = {Communication, Control, and Computing (Allerton), 2015 53rd Annual Allerton Conference on},
    publisher = {IEEE},
    pages = {66--73},
    year = {2015},
    doi={10.1109/ALLERTON.2015.7446987}
}

@article{Arias2014CommunityDI,
    author = {Ery Arias-Castro and Nicolas Verzelen},
    title = {Community detection in dense random networks},
    volume = {42},
    journal = {The Annals of Statistics},
    number = {3},
    publisher = {Institute of Mathematical Statistics},
    pages = {940 -- 969},
    year = {2014},
    doi = {10.1214/14-AOS1208},
    url = {https://doi.org/10.1214/14-AOS1208}
}

@inproceedings{Baranwal2021GraphCS,
  title = {Graph Convolution for Semi-Supervised Classification: Improved Linear Separability and Out-of-Distribution Generalization},
  author =       {Baranwal, Aseem and Fountoulakis, Kimon and Jagannath, Aukosh},
  booktitle = 	 {Proceedings of the 38th International Conference on Machine Learning},
  pages = 	 {684--693},
  year = 	 {2021},
  volume = 	 {139},
  series = 	 {Proceedings of Machine Learning Research},
  month = 	 {07},
  publisher =    {PMLR},
  url = 	 {https://proceedings.mlr.press/v139/baranwal21a.html},
}

@article{Baranwal2022EffectsGC,
    title={Effects of Graph Convolutions in Multi-layer Networks},
    author={Baranwal, Aseem and Fountoulakis, Kimon and Jagannath, Aukosh},
    booktitle={The Eleventh International Conference on Learning Representations},
    journal={arXiv preprint arXiv:2204.09297},
    year={2023}
}

@article{Bickel2009ANV,
  title={A nonparametric view of network models and Newman–Girvan and other modularities},
  author={Peter J. Bickel and Aiyou Chen},
  journal={Proceedings of the National Academy of Sciences},
  year={2009},
  volume={106},
  pages={21068 - 21073}
}

@article{Boppana1987EigenvaluesAG,
  title={Eigenvalues and graph bisection: An average-case analysis},
  author={Boppana, Ravi B.},
  journal={28th Annual Symposium on Foundations of Computer Science (sfcs 1987)},
  year={1987},
  pages={280-285}
}

@inproceedings{bresler2024thresholds,
  title={Thresholds for Reconstruction of Random Hypergraphs From Graph Projections},
  author={Bresler, Guy and Guo, Chenghao and Polyanskiy, Yury},
  booktitle={The Thirty Seventh Annual Conference on Learning Theory},
  pages={632--647},
  year={2024},
  organization={PMLR}
}

@article{Bui1984GraphBA,
  title={Graph bisection algorithms with good average case behavior},
  author={Bui,Thang Nguyen and Chaudhuri, Soma and Leighton, Frank Thomson and Sipser, Michael},
  journal={Combinatorica},
  year={1987},
  volume={7},
  pages={171–191}
}

@article{Chen2018SupervisedCD,
    title={Supervised Community Detection with Line Graph Neural Networks},
    author={Zhengdao Chen and Lisha Li and Joan Bruna},
    booktitle={International Conference on Learning Representations},
    year={2019},
    url={https://openreview.net/forum?id=H1g0Z3A9Fm},
}

@article{Chen2022GlobalAI,
    title={Global and individualized community detection in inhomogeneous multilayer networks},
    author={Chen, Shuxiao and Liu, Sifan and Ma, Zongming},
    journal={The Annals of Statistics},
    volume={50},
    number={5},
    pages={2664--2693},
    year={2022},
    publisher={Institute of Mathematical Statistics},
    doi = {10.1214/22-AOS2202},
    URL = {https://doi.org/10.1214/22-AOS2202}
}

@inproceedings{Chien2018CommunityDI,
    Author = {Chien, I and Lin, Chung-Yi and Wang, I-Hsiang},
    Booktitle = {International Conference on Artificial Intelligence and Statistics},
    Pages = {871--879},
    Title = {Community detection in hypergraphs: Optimal statistical limit and efficient algorithms},
    Year = {2018}
}

@article{Chien2019MinimaxMR,
    title={On the minimax misclassification ratio of hypergraph community detection},
    author={Chien, I Eli and Lin, Chung-Yi and Wang, I-Hsiang},
    journal={IEEE Transactions on Information Theory},
    volume={65},
    number={12},
    pages={8095--8118},
    year={2019},
    publisher={IEEE},
    doi={10.1109/TIT.2019.2928301}
}

@article{Choi2012StochasticBW,
  title={Stochastic blockmodels with a growing number of classes},
  author={Choi, David S and Wolfe, Patrick J and Airoldi, Edoardo M},
  journal={Biometrika},
  volume={99},
  number={2},
  pages={273--284},
  year={2012},
  publisher={Oxford University Press}
}

@article{Coja-Oghlan2010GraphPV,
    title={Graph Partitioning via Adaptive Spectral Techniques},
    volume={19},
    DOI={10.1017/S0963548309990514},
    number={2},
    journal={Combinatorics, Probability and Computing},
    publisher={Cambridge University Press}, 
    author={Amin Coja-Oghlan},
    year={2010},
    pages={227–284}
}

@article{Cole2020ExactRI,
    title = {Exact recovery in the hypergraph stochastic block model: A spectral algorithm},
    journal = {Linear Algebra and its Applications},
    volume = {593},
    pages = {45-73},
    year = {2020},
    issn = {0024-3795},
    doi = {https://doi.org/10.1016/j.laa.2020.01.039},
    url = {https://www.sciencedirect.com/science/article/pii/S0024379520300562},
    author = {Cole, Sam and Zhu, Yizhe}
}

@inproceedings{Condon1999AlgorithmsFG,
  title={Algorithms for graph partitioning on the planted partition model},
  author={Anne Condon and Richard M. Karp},
  booktitle={Random Struct. Algorithms},
  year={1999}
}

@article{Deng2021StrongCG,
  title={Strong Consistency, Graph Laplacians, and the Stochastic Block Model},
  author={Shaofeng Deng and Shuyang Ling and Thomas Strohmer},
  journal={Journal of Machine Learning Research},
  year={2021},
  volume={22},
  pages={1-44}
}

@article{Diaconis1986AnEP,
 ISSN = {00029890, 19300972},
 URL = {http://www.jstor.org/stable/2322709},
 author = {Diaconis, Persi and Freedman, David},
 journal = {The American Mathematical Monthly},
 number = {2},
 pages = {123--125},
 publisher = {Mathematical Association of America},
 title = {An Elementary Proof of Stirling's Formula},
 volume = {93},
 year = {1986}
}

@article{Dyer1989TheSO,
  title={The Solution of Some Random NP-Hard Problems in Polynomial Expected Time},
  author={Martin E. Dyer and Alan M. Frieze},
  journal={J. Algorithms},
  year={1989},
  volume={10},
  pages={451-489}
}

@article{Eckart1936TheAO,
  title={The approximation of one matrix by another of lower rank},
  author={Carl Eckart and G. Marion Young},
  journal={Psychometrika},
  year={1936},
  volume={1},
  pages={211-218}
}

@article{Eisenstat1998RelativePR, 
    title={Relative perturbation results for matrix eigenvalues and singular values}, 
    journal={BIT Numerical Mathematics}, 
    publisher={Cambridge University Press}, 
    author={Eisenstat, S. C. and Ipsen, I. C. F.}, 
    year={1998},
    volume = {38},
    number = {3},
    pages={502--509},
    doi = {10.1007/BF02510256},
    url={https://doi.org/10.1007/BF02510256}
}

@inproceedings{Eldridge2018UnperturbedSA,
  title={Unperturbed: spectral analysis beyond Davis-Kahan},
  author={Eldridge, Justin and Belkin, Mikhail and Wang, Yusu},
  booktitle={Algorithmic Learning Theory},
  pages={321--358},
  year={2018},
  organization={PMLR}
}

@article{Gao2017AchievingOM,
    title={Achieving optimal misclassification proportion in stochastic block models},
    author={Gao, Chao and Ma, Zongming and Zhang, Anderson Y and Zhou, Harrison H},
    journal={Journal of Machine Learning Research},
    volume={18},
    number={1},
    pages={1980--2024},
    year={2017},
    publisher={JMLR}
}

@article{Gaudio2023CommunityDI,
    title={Community detection in the hypergraph SBM: Optimal recovery given the similarity matrix},
    author={Gaudio, Julia and Joshi, Nirmit},
    journal={The Thirty Sixth Annual Conference on Learning Theory},
    doi = {https://doi.org/10.48550/arXiv.2208.12227},
    pages={469--510},
    year={2023},
    publisher={PMLR}
}

@article{Ghoshdastidar2014ConsistencyOS,
    author = {Ghoshdastidar, Debarghya and Dukkipati, Ambedkar},
    journal = {Advances in Neural Information Processing Systems},
    publisher = {Curran Associates, Inc.},
    pages = {},
    title = {Consistency of spectral partitioning of uniform hypergraphs under planted partition model},
    volume = {27},
    year = {2014},
    url = {https://proceedings.neurips.cc/paper_files/paper/2014/file/8f121ce07d74717e0b1f21d122e04521-Paper.pdf},
}

@article{Ghoshdastidar2017ConsistencyOS,
    title={Consistency of spectral hypergraph partitioning under planted partition model},
    author={Ghoshdastidar, Debarghya and Dukkipati, Ambedkar},
    journal={The Annals of Statistics},
    volume={45},
    number={1},
    pages={289--315},
    year={2017},
    doi = {ttps://doi.org/10.1214/16-AOS1453},
    publisher={Institute of Mathematical Statistics}
}

@article{Goemans1995ImprovedAA,
    author = {Goemans, Michel X. and Williamson, David P.},
    title = {Improved Approximation Algorithms for Maximum Cut and Satisfiability Problems Using Semidefinite Programming},
    year = {1995},
    issue_date = {Nov. 1995},
    publisher = {Association for Computing Machinery},
    address = {New York, NY, USA},
    volume = {42},
    number = {6},
    issn = {0004-5411},
    url = {https://doi.org/10.1145/227683.227684},
    doi = {10.1145/227683.227684},
    journal = {J. ACM},
    month = {11},
    pages = {1115–1145},
    numpages = {31}
}

@inproceedings{Gu2023WeakRT,
  title = {Weak Recovery Threshold for the Hypergraph Stochastic Block Model},
  author =    {Gu, Yuzhou and Polyanskiy, Yury},
  journal = 	{Proceedings of Thirty Sixth Conference on Learning Theory},
  pages = 	 {885--920},
  year = 	 {2023},
  editor = 	 {Neu, Gergely and Rosasco, Lorenzo},
  volume = 	 {195},
  series = 	 {Proceedings of Machine Learning Research},
  month = 	 {07},
  publisher =    {PMLR},
  url = 	 {https://proceedings.mlr.press/v195/gu23b.html},
}

@article{Hajek2016AchievingEC,
  title={Achieving exact cluster recovery threshold via semidefinite programming},
  author={Hajek, Bruce and Wu, Yihong and Xu, Jiaming},
  journal={IEEE Transactions on Information Theory},
  volume={62},
  number={5},
  pages={2788--2797},
  year={2016},
  publisher={IEEE}
}

@article{Hillar2013MostTP,
    title={Most tensor problems are NP-hard},
    author={Hillar, Christopher J and Lim, Lek-Heng},
    journal={Journal of the ACM (JACM)},
    volume={60},
    number={6},
    pages={45},
    year={2013},
    publisher={ACM},
    doi = {https://doi.org/10.48550/arXiv.0911.1393}
}

@article{Holland1983StochasticBM,
    title = {Stochastic blockmodels: First steps},
    author={Holland, Paul W and Laskey, Kathryn Blackmond and Leinhardt, Samuel},
    journal = {Social Networks},
    volume = {5},
    number = {2},
    pages = {109-137},
    year = {1983},
    issn = {0378-8733},
    doi = {https://doi.org/10.1016/0378-8733(83)90021-7},
    url = {https://www.sciencedirect.com/science/article/pii/0378873383900217},
    publisher={Elsevier}
}

@article{Jiang2020FasterIP,
  author={Jiang, Haotian and Kathuria, Tarun and Lee, Yin Tat and Padmanabhan, Swati and Song, Zhao},
  booktitle={2020 IEEE 61st Annual Symposium on Foundations of Computer Science (FOCS)}, 
  title={A Faster Interior Point Method for Semidefinite Programming}, 
  year={2020},
  volume={},
  number={},
  pages={910-918},
  doi={10.1109/FOCS46700.2020.00089}
  }

@article{Kim2018StochasticBM,
    title={Stochastic Block Model for Hypergraphs: Statistical limits and a semidefinite programming approach},
    author={Kim, Chiheon and Bandeira, Afonso S. and Goemans, Michel X.},
    journal={arXiv preprint arXiv:1807.02884},
    year={2018},
    url = {https://arxiv.org/abs/1807.02884},
    doi = {https://doi.org/10.48550/arXiv.1807.02884}
}

@article{Lee2020RobustHC,
    title={Robust hypergraph clustering via convex relaxation of truncated MLE},
    author={Lee, Jeonghwan and Kim, Daesung and Chung, Hye Won},
    journal={IEEE Journal on Selected Areas in Information Theory},
    volume={1},
    number={3},
    pages={613--631},
    year={2020},
    publisher={IEEE}
}

@article{latala1997estimation,
    title = {Estimation of Moments of Sums of Independent Real Random Variables},
    author = {Rafał Latała},
    journal = {The Annals of Probability},
    volume = {25},
    number = {3},
    pages = {1502 -- 1513},
    year = {1997},
    publisher = {Institute of Mathematical Statistics},
    doi = {10.1214/aop/1024404522},
    URL = {https://doi.org/10.1214/aop/1024404522}
}

@article{Ma2023CommunityDI,
    author={Ma, Zongming and Nandy, Sagnik},
    journal={IEEE Transactions on Information Theory}, 
    title={Community Detection With Contextual Multilayer Networks}, 
    year={2023},
    volume={69},
    number={5},
    pages={3203-3239},
    doi={10.1109/TIT.2023.3238352}
}

@article{McSherry2001SpectralPO,
  title={Spectral partitioning of random graphs},
  author={Frank McSherry},
  journal={Proceedings 2001 IEEE International Conference on Cluster Computing},
  year={2001},
  pages={529-537}
}

@article{Mossel2016ConsistencyTF,
    title={Consistency thresholds for the planted bisection model},
    author={Mossel, Elchanan and Neeman, Joe and Sly, Allan},
    volume = {21},
    journal = {Electronic Journal of Probability},
    number = {none},
    publisher = {Institute of Mathematical Statistics and Bernoulli Society},
    pages = {1 -- 24},
    year = {2016},
    doi = {10.1214/16-EJP4185},
    url = {https://doi.org/10.1214/16-EJP4185}
}

@article{Newman2002RandomGM,
  author={Newman, Mark EJ and Watts, Duncan J and Strogatz, Steven H},
  title = {Random graph models of social networks},
  journal = {Proceedings of the National Academy of Sciences},
  volume = {99},
  number = {suppl\_1},
  pages={2566-2572},
  year={2002},
  publisher={National Acad Sciences}
}

@inproceedings{Ng2002SpectralCA,
  title={On spectral clustering: Analysis and an algorithm},
  author={Ng, Andrew Y and Jordan, Michael I and Weiss, Yair},
  booktitle={Advances in neural information processing systems},
  pages={849--856},
  year={2002}
}

@article{Pal2021ComunityDI,
    author = {Pal, Soumik and Zhu, Yizhe},
    title = {Community detection in the sparse hypergraph stochastic block model},
    journal = {Random Structures \& Algorithms},
    volume = {59},
    number = {3},
    pages = {407-463},
    doi = {https://doi.org/10.1002/rsa.21006},
    url = {https://onlinelibrary.wiley.com/doi/abs/10.1002/rsa.21006},
    eprint = {https://onlinelibrary.wiley.com/doi/pdf/10.1002/rsa.21006},
    year = {2021}
}

@article{Rohe2011SpectralCA,
  title={Spectral clustering and the high-dimensional stochastic blockmodel},
  author={Rohe, Karl and Chatterjee, Sourav and Yu, Bin},
  journal={Annals of Statistics},
  year={2011},
  volume={39},
  pages={1878-1915}
}

@article{Snijders1997EstimationAP,
  title={Estimation and Prediction for Stochastic Blockmodels for Graphs with Latent Block Structure},
  author={Tom A. B. Snijders and Krzysztof Nowicki},
  journal={Journal of Classification},
  year={1997},
  volume={14},
  pages={75-100}
}

@inproceedings{Wang2023ProjectedTP,
  title = 	 {Projected Tensor Power Method for Hypergraph Community Recovery},
  author =       {Wang, Jinxin and Pun, Yuen-Man and Wang, Xiaolu and Wang, Peng and So, Anthony Man-Cho},
  booktitle = 	 {Proceedings of the 40th International Conference on Machine Learning},
  pages = 	 {36285--36307},
  year = 	 {2023},
  editor = 	 {Krause, Andreas and Brunskill, Emma and Cho, Kyunghyun and Engelhardt, Barbara and Sabato, Sivan and Scarlett, Jonathan},
  volume = 	 {202},
  series = 	 {Proceedings of Machine Learning Research},
  month = 	 {07},
  publisher =   {PMLR},
  url = 	 {https://proceedings.mlr.press/v202/wang23af.html}
}

@artile{Weyl1912DasAV,
  title={Das asymptotische Verteilungsgesetz der Eigenwerte linearer partieller Differentialgleichungen (mit einer Anwendung auf die Theorie der Hohlraumstrahlung)},
  author={Weyl, Hermann},
  journal={Mathematische Annalen},
  volume={71},
  pages = {441--479},
  year = {1912},
  doi = {10.1007/BF01456804},
  url = {https://doi.org/10.1007/BF01456804}
}

@article{Yun2016OptimalCR,
    title={Optimal Cluster Recovery in the Labeled Stochastic Block Model},
    author={Yun, Se-Young and Proutiere, Alexandre},
    journal={Advances in Neural Information Processing Systems},
    volume={29},
    year={2016},
    doi = {https://doi.org/10.48550/arXiv.1510.05956}
}

@article{Zachary1977InformationFM,
    ISSN = {00917710},
    URL = {http://www.jstor.org/stable/3629752},
    author = {Zachary, Wayne W.},
    journal = {Journal of Anthropological Research},
    number = {4},
    pages = {452--473},
    publisher = {University of New Mexico, University of Chicago Press},
    title = {An Information Flow Model for Conflict and Fission in Small Groups},
    urldate = {2023-05-17},
    volume = {33},
    year = {1977}
}

@article{Zhang2016MinimaxRO,
    author = {Anderson Y. Zhang and Harrison H. Zhou},
    title = {Minimax rates of community detection in stochastic block models},
    volume = {44},
    journal = {The Annals of Statistics},
    number = {5},
    publisher = {Institute of Mathematical Statistics},
    pages = {2252 -- 2280},
    year = {2016},
    doi = {10.1214/15-AOS1428},
    url = {https://doi.org/10.1214/15-AOS1428}
}

@article{Zhang2023ExactRI,
    author={Zhang, Qiaosheng and Tan, Vincent Y. F.},
    journal={IEEE Transactions on Information Theory}, 
    title={Exact Recovery in the General Hypergraph Stochastic Block Model}, 
    year={2023},
    volume={69},
    number={1},
    pages={453-471},
    doi={10.1109/TIT.2022.3205959}
  }

@article{Zhang2023FundamentalLO,
    author={Zhang, Anderson Ye},
    journal={IEEE Transactions on Information Theory}, 
    title={Fundamental Limits of Spectral Clustering in Stochastic Block Models}, 
    year={2024},
    volume={70},
    number={10},
    pages={7320-7348},
    doi={10.1109/TIT.2024.3425581}
}
